\documentclass[10pt,twoside,reqno]{amsart}
\usepackage{import}
\usepackage{anysize}
\usepackage{amsmath}
\usepackage{amsthm}
\usepackage{amsmath,amscd}
\usepackage[utf8]{inputenc}
\usepackage{amssymb}
\usepackage{stmaryrd}
\usepackage{wasysym}
\usepackage{mathrsfs}
\usepackage{hyperref}
\usepackage{cleveref}
\usepackage{enumitem}
\usepackage{dsfont}
\usepackage{soul}
\hypersetup{colorlinks=true,linkcolor=black,citecolor=black}
\usepackage{color}
\usepackage[all]{xy}
\usepackage{float}
\usepackage{bm}
\usepackage{mathtools}
\usepackage{thmtools}
\usepackage{setspace}
\usepackage{comment}
\usepackage{tikz}
\usepackage{tikz-cd}
\usepackage{mathdots}
\usepackage[myheadings]{fullpage}

\numberwithin{equation}{section}
\newcommand\mtop{.95in}
\newcommand\mbottom{.95in}
\newcommand\mleft{1in}
\newcommand\mright{1in}
\usepackage[top = \mtop, bottom = \mbottom, left = \mleft, right=\mright]{geometry}

\newtheorem{thm}{Theorem}[section]
\newtheorem{example}[thm]{Example}
\newtheorem{prop}[thm]{Proposition}

\newtheorem{lemma}[thm]{Lemma}

\theoremstyle{definition}
\newtheorem{defi}{Definition}

\newtheorem{rmk}{Remark}

\newcommand\reallywidehat[1]{%
\savestack{\tmpbox}{\stretchto{%
  \scaleto{%
    \scalerel*[\widthof{\ensuremath{#1}}]{\kern-.6pt\bigwedge\kern-.6pt}%
    {\rule[-\textheight/2]{1ex}{\textheight}}%WIDTH-LIMITED BIG WEDGE
  }{\textheight}% 
}{0.5ex}}%
\stackon[1pt]{#1}{\tmpbox}%
}
\DeclareSymbolFont{bbold}{U}{bbold}{m}{n}
\DeclareSymbolFontAlphabet{\mathbbold}{bbold}

\makeatletter
\def\@tocline#1#2#3#4#5#6#7{\relax
  \ifnum #1>\c@tocdepth % then omit
  \else
    \par \addpenalty\@secpenalty\addvspace{#2}%
    \begingroup \hyphenpenalty\@M
    \@ifempty{#4}{%
      \@tempdima\csname r@tocindent\number#1\endcsname\relax
    }{%
      \@tempdima#4\relax
    }%
    \parindent\z@ \leftskip#3\relax \advance\leftskip\@tempdima\relax
    \rightskip\@pnumwidth plus4em \parfillskip-\@pnumwidth
    #5\leavevmode\hskip-\@tempdima
      \ifcase #1
       \or\or \hskip 1em \or \hskip 2em \else \hskip 3em \fi%
      #6\nobreak\relax
    \hfill\hbox to\@pnumwidth{\@tocpagenum{#7}}\par% <---- \dotfill -> \hfill
    \nobreak
    \endgroup
  \fi}
\makeatother

\DeclarePairedDelimiter{\abs}{\lvert}{\rvert}
\newcommand{\R}{\mathbb{R}}
\newcommand{\Z}{\mathbb{Z}}
\newcommand{\Q}{\mathbb{Q}}
\newcommand{\N}{\mathbb{N}}
\newcommand{\C}{\mathbb{C}}
\newcommand{\F}{\mathbb{F}}

\renewcommand{\H}{\mathbb{H}}
\renewcommand{\O}{\mathcal{O}}
\newcommand{\E}{\mathbb{E}}

\newcommand{\mc}{\mathcal}

\newcommand{\bbone}{\mathbbold{1}}
\renewcommand{\l}{\lambda}
\newcommand{\eps}{\epsilon}

\newcommand{\dderiv}[2]{\frac{d #1}{d #2}}

\newcommand{\la}{\left\langle}
\newcommand{\ra}{\right\rangle}

\newcommand{\tth}{^{th}}

\newcommand{\bY}{\bar{Y}}
\newcommand{\bZ}{\bar{Z}}

\newcommand{\M}{\mu_{Haar}}

\newcommand{\tl}{\tilde{\lambda}}
\newcommand{\tm}{\tilde{\mu}} 
\newcommand{\tk}{\tilde{\kappa}} 
\newcommand{\tc}{\tilde{c}} 
\newcommand{\bl}{{\bar{\lambda}}}

\renewcommand{\hl}{{\hat{\lambda}}}
\renewcommand{\hm}{{\hat{\mu}}}
\newcommand{\hk}{{\hat{\kappa}}}
\renewcommand{\L}{\Lambda}

\newcommand{\tT}{\tilde{T}}
\newcommand{\G}{G}
\newcommand{\K}{K}
\newcommand{\ba}{\mathbf{a}}
\newcommand{\bd}{\mathbf{d}}
\newcommand{\bx}{\mathbf{x}}
\newcommand{\mbr}{M_D^{branch}}%probably want to go change these later
\newcommand{\mcj}{M_D^{Cauchy}}
\newcommand{\hx}{\hat{x}}
\newcommand{\tv}{\tilde{v}}
\newcommand{\bv}{\mathbf{v}}
\newcommand{\barv}{\bar{v}}
\newcommand{\tL}{\tilde{L}}
\newcommand{\bsomega}{\boldsymbol{\omega}}
\DeclareMathOperator{\Geom}{Geom}
\DeclareMathOperator{\len}{len}

\DeclareMathOperator{\Id}{Id}

\DeclareMathOperator{\Aut}{Aut}

\DeclareMathOperator{\GL}{GL}

\DeclareMathOperator{\Sp}{Sp}
\DeclareMathOperator{\U}{U}

\DeclareMathOperator{\coker}{coker}
\DeclareMathOperator{\Sig}{Sig}
\DeclareMathOperator{\Proj}{Proj}
\DeclareMathOperator{\GT}{GT}

\DeclareMathOperator{\SN}{SN}
\DeclareMathOperator{\diag}{diag}
\DeclareMathOperator{\Supp}{Supp}

\title[Limits and fluctuations of $p$-adic matrix products]{Limits and fluctuations of $p$-adic random matrix products}
\author{Roger Van Peski}

\date{\today}

\thanks{I thank my advisor Alexei Borodin for many helpful conversations throughout this project and feedback on several drafts; Vadim Gorin, for early encouragement, pointers to the matrix product literature, and other helpful discussions and comments; and Andrew Ahn, for many illuminating discussions on the complex case and its relation to this work, and detailed feedback on the exposition. Some preliminary material appeared in my unpublished undergraduate thesis \cite{vanpeski2018}, and it is my pleasure to thank Ju-Lee Kim for her guidance on that project. I also wish to thank Adam Block, Jason Fulman, Nathan Kaplan, Mario Kieburg, Sergei Korotkikh, and Kevin Lin for additional helpful conversations. This material is based on work supported by the National Science Foundation Graduate Research Fellowship under Grant No. \#$1745302$. 
}

\begin{document}

\maketitle

\begin{abstract} 
We show that singular numbers (also known as elementary divisors, invariant factors or Smith normal forms) of products and corners of random matrices over $\mathbb{Q}_p$ are governed by the Hall-Littlewood polynomials, in a structurally identical manner to the known relations between singular values of complex random matrices and Heckman-Opdam hypergeometric functions. This implies that the singular numbers of a product of corners of Haar-distributed elements of $\mathrm{GL}_N(\mathbb{Z}_p)$ form a discrete-time Markov chain distributed as a Hall-Littlewood process, with the number of matrices in the product playing the role of time. We give an exact sampling algorithm for the Hall-Littlewood processes which arise by relating them to an interacting particle system similar to PushTASEP. By analyzing the asymptotic behavior of this particle system, we show that the singular numbers of such products obey a law of large numbers and their fluctuations converge dynamically to independent Brownian motions. In the limit of large matrix size, we also show that the analogues of the Lyapunov exponents for matrix products have universal limits within this class of $\mathrm{GL}_N(\mathbb{Z}_p)$ corners.
\end{abstract}



\tableofcontents
\section{Introduction}

\subsection{Asymptotic results.}

For any nonsingular complex matrix $A \in M_{n \times m}(\C)$, by singular value decomposition there exist $U \in U(n), V \in U(m)$ with $UAV = \diag(e^{-r_1},e^{-r_2},\ldots,e^{-r_{\min(m,n)}})$ for some $\infty > r_1 \geq \cdots \geq r_{\min(m,n)}$. Studying the distributions of the singular values of various random matrices $A$, and their asymptotics, is a classical but still very active line of research.

For any nonsingular $p$-adic matrix\footnote{For background on the $p$-adic numbers and matrix groups over them, see \Cref{sec:3}.} $A \in M_{n \times m}(\Q_p)$, there similarly exist $U \in \GL_n(\Z_p), V \in \GL_m(\Z_p)$ such that $UAV = \diag(p^{\l_1},p^{\l_2},\ldots,p^{\l_{\min(m,n)}})$ for some integers $\infty > \l_1 \geq \ldots \geq \l_{\min(m,n)}$. We refer to the integers $\l_i$ as the \emph{singular numbers} of $A$ and write $\SN(A) = (\l_1,\ldots,\l_{\min(m,n)}) = \l$ in the above case. One can study the distribution of $\SN(A)$ for random $A \in M_{n \times m}(\Q_p)$ just as with singular values. 

Such distributions have received significant attention within number theory, starting with the 1987 work of Friedman-Washington \cite{friedman-washington}. They computed the $n \to \infty$ limiting distribution of $\SN(A)$ for $A \in M_n(\Z_p)$ with iid entries distributed according to the additive Haar measure on $\Z_p$, which may be viewed as the analogue of the Gaussian in this setting. They found that it matched the distribution of the $p$-torsion parts of class groups\footnote{Here we view $A$ as a map $\Z_p^n \to \Z_p^n$ and identify $\l = \SN(A)$ with the abelian $p$-group $\coker(A) \cong \bigoplus_i \Z/p^{\l_i}\Z$.} of quadratic imaginary number fields conjectured by Cohen and Lenstra \cite{cohen-lenstra} in 1983 on the basis of numerical data, and gave a heuristic explanation why this was the case. This launched many subsequent works on modeling groups appearing in number theory by suitable random $p$-adic matrices, see for instance Achter \cite{achter2006distribution}, Bhargava-Kane-Lenstra-Poonen-Rains \cite{bhargava2013modeling}, Ellenberg-Jain-Venkatesh \cite{ellenberg2011modeling}, Wood \cite{wood2015random,wood2018cohen} and particularly the survey \cite{woodexpos}. 

The asymptotic distributions obtained in the above works look quite different from their counterparts in the world of singular values. For instance, \cite{friedman-washington} show that for any nonnegative integers $\l_1 \geq \l_2 \geq \ldots \geq 0$ and random $A_n \in M_n(\Z_p)$ as above, one has
\begin{equation*}
    \lim_{n \to \infty} \Pr(\SN(A_n) = \l) = \frac{(p^{-1};p^{-1})_\infty}{\abs*{\Aut(\bigoplus_i \Z/p^{\l_i}\Z)}}.
\end{equation*}
In particular, the limiting $\l_i$ are integers, and with probability $1$ only a finite number of the $\l_i$ are nonzero. The motivating questions in number theory concern group-theoretic properties such as the probability of cyclicity (i.e. probability that $\SN(A) = (k,0,\ldots,0)$ for some $k$) or the distribution of ranks (the number of nonzero parts of $\l=\SN(A)$).

By contrast, the singular values of an $n \times n$ matrix with iid standard Gaussian entries, usually referred to as the \emph{Ginibre ensemble}, converge with rescaling to the celebrated Marchenko-Pastur law \cite{marchenko1967distribution} (a compactly supported probability distribution on $\R$). As far as we are aware, no continuous probability distributions on $\R$ appeared previously governing limits of singular numbers of random $p$-adic matrices. Indeed, such limits are in a sense orthogonal to the viewpoint of random abelian $p$-groups taken in most of the previous literature.

In this work we find Gaussian limits in the setting of products of a large number of $p$-adic matrices of finite size. Given random matrices $A_1, A_2, \ldots \in M_n(\Q_p)$, one may view the $\SN(A_1), \SN(A_2 A_1), \ldots$ as defining a discrete-time Markov chain on the set of weakly decreasing $n$-tuples of integers. Equivalently, for each $i$ the $i\tth$ largest singular number evolves as some random walk on $\Z$, with the $n$ such random walks sometimes colliding but never crossing. See \Cref{fig:random_walks} below.

In the case when $A_i$ are $n \times n$ corners of independent Haar-distributed elements of $\GL_N(\Z_p)$ with $N>n$, we show that the singular numbers of their products satisfy an explicit law of large numbers as $k \to \infty$, and furthermore the fluctuations converge to $n$ independent Brownian motions. In the limit as $N \to \infty$, the entries of such corners become independent and distributed according to the additive Haar measure on $\Z_p$, recovering the matrices studied in the previous literature, so we allow the case `$N=\infty$' below.

\begin{restatable}{thm}{llncltrmt}\label{thm:lln_and_func_clt_rmt_version}
Fix $n \geq 1$, and let $N_1,N_2,\ldots \in \Z \cup \{\infty\}$ with $N_j > n$ for all $j$. For each $j$, if $N_j < \infty$ let $A_j$ be the top left $n \times n$ corner of a Haar distributed element of $\GL_{N_j}(\Z_p)$, and if $N_j = \infty$ let $A_j$ have iid entries distributed by the additive Haar measure on $\Z_p$. For $k \in \N$ let
\begin{equation*}
(\l_1(k),\ldots,\l_n(k)) := \SN(A_k \cdots A_1).    
\end{equation*} 
Then we have a strong law of large numbers
\begin{equation*}
    \frac{\l_i(k)}{\sum_{j=1}^k \sum_{\ell=0}^{N_j-n-1} \frac{p^{-i-\ell}(1-p^{-1})}{(1-p^{-i-\ell-1})(1-p^{-i-\ell})}} \to 1 \text{       a.s.  as $k \to \infty$.}
\end{equation*}
Let 
\begin{equation*}
\bl_i(k) := \l_i(k) - \sum_{j=1}^k \sum_{\ell=0}^{N_j-n-1} \frac{p^{-i-\ell}(1-p^{-1})}{(1-p^{-i-\ell-1})(1-p^{-i-\ell})}    
\end{equation*}
and define the random function of $f_{\bl_i,k} \in C[0,1]$ as follows: set $f_{\bl_i,k}(0)=0$ and
\begin{equation*}
        (f_{\bl_i,k}(1/k),f_{\bl_i,k}(2/k),\ldots,f_{\bl_i,k}(1)) = \frac{1}{\sqrt{\sum_{j=1}^k \sum_{\ell=0}^{N_j-n-1} \frac{p^{-i-\ell}(1-p^{-1})(1-p^{-2i-2\ell-1})}{(1-p^{-i-\ell})^2(1-p^{-i-\ell-1})^2}}}(\bl_i(1),\ldots,\bl_i(k)),
\end{equation*}
then linearly interpolate from these values on each interval $[\ell/k,(\ell+1)/k]$. Then as $k \to \infty$, the $n$-tuple of random functions $(f_{\bl_1,k},\ldots,f_{\bl_n,k})$ converges in law in the sup norm topology on $C[0,1]$ to $n$ independent standard Brownian motions. 
\end{restatable}

\begin{figure}[hbtp]
  \centering
    \includegraphics[width=.9\textwidth]{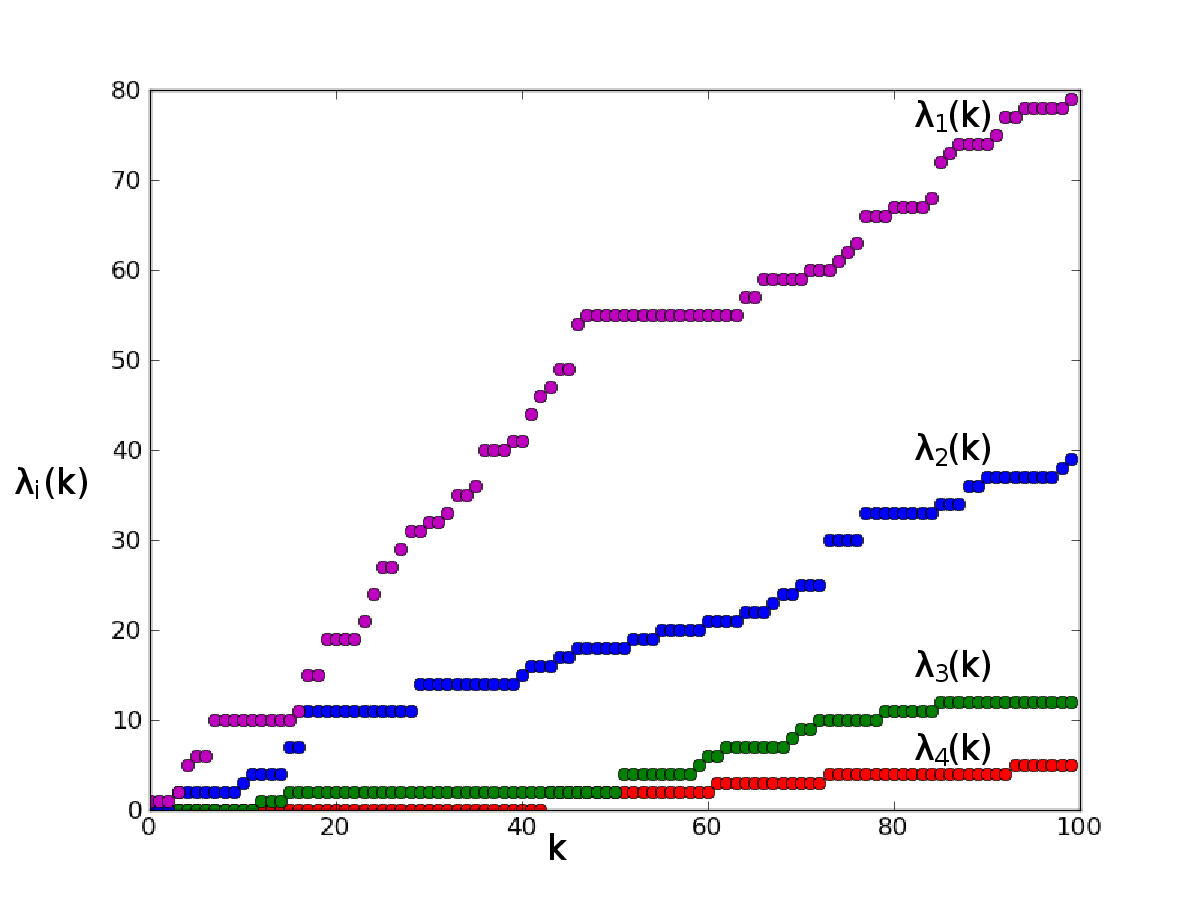}
\caption{A plot of $(\l_1(k),\l_2(k),\l_3(k),\l_4(k)) = \SN(A_k \cdots A_1)$ where $A_1,\ldots,A_{100} \in M_4(\Z_2)$ are random matrices with iid entries drawn from the additive Haar measure on $\Z_2$.}
\label{fig:random_walks}
\end{figure}

In particular, we have the central limit theorem that
\begin{equation*}
    \frac{\bl_i(k)}{\sqrt{\sum_{j=1}^k \sum_{\ell=0}^{N_j-n-1} \frac{p^{-i-\ell}(1-p^{-1})(1-p^{-2i-2\ell-1})}{(1-p^{-i-\ell})^2(1-p^{-i-\ell-1})^2}}} \to \mc N(0,1)
\end{equation*}
in law for each $i$. 
\begin{rmk}
If all $N_j$ are equal to some $N$, then the law of large numbers takes the more standard form
\begin{equation*}
    \frac{\l_i(k)}{k} \to \sum_{\ell=0}^{N-n-1} \frac{p^{-i-\ell}(1-p^{-1})}{(1-p^{-i-\ell-1})(1-p^{-i-\ell})}.
\end{equation*}
When $N=\infty$, evaluating the sum via the $q$-Gauss identity \cite[(3.5)]{koepf1998hypergeometric} yields an even more explicit limit:
\begin{equation*}
    \frac{\l_i(k)}{k} \to \frac{1}{p^i-1}.
\end{equation*}
\end{rmk}

The matrices $A_i$ above lie in $\GL_n(\Q_p)$ with probability $1$, so \Cref{thm:lln_and_func_clt_rmt_version} may be viewed as a statement about certain random walks on the group $\GL_n(\Q_p)$. Previous work by Brofferio-Schapira \cite{brofferio2011poisson} takes this perspective of random walks on groups and studies similar random walks from an ergodic theory perspective. They prove a law of large numbers for products of iid random matrices from a quite general class of probability distributions on $\GL_n(\Q_p)$ via a generalization of Oseledets' multiplicative ergodic theorem \cite{oseledets1968multiplicative} to matrices over $\Q_p$, due to Raghunathan \cite{raghunathan1979proof}. The family of probability distributions considered in \cite{brofferio2011poisson} is more general than that of \Cref{thm:exact_hl_results_in_p-adic_rmt}, but the latter covers cases when the matrices $A_i$ are not identically distributed and shows Gaussian fluctuations, which are not shown in \cite{brofferio2011poisson}. A different family of random walks on $\GL_n(\Q_p)$ were studied by Chhaibi \cite{chhaibi2017non}; the perspective in this work is more similar to ours in that it is heavily based on special functions on $p$-adic groups, though the presentation and problems considered are quite different. Finally, there is a body of work on random walks on Bruhat-Tits buildings which translates to results on random walks on $p$-adic groups, see for instance Cartwright-Woess \cite[Section 8]{cartwright2004isotropic}, Schapira \cite{schapira2009random}, and especially the survey of Parkinson \cite{parkinson2017buildings}. These works contain central limit theorems, but for quite different quantities and settings than ours.

In the setting of real and complex matrices, the study of asymptotics of singular values of products $A_k A_{k-1} \cdots A_1$ of random matrices as $k \to \infty$ dates back at least as far as the 1960 work of Furstenberg and Kesten \cite{furstenberg1960products}, who showed Gaussian fluctuations for the logarithm of the largest singular value under some assumptions on the $A_i$. In the case where the $A_i$ are iid square with complex Gaussian entries, Gaussian fluctuations for logarithms of all singular values (not just the largest as in \cite{furstenberg1960products}) were obtained in the physics literature by Akemann-Burda-Kieburg \cite{akemann2014universal} and in the mathematics literature by Liu-Wang-Wang \cite[Thm. 1.1]{liu2018lyapunov}. The case of products of $n \times n$ corners of unitary matrices, often referred to as the \emph{truncated unitary ensemble}, has also been considered; the term Jacobi ensemble is also sometimes used for such corners due to the relation with the classical Jacobi orthogonal polynomial ensemble. Dynamical convergence of the logarithm of the largest singular value of $A_k \cdots A_1$ to Brownian motion for products of such unitary corners is suggested, though not directly implied, by the work of Ahn (one should send $\hat{T} \to \infty$ in \cite[Thm. 1.7]{ahn2019fluctuations}). Together these results strongly suggest that as the number of products goes to infinity, the fluctuations of the $n$ log-singular values should converge to $n$ independent Brownian motions, as holds in the $p$-adic case by \Cref{thm:lln_and_func_clt_rmt_version}, but this full result has to our knowledge not appeared. 

There is now a large body of both mathematics and physics literature on asymptotics of matrix products in various regimes, often from the perspective of ergodic theory and often motivated by connections to chaotic dynamical systems and disordered systems in statistical physics, neural networks, and other areas. See for example Ahn \cite{ahn2019fluctuations,ahn2020product}, Akemann-Burda-Kieburg \cite{akemann2014universal,akemann2019integrable,akemann2020universality}, Akemann-Ipsen \cite{akemann2015recent}, Akemann-Ipsen-Kieburg \cite{akemann2013products}, Akemann-Kieburg-Wei \cite{akemann2013singular}, Crisanti-Paladin-Vulpiani \cite{crisanti2012products}, Forrester \cite{forrester2015asymptotics} and Forrester-Liu \cite{forrester2016singular}, Gol'dsheid-Margulis \cite{gol1989lyapunov}, Gorin-Sun \cite{gorin2018gaussian}, Kieburg-K{\"o}sters \cite{kieburg2019products}, and Liu-Wang-Wang \cite{liu2018lyapunov}. Such considerations motivate the study of the \emph{Lyapunov exponents}, so named because of the connection with dynamical systems: given random complex matrices $A_1,A_2,\ldots$, the $i\tth$ Lyapunov exponent is defined as
\begin{equation*}
         \lim_{k \to \infty} \frac{1}{k} \log(i\tth\text{ largest singular value of }A_k \cdots A_1).
\end{equation*}
In the limit as the sizes of the $A_i$ grows, the largest Lyapunov exponents converge (with appropriate scaling) to an evenly spaced sequence $0,-1,-2,\ldots$ in several known cases. Note that this statement has no content if one may scale each Lyapunov exponent individually, but it is quite surprising that applying the same additive shift and multiplicative scaling to all of the Lyapunov exponents together produces this evenly spaced sequence. For complex Ginibre matrices this convergence statement is an easy corollary of results of Liu-Wang-Wang\footnote{To obtain $0,-1,-2,\ldots$, take $N \to \infty$ in (1.7) of \cite{liu2018lyapunov} with scaling and use that the digamma function $\psi(z)$ is asymptotic to $\log(z)$.} \cite{liu2018lyapunov}. For products of corners of Haar-distributed unitary matrices and Ginibre matrices, the result will appear soon \cite{lyapunovnote}. Such limits also appear in multiplicative Dyson Brownian motion on complex matrices, see \cite{jones2006weyl}, to which these matrix product processes may be viewed as a kind of discrete random walk approximation. These examples seem to support the notion that such evenly spaced Lyapunov exponents are universal, though we are not aware of a precise conjecture in the literature regarding the scope of this class.

In the $p$-adic case, we are able to prove that at least within the class of truncated $\GL_N(\Z_p)$ matrices--and iid additive Haar matrices, by the $N \to \infty$ limit)--the appropriate analogues of Lyapunov exponents have universal limits. Consider random $A_1,A_2,\ldots \in M_n(\Q_p)$. For appropriate $U,V \in \GL_n(\Z_p)$ such that $U(A_k \cdots A_1)V = \diag(p^{\l_1(k)},\ldots,p^{\l_n(k)})$ with $\l_1(k) \geq \ldots \geq \l_n(k)$, we have that the $i\tth$ \emph{smallest} part $\l_{n-i+1}(k)$ of $\SN(A_k \cdots A_1)$ is the analogue of $-\log(\text{$i\tth$ \emph{largest} singular value})$ in the complex setting, because $p^{D}$ is small in the $p$-adic norm for large $D$. Hence 
\begin{equation*}
    \lim_{k \to \infty} \frac{\l_{n-i+1}(k)}{k}
\end{equation*} 
should be regarded as the appropriate analogue of the $i\tth$ Lyapunov exponent in the $p$-adic setting. Our next result shows that within the class of products of arbitrary Haar corners, these analogues of Lyapunov exponents converge to values $1,p,p^2,\ldots$ in geometric progression, much like the arithmetic progression $0,-1,-2,\ldots$ in the complex case mentioned previously.

\begin{restatable}[{Large-$n$ universality of Lyapunov exponents}]{thm}{lyapunov} \label{thm:lyapunov}
For each $n \in \N$, let $N_1^{(n)}, N_2^{(n)},\ldots \in \Z_{\geq 0} \cup \{\infty\}$ be such that $N_j^{(n)} > n$ and the limiting frequencies
\begin{equation*}
       \rho_n(N) := \lim_{k \to \infty} \frac{|\{1 \leq j \leq k: N_j^{(n)} = N\}|}{k}
\end{equation*}
exist for all $N > n$. Let $A_j^{(n)}$ be $n \times n$ corners of independent Haar distributed matrices in $\GL_{N_j^{(n)}}(\Z_p)$ (with the case $N_j^{(n)}=\infty$ treated as in \Cref{thm:lln_and_func_clt_rmt_version}). Then for each $n$, the Lyapunov exponents
\begin{equation*}
    L_i^{(n)} := \lim_{k \to \infty} \frac{\l_{n-i+1}(k)}{k}
\end{equation*}
exist almost surely, where $\l_{n-i+1}(k)$ is as in \Cref{thm:lln_and_func_clt_rmt_version}. Furthermore, the Lyapunov exponents have limits
\begin{equation*}
    \lim_{n \to \infty} \frac{L_i^{(n)}}{p^{-n}(1-c(n))} = p^{i-1}
\end{equation*}
for every $i$, where $c(n) := \sum_{N > n} \rho_n(N) p^{-(N-n)}$.
\end{restatable}

We note that \Cref{thm:lyapunov} does not require any relation between the $N_j^{(n)}$ for different $n$; one can for example alternate $N_j^{(n)} = n+1$ for $n$ even and $N_j^{(n)} = \infty$ for $n$ odd, and the result still holds. We believe this behavior to be universal for the large-$n$ limits of products from a broad class of distributions on $M_n(\Z_p)$, though extending beyond the setting of \Cref{thm:lyapunov} would require new methods; in this setting the proof is not difficult using the explicit law of large numbers in \Cref{thm:lln_and_func_clt_rmt_version}. The latter theorem requires substantial work, however, and brings us into the realm of nonasymptotic random matrix theory.

%It would be interesting to see how far this universality can be pushed, as the class of random matrices coming from corners of Haar elements of $\GL_N(\Z_p)$ is still quite rigid. Some universality statements for $p$-adic matrices have been proven in the regimes considered by number theorists, see Wood \cite{wood2017distribution,wood2015random}, but as mentioned above such regimes are quite different from ours.

\subsection{Nonasymptotic results and Hall-Littlewood polynomials.} Given $A \in M_{n \times m}(\Q_p)$ random with a distribution which is invariant under left- and right-multiplication by $\GL_n(\Z_p)$ and $\GL_m(\Z_p)$ respectively, one may ask the following.

\begin{enumerate}
    \item[(Q1)] For `natural' choices of the distribution of $A$, what is the distribution of $\SN(A)$? %\label{item:p_ensemble}
    
    \item[(Q2)] Let $A_{col}$ be the matrix given by removing the last column from $A$. What is the conditional distribution of $\SN(A_{col})$ given $\SN(A)$?
    
    \item[(Q3)] Let $B \in M_{m \times k}(\Q_p)$ be another random matrix. Given the distributions of $\SN(A),\SN(B)$, what is the distribution of $\SN(AB)$? 
\end{enumerate}

We give complete answers to (Q2), (Q3) and a family of cases of (Q1) in \Cref{thm:exact_hl_results_in_p-adic_rmt} below using relations between $p$-adic matrices and the classical Hall-Littlewood symmetric polynomials $P_\l(x_1,\ldots,x_n;t)$. The latter are a distinguished basis for the ring $\Lambda_n := \C[x_1,\ldots,x_n]^{S_n}$ of symmetric polynomials in $n$ variables $x_1,\ldots,x_n$, indexed by the set $\Sig_n^+ := \{(\l_1,\ldots,\l_n) \in \Z^n: \l_1 \geq \cdots \geq \l_n \geq 0\}$ of nonnegative \emph{integer signatures}, and feature an additional parameter $t$ which we take to be real. For a full definition, see \Cref{subsec:hl}, but using only the property that they form a basis and some positivity properties below, one can use them to define probability measures, Markov dynamics, and randomized convolution operations on signatures\footnote{These operations also make sense for signatures with some negative parts, but we refer to \Cref{sec:2} for conventions on these.}.
\begin{enumerate}
    \item (Probability measures) For real $a_1,\ldots,a_n \geq 0$, and $t \in [0,1)$ one has $P_\l(a_1,\ldots,a_n;t) \geq 0$. Hence for any sets $\{a_i\},\{b_i\}$ of nonnegative reals with all $a_ib_j < 1$one may define the \emph{Hall-Littlewood measure} on $\Sig_n^+$ via 
    \begin{equation}\label{eq:mac_measure_first}
        \Pr(\l) = \frac{1}{\Pi_{(0,t)}(a_1,\ldots,a_n;b_1,\ldots,b_n)}P_\l(a_1,\ldots,a_n;t) Q_\l(b_1,\ldots,b_n;t)
    \end{equation}
    where $\Pi_{(0,t)}(a_1,\ldots,a_n;b_1,\ldots,b_n)$ is a normalizing constant and $Q_\l$ is a certain constant multiple of $P_\l$, see \Cref{sec:2}. 
    
    \item (Markov dynamics) Because the $P_\l$ form a basis for the vector space of symmetric polynomials in $n$ variables, 
    \begin{equation*}
        P_\l(x_1,\ldots,x_n;t) = \sum_\mu P_{\l/\mu}(x_{k+1},\ldots,x_n;t) P_\mu(x_1,\ldots,x_{k};t)
    \end{equation*}
    for some symmetric polynomials $P_{\l/\mu} \in \L_{n-k}$, called \emph{skew Hall-Littlewood polynomials}. Substituting positive real numbers $a_i$ for the variables naturally yields Markov dynamics $\Sig_n^+ \to \Sig_{k}^+$ given by
    \begin{equation*}
        \Pr(\l \to \mu) = \frac{P_{\l/\mu}(a_{k+1},\ldots,a_n;t)P_\mu(a_1,\ldots,a_k;t)}{P_\l(a_1,\ldots,a_n;t)}.
    \end{equation*}

    \item (Product convolution) Again using that the $P_\l$ form a basis,
    \begin{equation*}
    P_\l(x_1,\ldots,x_n;t) \cdot P_\mu(x_1,\ldots,x_n;t) = \sum_{\nu} c_{\l,\mu}^\nu(0,t) P_\nu(x_1,\ldots,x_n;t)
    \end{equation*}
    for some structure coefficients $c_{\l,\mu}^\nu(0,t)$--these are often called \emph{Littlewood-Richardson coefficients}, particularly in the case $t=0$ corresponding to the classical Schur polynomials. One may then, given two fixed signatures $\l,\mu$, define their `convolution' $\l \boxtimes_{\ba} \mu$ (a random signature) by
    \begin{equation*}
        \Pr(\l \boxtimes_{\ba} \mu = \nu) = \frac{P_\nu(a_1,\ldots,a_n;t)}{P_\l(a_1,\ldots,a_n;t)P_\mu(a_1,\ldots,a_n;t)} c_{\l,\mu}^\nu(0,t)
    \end{equation*}
    for each $\nu \in \Sig_n$. Convolutions of signatures which are themselves random may be obtained from this by mixtures.
    
\end{enumerate}

We may now state the main structural result, which shows that the matrix operations of products and corners mirror the above operations on the level of symmetric functions.

\begin{restatable}{thm}{exactresults}\label{thm:exact_hl_results_in_p-adic_rmt}
Fix a prime $p$ and let $t=1/p$.
\begin{enumerate}[label={\arabic*.}]
    \item \label{item:thm_jacobi} \emph{(Truncated Haar ensemble)} Let $1 \leq n \leq m \leq N$ be integers, and $A$ be the top-left $n \times m$ submatrix of a Haar-distributed element of $\GL_N(\Z_p)$. Then $\SN(A)$ is a random nonnegative signature with distribution given by the Hall-Littlewood measure
    \begin{equation}\label{eq:p-adic_jacobi}
        \Pr(\SN(A) = \l) = \frac{P_\l(1,t,\ldots,t^{n-1};t)Q_\l(t^{m-n+1},\ldots,t^{N-n};t)}{\Pi_{(0,t)}(1,t,\ldots,t^{n-1};t^{m-n+1},\ldots,t^{N-n})}.
    \end{equation}
    
    \item \label{item:thm_corners}\emph{(Corners process)} Let $n,k,N$ be integers with $1 \leq n  \leq N$ and $1 \leq k \leq N-n$, $\l \in \Sig_n$, and $A \in M_{n \times N}(\Q_p)$ be random with $\SN(A)=\l$ and distribution invariant under $\GL_n(\Z_p) \times \GL_N(\Z_p)$ acting on the right and left. Let $A_{col} \in M_{n \times (N-k)}(\Q_p)$ be the first $N-k$ columns of $A$. Then $\SN(A_{col})$ is a random element of $\Sig_n$ with distribution given by
    \begin{equation}\label{eq:p-adic_cauchy_corners}
        \Pr(\SN(A_{col}) = \nu) = \frac{Q_{\nu/\l}(1,\ldots,t^{-(k-1)};t)P_\nu(t^{N-n},\ldots,t^{N-1};t)}{P_\l(t^{N-n},\ldots,t^{N-1};t) \Pi_{(0,t)}(1,\ldots,t^{-(k-1)};t^{N-n},\ldots,t^{N-1})}.
    \end{equation}
    
    Now let $1 \leq d \leq n$ and $A_{row} \in M_{(n-d) \times N}$ be the first $n-d$ rows of $A$. Then $\SN(A_{row})$ is a random element of $\Sig_{n-d}$ with distribution
    \begin{equation}\label{eq:p-adic_corners}
        \Pr( \SN(A_{row}) = \mu) = P_{\l/\mu}(1,\ldots,t^{d-1};t) \frac{P_\mu(t^d,\ldots,t^{n-1};t)}{P_\l(1,\ldots,t^{n-1};t)}.
    \end{equation} 
    \item \label{item:thm_product} \emph{(Product process)} Let $A,B$ be random elements of $M_n(\Q_p)$ with fixed singular numbers $\SN(A) = \l,\SN(B)=\mu$, invariant under left- and right-multiplication by $\GL_n(\Z_p)$. Then for any $\nu \in \Sig_n$, $\SN(AB)$ has distribution $\l \boxtimes_{(1,\ldots,t^{n-1})} \mu$, i.e.
    \begin{equation}\label{eq:p-adic_product_conv}
         \Pr(\SN(AB) = \nu)  = c^{\nu}_{\l, \mu}(0,t) \frac{P_\nu(1,\ldots,t^{n-1};t)}{ P_\l(1,\ldots,t^{n-1};t)P_\mu(1,\ldots,t^{n-1};t)}.
    \end{equation} 
\end{enumerate}
\end{restatable}

In the limit $N \to \infty$, \Cref{thm:exact_hl_results_in_p-adic_rmt} Part 1 recovers the distribution of singular numbers of matrices with iid additive Haar entries, the analogue of the Ginibre ensemble.

\begin{restatable}{cor}{pwishart}\label{cor:reprove_wood}
Fix a prime $p$ and let $t=1/p$. Let $1 \leq n \leq m$, and $A \in M_{n \times m}(\Z_p)$ be random with iid entries distributed according to the additive Haar measure on $\Z_p$. Then for any $\l \in \Sig_n^+$,
\begin{equation*}
    \Pr(\SN(A) = \l) = \frac{P_\l(1,\ldots,t^{n-1};t)Q_\l(t^{m-n+1},t^{m-n+2},\ldots;t)}{\Pi_{(0,t)}(1,\ldots,t^{n-1};t^{m-n+1},t^{m-n+2},\ldots)}
\end{equation*}
\end{restatable}

\begin{rmk}\label{rmk:haar_gaussian}
As touched on above, in a probabilistic context it is often helpful to view the additive Haar measure on $\Z_p$ as an analogue of the Gaussian on $\R$ or $\C$. One shared feature is that additive convolution preserves both classes of measures: if $X,Y$ are distributed according to the additive Haar measure on $\Z_p$, $X+Y$ is as well. In fact, random vectors $v \in \Z_p^n$ with iid Haar-distributed entries are invariant under $\GL_n(\Z_p)$ just as Gaussian vectors are invariant under $U(n)$, and both are characterized up to scaling by this invariance together with independence of entries. Another shared feature is that both the Gaussian and the additive Haar measure are their own Fourier transform. See Tao \cite{taopgaussian} and Evans \cite{evans2001local} for more discussion.
\end{rmk}

Explicit formulas for the probabilities in \Cref{cor:reprove_wood} and Part 1 of \Cref{thm:exact_hl_results_in_p-adic_rmt} may be obtained using the explicit formulas for Hall-Littlewood polynomials, see \Cref{prop:hl_principal_formulas}. As mentioned before, the $m=n$ case appeared in \cite{friedman-washington}, and also in the work of Evans \cite{evans2002elementary}. The rectangular case was considered in work of Wood \cite[Thm. 1.3]{wood2015random}, which studies the $n \to \infty$ asymptotics of $n \times (n+u)$ matrices for fixed $u$; this work shows that the limit is universal for many choices of the distribution of matrix entries, but does not consider the exact result of \Cref{cor:reprove_wood} for finite $n$. 

Subcases of the Hall-Littlewood measures we consider also appear in the work of Fulman \cite{fulman_main,fulman-CL} on Jordan blocks of uniformly random elements of $\GL_n(\F_q)$, though the language of Hall-Littlewood measures was not used. A different family of Hall-Littlewood measures\footnote{These have one Plancherel specialization, while ours have both specializations geometric progressions in $t$.} appears in work of Borodin \cite{borodin1999lln} and Bufetov-Petrov \cite{bufetov2015lln} on the related problem of Jordan blocks of random upper-triangular matrices over $\F_q$.

\Cref{thm:exact_hl_results_in_p-adic_rmt} yields that the matrix product Markov chains on $\Sig_n$ in \Cref{thm:lln_and_func_clt_rmt_version} are Hall-Littlewood processes, see \Cref{sec:2} for definitions. The proof of \Cref{thm:lln_and_func_clt_rmt_version} contains one other key ingredient in addition to \Cref{thm:exact_hl_results_in_p-adic_rmt}, which is a particle system interpretation and exact sampling algorithm for these Hall-Littlewood processes. In \Cref{sec:4} we show that, viewing these elements of $\Sig_n$ as configurations of $n$ particles on $\Z$, the evolution of these particles under our Markov chain is described by a simple set of rules similar to \emph{pushing totally asymmetric simple exclusion process (PushTASEP)}. This description in particular yields an exact sampling algorithm for the relevant Hall-Littlewood processes, see \Cref{prop:sampling_alg}. This algorithm is equivalent to one given by Fulman \cite[Thm. 11]{fulman1999probabilistic}, though the latter makes no reference to the particle system interpretation. Our explicit description of this particle system for singular numbers may also be viewed as a kind of discrete, $p$-adic analogue of multiplicative Dyson Brownian motion on the space of complex matrices--we refer to Jones-O'Connell \cite{jones2006weyl} for an exposition of the latter. In the regime considered, these particles spread apart on $\Z$ and interactions between them do not contribute in the limit. After justifying that this is the case, the law of large numbers is an explicit computation and the convergence of fluctuations to independent Brownian motions follows from Donsker's theorem applied separately to each particle, yielding \Cref{thm:lln_and_func_clt_rmt_version}.

\begin{rmk}
Bufetov-Petrov \cite{bufetov2015lln} use a similar strategy to prove a law of large numbers for Jordan blocks of random upper-triangular matrices over a finite field. In this setting, analogous results to \Cref{thm:exact_hl_results_in_p-adic_rmt} relate the sizes of these Jordan blocks to certain Hall-Littlewood processes. \cite{bufetov2015lln} then gives a sampling algorithm for these Hall-Littlewood processes via an RSK algorithm, and analyzes it in a similar manner to obtain the law of large numbers. It is worth noting that Bufetov-Matveev \cite{bufetov2018hall} give a more general Hall-Littlewood RSK algorithm which could be used to sample from our Hall-Littlewood process, but this algorithm would be more detailed and hence more cumbersome to analyze.
\end{rmk}

\begin{rmk}\label{rmk:any_local_field}
Though we have chosen to state them in the case of $\Q_p$ because it is most commonly considered in the literature, all of the results and proofs in this paper are actually valid for matrices over any non-Archimedean local field $K$ with finite residue field, i.e. any algebraic extensions of $\Q_p$ or $\F_q((t))$. Any such field has a ring of integers $R$ which plays the role of $\Z_p$, and a unique maximal ideal $(\omega) \subset R$ generated by a uniformizer $\omega$ which plays the role of $p$. The residue field $R/(\omega)$ is a finite field $\F_q$ for some $q$. Replacing $\Q_p$ by $K$, $\Z_p$ by $R$, $p$ by $\omega$ (in the context of matrix entries), and setting $t=1/q$ in Hall-Littlewood specializations, our results translate mutatis mutandis. This is essentially a consequence of the fact that Part 3 of \Cref{thm:exact_hl_results_in_p-adic_rmt} holds in this generality, see \cite[Ch. V]{mac} and the discussion in \Cref{sec:3}. 
\end{rmk}

\subsection{Macdonald polynomials and connections to the complex case.}

Let us say a few words about the proof of \Cref{thm:exact_hl_results_in_p-adic_rmt}. Part 3 is essentially a probabilistic reframing of results on the Hecke ring of the pair $(\GL_n(\Q_p), \GL_n(\Z_p))$ in \cite[Ch. V]{mac}. To prove Part 1 and Part 2, we use limiting cases of Part 3 corresponding to projection matrices. For example, if $U \in \GL_N(\Z_p)$ is random with Haar distribution, then the corners described in Part 1 are given by the convolution $P_n U P_m$ of projection matrices of rank $n$ and $m$, which may be treated by a limiting case of the product operation in Part 3. This link also explains the appearance of similar geometric progressions in $t$ in the formulas in Parts 1, 2, 3: those in Parts 1, 2 come from those in Part 3 via this degeneration. We remark that in the complex case, the relation between products of randomly-rotated projection matrices and truncated unitary/Jacobi ensembles was observed and exploited by Collins \cite{collins2005product}.

To implement this strategy in the $p$-adic setting, in view of \eqref{eq:p-adic_product_conv} it is necessary to establish some combinatorial results on asymptotics of the structure coefficients $c_{\l,\mu}^\nu(0,t)$. We prove some quite general results in this direction in \Cref{sec:3}, which are valid for the more general class of \emph{Macdonald polynomials} $P_\l(x_1,\ldots,x_n;q,t)$, another family of symmetric polynomials indexed by signatures. These polynomials have two parameters $q,t$ and specialize to Hall-Littlewood polynomials when $q=0$, but the measures, Markov dynamics and convolution operations on signatures defined for Hall-Littlewood polynomials work exactly the same way. The aforementioned results on their structure coefficients, given later in \Cref{prop:measure_convergence_branching_corners} and \Cref{prop:measure_convergence_jacobi_and_cauchy}, are consequences of the following asymptotic factorization property of Macdonald polynomials, which we believe may be of independent interest for other asymptotic problems.

\begin{restatable}{thm}{explicitlow}\label{thm:explicit_low_deg_terms}
Let $q,t \in (-1,1)$. Fix positive integers $k, N$, let $r_1,\ldots,r_k$ be positive integers such that $\sum_i r_i = N$, and set $s_i = \sum_{j=1}^i r_j$ with the convention $s_0=0$. Let $L_1 > \cdots > L_k$ be integers and $\l^{(i)} \in \Sig_{r_i}$ be any signatures, and define the signature $\l(D) = (L_1D + \l^{(1)}_1,\ldots,L_1D+\l^{(1)}_{r_1},\ldots, L_kD + \l^{(k)}_1,\ldots,L_kD + \l^{(k)}_{r_k}) \in \Sig_N$ for each $D \in \N$ large enough so that this is a valid signature. Then

\begin{multline*}
    \frac{P_{\l(D)}(x_1,\ldots,x_N;q,t)}{\prod_{i=1}^{k} (x_{s_{i-1}+1} \cdots x_{s_{i}})^{L_i D}} \to 
     \prod_{i=1}^{k} 
    P_{\l^{(i)}}(x_{s_{i-1}+1},\ldots,x_{s_i};q,t)  \prod_{i=1}^{k-1} \Pi_{(q,t)}(x_{s_{i-1}+1}^{-1},\ldots,x_{s_i}^{-1}; x_{s_i+1},\ldots,x_{N}) 
\end{multline*}
as $D \to \infty$ in the sense that the coefficient of each Laurent monomial $x_1^{d_1} \cdots x_N^{d_N}$ on the LHS converges to the corresponding coefficient on the RHS. 
\end{restatable}

We chose to work at the level of Macdonald polynomials partially to highlight the similarities between our results and existing results for complex random matrices. One may define measures, Markov kernels and randomized convolution operations on integer signatures by the formulas \eqref{eq:p-adic_jacobi}, \eqref{eq:p-adic_corners}, \eqref{eq:p-adic_cauchy_corners} and \eqref{eq:p-adic_product_conv} in \Cref{thm:exact_hl_results_in_p-adic_rmt} but with Macdonald polynomials $P_\l(x_1,\ldots,x_n;q,t)$ substituted in for Hall-Littlewood polynomials $P_\l(x_1,\ldots,x_n;0,t)$, and all probabilities will still be nonnegative provided $q \in [0,1)$. Sending $q \to 0$ recovers the probabilities in \Cref{thm:exact_hl_results_in_p-adic_rmt}. There is another limit where $t = q^{\beta/2}$ and $q,t \to 1$ while the signatures are also scaled at some rate dependent on $q$, which yields probability measures, Markov kernels and convolutions on \emph{real} signatures. In these limits, the formulas in \Cref{thm:exact_hl_results_in_p-adic_rmt} (with Macdonald polynomials in place of Hall-Littlewood, but no other changes) degenerate when $\beta = 1,2,4$ to formulas for singular values under the corresponding corners and product matrix operations on real, complex and quaternion random matrices. We discuss this limit in more detail in \Cref{sec:appendixA}. 

We note that in a quite different direction, nontrivial $p$-adic analogues of results in complex random matrix theory have been obtained by Neretin \cite{neretin2013hua} and Assiotis \cite{assiotis2020infinite}. It would be interesting to see to what extent their results and the analogous results in complex random matrix theory follow the above pattern of separate degenerations of the same Macdonald-level objects.

\subsection{Plan of paper.} 

In \Cref{sec:2} we give the necessary background on symmetric functions, Macdonald and Hall-Littlewood polynomials, and probabilistic constructions from them. In \Cref{sec:3} we give background on $p$-adic random matrices and prove \Cref{thm:exact_hl_results_in_p-adic_rmt} conditional on some results on Macdonald polynomials, and in \Cref{sec:3.2} we prove \Cref{thm:explicit_low_deg_terms} and use it to deduce these. In \Cref{sec:4} we relate the Hall-Littlewood processes which appear to the particle system mentioned above, and use this to prove the asymptotic statements \Cref{thm:lln_and_func_clt_rmt_version} and \Cref{thm:lyapunov}. In \Cref{sec:appendixA} we give a more in-depth explanation of the relations to the real/complex/quaternion cases alluded to throughout and a sketch of how the strategy used to obtain \Cref{thm:exact_hl_results_in_p-adic_rmt} could be adapted to these cases, together with an alternate proof of one of our results (\Cref{prop:measure_convergence_jacobi_and_cauchy}) which would be helpful for this approach.

\section{Symmetric functions preliminaries}\label{sec:2}

In this section we give a quick review of general symmetric functions, Macdonald polynomials, and measures and Markov dynamics on integer signatures arising from these. For a more detailed introduction to symmetric functions see \cite{mac}, and for Macdonald measures see \cite{borodin2014macdonald}. 

\subsection{Macdonald polynomials.}

We have chosen to express everything in terms of integer signatures, which have finite length and may have negative parts, rather than the more common integer partitions. This is because we consider matrices of fixed size, so there is a natural length already fixed for the tuple of singular numbers, and because it is desirable to consider matrices over $\Q_p$ rather than just $\Z_p$ and this requires possibly-negative singular numbers. We begin with a block of notation for signatures.

\begin{defi}
$\Sig_n$ denotes the set of integer signatures of length $n$, which are weakly decreasing $n$-tuples of integers. Given $\l = (\l_1,\ldots,\l_n) \in \Sig_n$, we refer to the integers $\l_i$ as the \emph{parts} of $\l$. When the length of $\l$ is not clear from context we will denote it by $\len(\l)$. $\Sig_n^+$ is the set of signatures with all parts nonnegative. We set $|\l| := \sum_{i=1}^n \l_i$ and $m_k(\l) = |\{i: \l_i = k\}|$. For $\l \in \Sig_n$ and $\mu \in \Sig_{n-1}$, write $\mu \prec_P \l$ if $\l_i \geq \mu_i$ and $\mu_i \geq \l_{i+1}$ for $1 \leq i \leq n-1$. For $\nu \in \Sig_n$ write $\nu \prec_Q \l$ if $\l_i \geq \nu_i$ for $1 \leq i \leq n$ and $\nu_i \geq \l_{i+1}$ for $1 \leq i \leq n-1$. We write $c[k]$ for the signature $(c,\ldots,c)$ of length $k$, and often abuse notation by writing $(\l,\mu)$ to refer to the tuple $(\l_1,\ldots,\l_n,\mu_1,\ldots,\mu_m)$ when $\l \in \Sig_n, \mu \in \Sig_m$. We additionally write $-\l := (-\l_n,\ldots,-\l_1) \in \Sig_n$ for any $\l \in \Sig_n$. Finally, we denote the empty signature by $()$.
\end{defi}

Recall that we denote by $\L_n$ the ring $\R[x_1,\ldots,x_n]^{S_n}$ of symmetric polynomials in $n$ variables $x_1,\ldots,x_n$. It is a very classical fact that the power sum symmetric polynomials $p_k = \sum_{i=1}^n x_i^k, k =1,\ldots,n$, are algebraically independent and algebraically generate $\L_n$. An immediate consequence is that $\L_n$ has a natural basis given by the polynomials
\begin{equation*}
    p_\l := \prod_{i \geq 1} p_{\l_i}
\end{equation*}
for $n \geq \l_1 \geq \l_2 \geq \ldots$ a weakly decreasing sequence of nonnegative integers which is eventually $0$ (i.e. an integer partition). Hence given generic real parameters $q,t$, one may define an inner product on $\L_n$ by setting
\begin{equation*}
    \la p_\l, p_\mu \ra_{(q,t)} = \delta_{\l \mu} \prod_{i \geq 1: \l_i > 0} \frac{1-q^{\l_i}}{1-t^{\l_i}} \prod_{j \geq 1} j^{m_j(\l)} \cdot m_j(\l)!
\end{equation*}

The \emph{Macdonald symmetric polynomials} $\{P_\l(x_1,\ldots,x_n;q,t)\}_{\l \in \Sig_n^+}$ are a distinguished basis for $\L_n$ characterized by the properties
\begin{itemize}
    \item They are orthogonal with respect to the inner product $\la \cdot,\cdot \ra_{(q,t)}$.
    \item They may be written as 
    \begin{equation*}
        P_\l(x_1,\ldots,x_n;q,t) = x_1^{\l_1}x_2^{\l_2}\cdots x_n^{\l_n} + \text{(lower-order monomials in the lexicographic order)}.
    \end{equation*}
\end{itemize}
It is not \emph{a priori} clear that such polynomials exist, see \cite{macdonald1998orthogonal} for a proof of this, but it is clear that they form a basis for $\L_n$. 
The dual basis $Q_\l(\bx;q,t)$ is defined by 
\begin{equation*}
    Q_\l(x_1,\ldots,x_n;q,t) = \frac{P_\l(x_1,\ldots,x_n;q,t)}{\la P_\l,P_\l \ra_{(q,t)}}.
\end{equation*}

Because the $P_\l$ form a basis for the vector space of symmetric polynomials in $n$ variables, there exist symmetric polynomials $P_{\l/\mu}(x_1,\ldots,x_{n-k};q,t) \in \L_{n-k}$ indexed by $\l \in \Sig_n^+, \mu \in \Sig_k^+$ which are defined by
\begin{equation*}\label{eq:def_skewP}
    P_\l(x_1,\ldots,x_n;q,t) = \sum_{\mu \in \Sig_k^+} P_{\l/\mu}(x_{k+1},\ldots,x_n;q,t) P_\mu(x_1,\ldots,x_k;q,t).
\end{equation*}
We define the skew $Q$ polynomials in a slightly nonstandard way where the lengths of both signatures are the same, in contrast to the skew $P$ polynomials. This differs from the classical treatment \cite{mac}, and is inspired by the higher spin Hall-Littlewood polynomials introduced in \cite{borodin2017family}. For $\l,\nu \in \Sig_n^+$ and $k \geq 1$ arbitrary, define $Q_{\l/\nu}(x_1,\ldots,x_k;q,t)$ by
\begin{equation*}
    Q_{(\l,0[k])}(x_1,\ldots,x_{n+k};q,t) = \sum_{\nu \in \Sig_n^+} Q_{\l/\nu}(x_{n+1},\ldots,x_{n+k};q,t) Q_\nu(x_1,\ldots,x_k;q,t).
\end{equation*}
In particular, 
\begin{equation}\label{eq:Qs_agree}
    Q_{\l/(0[n])}(x_1,\ldots,x_{n+k};q,t) = Q_{(\l,0[k])}(x_1,\ldots,x_{n+k};q,t)
\end{equation}
and we will use both interchangeably. The following lemma may be easily derived from the corresponding statement for symmetric functions in infinitely many variables {\cite[VI.6 Ex. 2(a)]{mac}}. Recall the notation $(a;q)_n := (1-a)(1-aq) \cdots (1-aq^{n-1})$ for $n \geq 0$, with $(a;q)_0 =1$ and $(a;q)_\infty$ defined in the obvious way.
\begin{lemma} \label{lem:branching_formulas}
For $\l \in \Sig_n^+, \mu \in \Sig_{n-1}^+$ with $\mu \prec_P \l$, let
\begin{equation}\label{eq:pbranch}
    \psi_{\l/\mu} :=  \prod_{1 \leq i \leq j \leq n-1} \frac{f(t^{j-i}q^{\mu_i-\mu_j})f(t^{j-i}q^{\l_i-\l_{j+1}})}{f(t^{j-i}q^{\l_i-\mu_j})f(t^{j-i}q^{\mu_i-\l_{j+1}})}
    %\prod_{i < j: \mu(i,j-1) = \l(i,j-1)+1} \frac{(1-t^{\l(i,j-1)}q^{j-i-1})(1-t^{\mu(i,j-1)}q^{j-i+1})}{(1-t^{\l(i,j-1)}q^{j-i})(1-t^{\mu(i,j-1)}q^{j-i})}.
\end{equation}
where $f(u) := (tu;q)_\infty/(qu;q)_\infty$. For $\nu \in \Sig_n^+$ with $\nu \prec_Q \l$, let
\begin{equation}\label{eq:qbranch}
    \varphi_{\l/\nu} := \prod_{1 \leq i \leq j \leq N} \frac{f(t^{j-i}q^{\l_i-\l_j})}{f(t^{j-i}q^{\l_i-\nu_j})} \prod_{1 \leq i \leq j \leq N-1} \frac{f(t^{j-i}q^{\nu_i-\nu_{j+1}})}{f(t^{j-i}q^{\nu_i-\l_{j+1}})}.
\end{equation}

Then for $\l,\nu \in \Sig_{n}^+, \mu \in \Sig_{n-k}^+$, we have
\begin{equation}\label{eq:skewP_branch_formula}
    P_{\l/\mu}(x_1,\ldots,x_k;q,t) = \sum_{\mu = \l^{(1)} \prec_P \l^{(2)} \prec_P \cdots \prec_P \l^{(k)}= \l} \prod_{i=1}^{k-1} x_i^{|\l^{(i+1)}|-|\l^{(i)}|}\psi_{\l^{(i+1)}/\l^{(i)}}%P_{\l^{(i)}/\l^{(i+1)}}(x_i)
\end{equation}
and
\begin{equation}\label{eq:skewQ_branch_formula}
    Q_{\l/\nu}(x_1,\ldots,x_k;q,t) = \sum_{\nu = \l^{(1)} \prec_Q \l^{(2)} \prec_Q \cdots \prec_Q \l^{(k)}=\l} \prod_{i=1}^{k-1} x_i^{|\l^{(i+1)}|-|\l^{(i)}|}\varphi_{\l^{(i+1)}/\l^{(i)}}.
\end{equation}
\end{lemma}

This suggests using the above formulas to extend the definition of $P$ and $Q$ to arbitrary signatures, possibly with negative parts, which we do now. 

\begin{defi}\label{def:skew_gt_patterns}
For $\mu \in \Sig_n, \l \in \Sig_{n+k}$, we define $\GT_P(\l/\mu)$ to be the set of sequences of interlacing signatures $\mu = \l^{(1)} \prec_P \l^{(2)} \prec_P \cdots \prec_P \l^{(k)}= \l$. We will often write $\GT_P(\l)$ for $\GT_P(\l/())$. 

For $\l,\nu \in \Sig_n$, we define $\GT_{Q,k}(\l/\nu)$ to be the set of sequences of length $n$ interlacing signatures $\nu = \l^{(1)} \prec_Q \l^{(2)} \prec_Q \cdots \prec_Q \l^{(k)}=\l$. We refer to elements of either $\GT_P$ or $\GT_{Q,k}$ as Gelfand-Tsetlin patterns, see \Cref{fig:shifted_GT_scheme}.

For $T \in \GT_P(\l/\mu)$ with $\len(\l) = \len(\mu)+k$, set $\psi(T) := \prod_{i=1}^{k-1} \psi_{\l^{(i+1)}/\l^{(i)}}$. For $T \in \GT_{Q,k}(\l/\nu)$, set $\varphi(T) := \prod_{i=1}^{k-1} \varphi_{\l^{(i+1)}/\l^{(i)}}$. In both cases, let $wt(T) := (|\l^{(1)}|,|\l^{(2)}|-|\l^{(1)}|,\ldots,|\l^{(k)}|-|\l^{(k-1)}|) \in \Z^k$. 
\end{defi}

%from https://arxiv.org/format/1310.8007
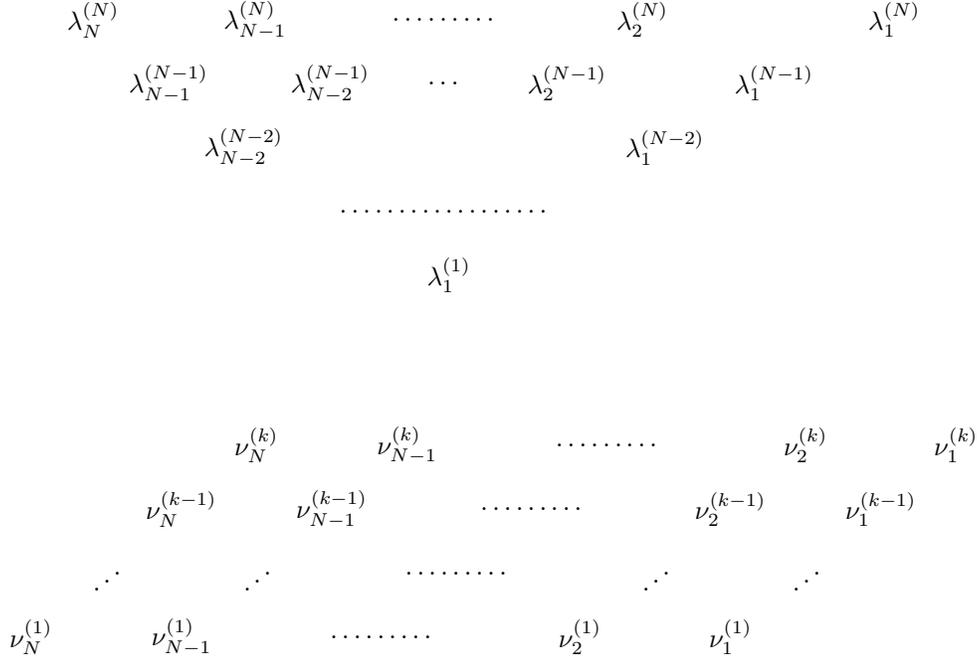
\begin{figure}[htbp]
	\begin{center}
	\begin{tikzpicture}[scale=1]
		\def\h{0.85}
		\def\x{.7}
		\def\y{.15}
		\node at (-5,5) {$\l^{(N)}_N$};
		\node at (-3+\y,5) {$\l_{N-1}^{(N)}$};
		\node at (0-\x/2,5) {$\ldots\ldots\ldots$};
		\node at (3-\x+0*\y,5) {$\l_{2}^{(N)}$};
		\node at (5-\x+9*\y,5) {$\l_{1}^{(N)}$};
		\node at (-4,5-\h) {$\l^{(N-1)}_{N-1}$};
		\node at (-2+\y,5-\h) {$\l^{(N-1)}_{N-2}$};
		\node at (0-\x/2,5-\h) {$\ldots$};
		\node at (2-\x+0*\y,5-\h) {$\l^{(N-1)}_2$};
		\node at (4-\x+5*\y,5-\h) {$\l^{(N-1)}_1$};
		\node at (-3,5-2*\h) {$\l^{(N-2)}_{N-2}$};
		\node at (3-\x+2*\y,5-2*\h) {$\l^{(N-2)}_1$};
    	\node at (0-\x/2,5-3*\h) {$\ldots\ldots\ldots\ldots\ldots\ldots$};
    	\node at (0-\x/2+.08,5-4*\h) {$\l^{(1)}_1$};
	\end{tikzpicture} 
	
    \vspace{10ex}
    
	\begin{tikzpicture}[scale=1]
		\def\h{0.85}
		\def\x{.7}
		\def\y{0}
		\node at (-5,5) {$\nu^{(k)}_N$};
		\node at (-3+\y,5) {$\nu_{N-1}^{(k)}$};
		\node at (0-\x/2,5) {$\ldots\ldots\ldots$};
		\node at (3-\x+0*\y,5) {$\nu_{2}^{(k)}$};
		\node at (5-\x+9*\y,5) {$\nu_{1}^{(k)}$};
		
		\node at (-6,5-\h) {$\nu^{(k-1)}_N$};
		\node at (-4+\y,5-\h) {$\nu_{N-1}^{(k-1)}$};
		\node at (-1-\x/2,5-\h) {$\ldots\ldots\ldots$};
		\node at (2-\x+0*\y,5-\h) {$\nu_{2}^{(k-1)}$};
		\node at (4-\x+9*\y,5-\h) {$\nu_{1}^{(k-1)}$};

		\node at (-7,5-2*\h) {$\iddots$};		
        \node at (-5+\y,5-2*\h) {$\iddots$};
		\node at (-2-\x/2,5-2*\h) {$\ldots\ldots\ldots$};
		\node at (1-\x+0*\y,5-2*\h) {$\iddots$};
		\node at (3-\x+9*\y,5-2*\h) {$\iddots$};
		
		\node at (-8,5-3*\h) {$\nu^{(1)}_N$};
		\node at (-6+\y,5-3*\h) {$\nu_{N-1}^{(1)}$};
		\node at (-3-\x/2,5-3*\h) {$\ldots\ldots\ldots$};
		\node at (0-\x+0*\y,5-3*\h) {$\nu_{2}^{(1)}$};
		\node at (2-\x+9*\y,5-3*\h) {$\nu_{1}^{(1)}$};
	\end{tikzpicture} 
	\end{center}
	\label{fig:shifted_GT_scheme}
	\caption{An element of $\GT_P(\l^{(N)})$ (top) and an element of $\GT_{Q,k}(\nu^{(k)}/\nu^{(1)})$ (bottom). Note that $\GT_{Q,k}$ depends on a parameter $k$ specifying the number of rows, while for $\GT_P$ the data of the number of rows was already determined by the respective lengths of $\l$ and $\mu$.}
\end{figure}

In what follows, we often write $\bx$ for the collection of variables $x_1,\ldots,x_n$ when $n$ is clear from context, and $\ba = (a_1,\ldots,a_n)$ for a collection of $n$ real numbers. For $\bd \in \Z^n$ we write $\bx^\bd := x_1^{d_1} \cdots x_n^{d_n}$ and similarly for $\ba^\bd$.

\begin{defi}\label{def:skew_signature_functions}
For any $n \geq 0, k \geq 1$ and $\l \in \Sig_{n+k}, \mu \in \Sig_{n}$, we let 
\begin{equation*}%\label{eq:Pbranch_with_GT}
    P_{\l/\mu}(x_1,\ldots,x_k;q,t) = \sum_{T \in \GT_P(\l/\mu)} \psi(T) \bx^{wt(T)}%(x_1 \cdots x_n)^k P_{(\l_1+k,\ldots,\l_{N+n}+k)/(\mu_1+k,\ldots,\mu_{N}+k)}(x_1,\ldots,x_n;q,t).
\end{equation*}
For any $\nu, \kappa \in \Sig_n$ we let 
\begin{equation*}%\label{eq:Qbranch_with_GT}
    Q_{\kappa/\nu}(x_1,\ldots,x_k;q,t) = \sum_{T \in \GT_{Q,k}(\kappa/\nu)} \varphi(T) \bx^{wt(T)}. %Q_{(\kappa_1+k,\ldots,\kappa_N+k)/(\nu_1+k,\ldots,\nu_N+k)}(x_1,\ldots,x_n;q,t)
\end{equation*}
\end{defi}

Note that these are just the formulas in \Cref{lem:branching_formulas}, with the only change being that we do not require the signatures to be nonnegative. These combinatorial formulas make some symmetries readily apparent, as noted in \cite{gorin2020crystallization}.

\begin{lemma}\label{lem:basic_signature_func_properties}
Let $\l,\nu \in \Sig_{n+k}, \mu \in \Sig_{n}$. Then
\begin{align}
    P_{-\l/-\mu}(x_1,\ldots,x_k;q,t) &= P_{\l/\mu}(x_1^{-1},\ldots,x_k^{-1};q,t) \label{eq:p_invert_vars}\\
    Q_{-\l/-\nu}(x_1,\ldots,x_r;q,t) &= Q_{\l/\nu}(x_1^{-1},\ldots,x_r^{-1};q,t) \label{eq:q_invert_vars}\\
    P_{(\l + (d[n+k]))/(\mu+d[n])}(x_1,\ldots,x_k;q,t) &= (x_1\cdots x_k)^d P_{\l/\mu}(x_1,\ldots,x_k;q,t) \label{eq:p_add_constant}\\
    Q_{(\l+(d[n+k]))/(\nu + (d[n+k]))}(x_1,\ldots,x_r;q,t) &=  Q_{\l/\nu}(x_1,\ldots,x_r;q,t). \label{eq:q_add_constant}
\end{align}
\end{lemma}
\begin{proof}
First interpret the $P$ and $Q$ polynomials as sums over GT patterns by \Cref{def:skew_signature_functions}.
For \eqref{eq:p_invert_vars} and \eqref{eq:q_invert_vars}, note that we have bijections
\begin{align*}
    \GT_P(\l/\mu)& \leftrightarrow \GT_P(-\l/-\mu) \\
    \GT_{Q,r}(\l/\nu)& \leftrightarrow \GT_{Q,r}(-\l/-\nu)
\end{align*}
which can be directly verified to preserve the branching coefficients $\psi$ defined in \eqref{eq:pbranch} and $\varphi$ defined in \eqref{eq:qbranch}.
For \eqref{eq:p_add_constant} and \eqref{eq:q_add_constant}, one similarly has bijections
\begin{align*}
    \GT_P(\l/\mu)& \leftrightarrow \GT_P((\l + (d[n+k]))/(\nu+(d[n])) \\
    \GT_{Q,r}(\l/\nu)& \leftrightarrow \GT_{Q,r}((\l+(d[n+k]))/(\nu + (d[n+k]))
\end{align*}
by adding $d$ to each entry of the GT patterns, and these preserve the branching coefficients.
\end{proof}

\begin{lemma}\label{lem:asym_cauchy}
Let $\nu \in \Sig_N, \mu \in \Sig_{N+k}$, and $x_1,\ldots,x_k, y_1,\ldots,y_r$ be indeterminates. Then
\begin{multline}\label{eq:asym_cauchy}
    \sum_{\kappa \in \Sig_{N+k}} Q_{\kappa/\mu}(y_1,\ldots,y_r;q,t) P_{\kappa/\nu}(x_1,\ldots,x_k;q,t) \\
    = \Pi_{(q,t)}(x_1,\ldots,x_k; y_1,\ldots,y_r)\sum_{\kappa \in \Sig_N} Q_{\nu/\l}(y_1,\ldots,y_r;q,t) P_{\mu/\l}(x_1,\ldots,x_k;q,t) 
\end{multline}
where 
\begin{equation*}
    \Pi_{(q,t)}(x_1,\ldots,x_k; y_1,\ldots,y_r) = \prod_{\substack{1 \leq i \leq k \\ 1 \leq j \leq r}} \frac{(tx_iy_j;q)_\infty}{(x_iy_j;q)_\infty}
\end{equation*}
and \eqref{eq:asym_cauchy} is interpreted as an equality of formal power series in the variables.
\end{lemma}
\begin{proof}
When $\nu \in \Sig_N^+, \mu \in \Sig_{N+k}^+$, the result follows by specializing the usual Cauchy identity \cite[Ch. VI.7, Ex. 6(a)]{mac} to finitely many variables and replacing partitions by nonnegative signatures. The result then follows for general signatures since replacing $\nu,\mu$ by $\nu+(d[N]), \mu + (d[N+k])$ in \eqref{eq:asym_cauchy} multiplies both sides by $(x_1\cdots x_k)^d $ by \Cref{lem:basic_signature_func_properties}. 
\end{proof}

We note that the set $\{P_\l(\bx;q,t): \l \in \Sig_n\}$ forms a basis for the ring of symmetric Laurent polynomials $\L_n[(x_1 \cdots x_n)^{-1}]$. Hence for any $\l,\mu \in \Sig_n$ one has
\begin{equation}\label{eq:mac_func_struct_coefs}
    P_\l(\bx;q,t) \cdot P_\mu(\bx;q,t) = \sum_{\nu \in \Sig_n} c_{\l,\mu}^\nu(q,t) P_\nu(\bx;q,t)
\end{equation}
for some structure coefficients $c_{\l,\mu}^\nu(q,t)$. By matching degrees it is clear that these coefficients are nonzero only if $|\l|+|\mu|=|\nu|$. These multiplicative structure coefficients for the $P$ polynomials are related to the `comultiplicative' structure constants of the $Q$ polynomials.

\begin{prop}\label{prop:q_coproduct_coefs}
Let $m,n \in \N$ and $\l,\nu \in \Sig_n$. Then
\begin{equation*}
    Q_{\l/\nu}(x_1,\ldots,x_m;q,t) = \sum_{\mu \in \Sig_n} c_{\nu,\mu}^\l Q_\mu(x_1,\ldots,x_m;q,t).
\end{equation*}
\end{prop}

When all signatures are nonnegative this follows by specializing the corresponding statement for symmetric functions, see \cite[Ch. VI.7]{mac} where \Cref{prop:q_coproduct_coefs} is taken as the definition of the skew $Q$ polynomials. The case of general signatures follows by shifting arguments as before.

\subsection{Probabilistic constructions from symmetric functions.} 

For the remainder of this section assume $q,t \in [0,1)$.

\begin{defi}\label{def:mac_measure}
Let $\ba = (a_1,\ldots,a_n), \mathbf{b} = (b_1,\ldots,b_n)$ be nonnegative reals such that $a_ib_j < 1$ for all $i,j$. The \emph{Macdonald measure} with specializations $\ba,\mathbf{b}$ is the measure on $\Sig_n^+$ given by
\begin{equation*}
    \Pr(\l) = \frac{P_\l(\ba;q,t) Q_{\l}(\mathbf{b};q,t)}{\Pi_{(q,t)}(\ba;\mathbf{b})}.
\end{equation*}
\end{defi}
It is known, see e.g. \cite{borodin2014macdonald}, that for real numbers $a_1,\ldots,a_n \geq 0$ and $q,t \in [0,1)$, $P_\l(a_1,\ldots,a_n;q,t) \geq 0$ and similarly for $Q_\l$, so the above is indeed positive, and the sum over $\Sig_n^+$ is $1$ by  \Cref{lem:asym_cauchy}. The condition $a_ib_j < 1$ is needed so that the sum of weights $P_\l(a_1,\ldots,a_n;q,t) Q_{\l}(b_1,\ldots,b_n;q,t)$ does not diverge. 

Slightly more generally, \Cref{lem:asym_cauchy} may be used to define Markov transition dynamics. Note that now the signatures do not have to be nonnegative.

\begin{prop}\label{def:mac_cauchy_dynamics}
Let $\ba = (a_1,\ldots,a_n), \mathbf{b} = (b_1,\ldots,b_k)$ be nonnegative reals such that $a_ib_j < 1$ for all $i,j$. Then the formulas
\begin{equation*}
    \Pr(\nu \to \l) = \frac{P_\l(\ba;q,t) Q_{\l/\nu}(\mathbf{b};q,t)}{P_\nu(\ba;q,t)\Pi_{(q,t)}(\ba;\mathbf{b})}
\end{equation*}
and 
\begin{equation*}
    \Pr(\l \to \mu) =  \frac{P_{\l/\mu}(a_{k+1},\ldots,a_n;q,t)P_\mu(a_1,\ldots,a_k;q,t)}{P_\l(a_1,\ldots,a_n;q,t)}.
\end{equation*}
define Markov transition dynamics $\Sig_n \to \Sig_n$ and $\Sig_n \to \Sig_k$ respectively.
\end{prop}

The content of the coming \Cref{sec:3} and \Cref{sec:3.2} is that all of the previous measures and Markov transition dynamics may be recovered as limiting cases of product convolution defined below.

\begin{defi}\label{def:product_dynamics}
Let $\ba = (a_1,\ldots,a_n)$ be nonnegative reals. Then given $\l,\mu \in \Sig_n$, we define a random signature $\l \boxtimes_\ba \mu$ by 
\begin{equation*}
    \Pr(\l \boxtimes_\ba \mu = \nu) = \frac{P_\nu(a_1,\ldots,a_n;q,t)}{P_\l(a_1,\ldots,a_n;q,t)P_\mu(a_1,\ldots,a_n;q,t)} c_{\l,\mu}^\nu(q,t).
\end{equation*}
\end{defi}

Let us say a brief word on specializations in infinitely many real parameters. Though we mostly work with symmetric Laurent polynomials, we will sometimes wish to consider $Q_{\l/\nu}(a_1,a_2,\ldots;q,t)$ with $a_i \in \R$. These can be made sense of as follows. In general, for $\l,\nu \in \Sig_n$, $Q_{\l/\nu}(x_1,\ldots,x_r;q,t)$ is a polynomial in the power sums $p_k$, and this polynomial is independent of $r$ for all sufficiently large $r$. Provided that $\sum_{i \geq 1} a_i < \infty$, we may define $p_k(a_1,a_2,\ldots) := \sum_{i \geq 1} a_i^k$, and this defines $Q_{\l/\nu}(a_1,\ldots;q,t)$ since $Q_{\l/\nu}$ is a polynomial in the $p_k$.

This is of course a very special case of the general machinery of symmetric functions and their specializations, for which one may refer to e.g. \cite{borodin2014macdonald}.

\begin{defi}\label{def:mac_proc}
Let $k,n \in \N$ and $\ba^{(i)}$ be sequences (finite or infinite) of real numbers $a^{(i)}_j$ for $i=1,\ldots,k$, and $\mathbf{b} = (b_1,\ldots,b_n)$, such that $\sum_{j \geq 1} a^{(i)}_j < \infty$ for each $i$ and $a^{(i)}_jb_\ell < 1$ for all $i,j,\ell$. Then the (ascending) Macdonald process with these specializations is the probability measure on $\Sig_n^k$ with
\begin{equation*}
    \Pr(\l^{(1)},\ldots,\l^{(k)}) = \frac{Q_{\l^{(1)}/(0[n])}(\ba^{(1)};q,t) Q_{\l^{(2)}/\l^{(1)}}(\ba^{(2)};q,t) \cdots Q_{\l^{(k)}/\l^{(k-1)}}(\ba^{(k)};q,t) P_{\l^{(k)}}(\mathbf{b};q,t)}{\Pi_{(q,t)}(\mathbf{b}; \ba^{(1)},\ldots,\ba^{(k)})}.
\end{equation*}
\end{defi}

\begin{prop}\label{prop:product_gives_proc}
Fix $\ba^{(i)}, \mathbf{b}$ as in \Cref{def:mac_proc}, and let $\nu^{(i)}$ be distributed by the Macdonald measure with specializations $\ba^{(i)},\mathbf{b}$ for each $i=1,\ldots,k$. Then for any fixed $\l^{(i)} \in \Sig_n, i = 1,\ldots,k$ we have
\begin{multline*}
    \Pr(\nu^{(1)} \boxtimes_{\mathbf{b}} \cdots \boxtimes_{\mathbf{b}} \nu^{(\tau)} = \l^{(\tau)} \text{ for all }\tau=1,\ldots,k) \\%\nu^{(1)} = \l^{(1)}, \nu^{(1)} \boxtimes_{\mathbf{b}} \nu^{(2)} = \l^{(2)},\ldots,\nu^{(1)} \boxtimes_{\mathbf{b}} \nu^{(2)} \boxtimes_{\mathbf{b}} \cdots \boxtimes_{\mathbf{b}} \nu^{(k)} = \l^{(k)}) \\
    = \frac{Q_{\l^{(1)}/(0[n])}(\ba^{(1)};q,t) Q_{\l^{(2)}/\l^{(1)}}(\ba^{(2)};q,t) \cdots Q_{\l^{(k)}/\l^{(k-1)}}(\ba^{(k)};q,t) P_{\l^{(k)}}(\mathbf{b};q,t)}{\Pi_{(q,t)}(\mathbf{b}; \ba^{(1)},\ldots,\ba^{(k)})}.
\end{multline*}
\end{prop}

The proof is simple algebraic manipulation using \Cref{prop:q_coproduct_coefs} to absorb the structure coefficients $c_{\l,\mu}^\nu$.

\subsection{Hall-Littlewood polynomials.}\label{subsec:hl}

Though we will prove some results for general Macdonald polynomials in order to indicate how the same structure lies behind both complex and $p$-adic random matrix theory, for $p$-adic results we will only need the case when $q=0$, known as the \emph{Hall-Littlewood polynomials}. We record several formulas for these which are useful for turning our results into explicit calculations, all of which may be found in \cite[Ch. III]{mac}.

\begin{prop}[Explicit formulas for Hall-Littlewood polynomials]\label{prop:hl_explicit_formulas} %quote from mac, explicit formula for general hl poly and for princ spec ones, fill in later
For $\l \in \Sig_n$, let
\begin{equation*}
    v_\l(t) = \prod_{i \in \Z} \frac{(t;t)_{m_i(\l)}}{(1-t)^{m_i(\l)}}.
\end{equation*}
When $q=0$ we have
\begin{equation*}
    P_\l(\bx;0,t) = \frac{1}{v_\l(t)} \sum_{\sigma \in S_n} \sigma\left(\bx^\l \prod_{1 \leq i < j \leq n} \frac{x_i-tx_j}{x_i-x_j}\right)
\end{equation*}
where $\sigma$ acts by permuting the variables, and $Q_\l(\bx;0,t) = \prod_{i \in \Z} (t;t)_{m_i(\l)} P_\l(\bx;0,t)$. 
\end{prop}

The following is easily derived from \Cref{def:skew_signature_functions} by setting $q=0$.

\begin{prop}[Branching coefficients for Hall-Littlewood polynomials]\label{prop:branching_hl}
Let $\l,\nu \in \Sig_n, \mu \in \Sig_{n-1}$ with $\nu \prec_Q \l, \mu \prec_P \l$. Then
\begin{align*}
    P_{\l/\mu}(x;0,t) &= x^{|\l|-|\mu|} \prod_{\substack{i \in \Z \\ m_i(\mu) = m_i(\l)+1}} (1-t^{m_i(\mu)}) \\
    Q_{\l/\nu}(x;0,t) &= x^{|\l|-|\nu|} \prod_{\substack{i \in \Z \\ m_i(\l) = m_i(\nu)+1}} (1-t^{m_i(\l)}).
\end{align*}
\end{prop}

Let $n(\l) := \sum_{i=1}^n (i-1)\l_i$. The following formulas may be easily derived from \Cref{prop:hl_explicit_formulas} and \eqref{eq:Qs_agree}. 

\begin{prop}[Principal specialization formulas]\label{prop:hl_principal_formulas}
For $J,n \geq 1$ and $\l \in \Sig_n^+$,
\begin{align*}
    P_\l(x,xt,\ldots,xt^{n-1};0,t) &= x^{|\l|} t^{n(\l)} \frac{(t;t)_n}{\prod_{i \in \Z} (t;t)_{m_i(\l)}} \\
    Q_{\l/(0[n])}(x,xt,\ldots,xt^{J-1};0,t) &= x^{|\l|} t^{n(\l)} \frac{(t;t)_J}{(t;t)_{m_0(\l)+J-n}} \bbone(m_0(\l)+J-n \geq 0)
\end{align*}
\end{prop}

Note that the principal specialization formula for $Q$ differs from the statement in \cite[Ch. III.2, Ex. 1]{mac} due to our conventions on signatures, but it may be derived directly from that statement using \eqref{eq:Qs_agree} to translate between skew and non-skew $Q$ polynomials. There are also nice explicit formulas for principally specialized Macdonald polynomials, see \cite[Ch. VI]{mac}, but we will not need these except briefly in \Cref{sec:appendixA}.

From now on, we will typically omit the `$;q,t$' in the arguments of Macdonald or Hall-Littlewood polynomials when they are clear from context.

\section{{$p$}-adic random matrices and Hall-Littlewood processes}\label{sec:3}

In this section, we give the necessary background on $p$-adic random matrices. We state (but do not yet prove) two results on limits of Macdonald structure coefficients $c_{\l,\mu}^\nu(q,t)$, \Cref{prop:measure_convergence_jacobi_and_cauchy} and \Cref{prop:measure_convergence_branching_corners}. Using these, we deduce the main structural result on $p$-adic random matrices, \Cref{thm:exact_hl_results_in_p-adic_rmt}, and the `Ginibre ensemble' limit \Cref{cor:reprove_wood}. 

\subsection{$p$-adic random matrix background and proof of {\Cref{thm:exact_hl_results_in_p-adic_rmt}}.} 
The following is a condensed version of the exposition in \cite[Section 2]{evans2002elementary}, to which we refer any reader desiring a more detailed introduction to $p$-adic numbers. Fix a prime $p$. Any nonzero rational number $r \in \Q^\times$ may be written as $r=p^k (a/b)$ with $k \in \Z$ and $a,b$ coprime to $p$. Define $|\cdot|: \Q \to \R$ by setting $|r|_p = p^{-k}$ for $r$ as before, and $|0|_p=0$. Then $|\cdot|_p$ defines a norm on $\Q$ and $d_p(x,y) :=|x-y|_p$ defines a metric. We define the \emph{field of $p$-adic numbers} $\Q_p$ to be the completion of $\Q$ with respect to this metric, and the \emph{$p$-adic integers} $\Z_p$ to be the unit ball $\{x \in \Q_p : |x|_p \leq 1\}$. It is not hard to check that $\Z_p$ is a subring of $\Q_p$. We remark that $\Z_p$ may be alternatively defined as the inverse limit of the system $\ldots \to \Z/p^{n+1}\Z \to \Z/p^n \Z \to \cdots \to \Z/p\Z \to 0$, and that $\Z$ naturally includes into $\Z_p$. 

$\Q_p$ is noncompact but is equipped with a left- and right-invariant (additive) Haar measure; this measure is unique if we normalize so that the compact subgroup $\Z_p$ has measure $1$. The restriction of this measure to $\Z_p$ is the unique Haar probability measure on $\Z_p$, and is explicitly characterized by the fact that its pushforward under any map $r_n:\Z_p \to \Z/p^n\Z$ is the uniform probability measure. For concreteness, it is often useful to view elements of $\Z_p$ as `power series in $p$' $a_0 + a_1 p + a_2 p^2 + \ldots$, with $a_i \in \{0,\ldots,p-1\}$; clearly these specify a coherent sequence of elements of $\Z/p^n\Z$ for each $n$. The Haar probability measure then has the alternate explicit description that each $a_i$ is iid uniformly random from $\{0,\ldots,p-1\}$. Additionally, $\Q_p$ is isomorphic to the ring of Laurent series in $p$, defined in exactly the same way.

Similarly, $\GL_N(\Q_p)$ has a unique left- and right-invariant measure for which the total mass of the compact subgroup $\GL_N(\Z_p)$ is $1$. We denote this measure by $\M$. The restriction of $\M$ to $\GL_N(\Z_p)$ pushes forward to $\GL_N(\Z/p^n\Z)$; these measures are the uniform measures on the finite groups $\GL_N(\Z/p^n\Z)$. This gives an alternative characterization of the measure.

The following standard result is sometimes known as Smith normal form and holds also for more general rings.

\begin{prop}\label{prop:smith}
Let $n \leq m$. For any $A \in M_{n \times m}(\Q_p)$, there exist $U \in \GL_n(\Z_p), V \in \GL_m(\Z_p)$ such that $UAV = \diag_{n \times m}(p^{\l_1},\ldots,p^{\l_n})$ for some signature $\l$, where when $A$ is singular we formally allow parts of $\l$ to equal $\infty$ and define $p^\infty = 0$. Furthermore, there is a unique such signature $\l$ with possibly infinite parts.
\end{prop}

\begin{defi}\label{def:singular_numbers}
We denote by $\Sig_n^*$ the set of \emph{extended signatures} $\l = (\l_1,\ldots,\l_n)$ where $\l_i \in \Z \cup \{\infty\}$, $\l_i \geq \l_{i+1}$ for all $i$, and we take $\infty > k$ for all $k \in \Z$. For any $n \leq m$ and $A \in M_{n \times m}(\Q_p)$, we let $\SN(A) \in \Sig_n^*$ denote the extended signature of \Cref{prop:smith}, which we refer to as the \emph{singular numbers} of $A$. Note the convention that the length of $\l$ is the smaller dimension of $A$.
\end{defi}

We will often write $\diag_{n \times N}(p^\l)$ for $\diag_{n \times N}(p^{\l_1},\ldots,p^{\l_n})$, and also omit the dimensions $n \times N$ when they are clear from context. We note also that for any $\l \in \Sig_N$, the double coset $\GL_N(\Z_p) \diag(p^\l) \GL_N(\Z_p)$ is compact. The restriction of $\M$ to such a double coset, normalized to be a probability measure, is the unique $\GL_N(\Z_p) \times \GL_N(\Z_p)$-invariant probability measure on $\GL_N(\Q_p)$ with singular numbers $\l$, and all $\GL_N(\Z_p) \times \GL_N(\Z_p)$-probability measures and convex combinations of these for different $\l$. These measures may be equivalently described as $U \diag(p^{\l_1},\ldots,p^{\l_N}) V$ where $U,V$ are independently distributed by the Haar probability measure on $\GL_N(\Z_p)$. More generally, if $n \leq m$ and $U \in \GL_n(\Z_p), V \in \GL_m(\Z_p)$ are Haar distributed and $\mu \in \Sig^*_n$, then $U \diag_{n \times m}(p^\mu) V$ is invariant under $\GL_n(\Z_p) \times \GL_m(\Z_p)$ acting on the left and right, and is the unique such bi-invariant measure with singular numbers given by $\mu$. These bi-invariant measures on rectangular matrices are the ones which appear in \Cref{thm:exact_hl_results_in_p-adic_rmt}.

\exactresults*

Let us flesh out the discussion from the Introduction on how \Cref{thm:exact_hl_results_in_p-adic_rmt} is proven. For Part 3, the results are essentially already contained in \cite[Ch. V]{mac} and must be translated to probabilistic language. Parts 1 and 2 both concern the operation of taking submatrices of a random matrix, which is equivalent to the multiplicative convolution of Part 3 with projection matrices. There is a slight difficulty because Part 3 is a statement about the pair $(\GL_n(\Q_p),\GL_n(\Z_p))$ and hence holds only for nonsingular matrices $A,B$, so one must make a limiting argument with nonsingular matrices which are very close to projection matrices, e.g. $\diag(1[N-k], p^D[k])$ for large $D$--recall that in the $p$-adic norm, $p^D$ for large $D$ is very small. Since Part 3 relates matrix products to the structure coefficients $c_{\l,\mu}^\nu(0,t)$, to implement the above idea of matrices which limit to projectors we must understand asymptotics of $c_{\l,\mu}^\nu(0,t)$ for sequences of $\l,\mu$ which approach the singular numbers of projection matrices. We will first state the relevant asymptotic results on Macdonald structure coefficients, \Cref{prop:measure_convergence_jacobi_and_cauchy} and \Cref{prop:measure_convergence_branching_corners}, then prove \Cref{thm:exact_hl_results_in_p-adic_rmt} conditional on these. This illustrates why these are the right asymptotic results on structure coefficients for our setting, which may not be apparent from the first glance. In the next subsection we will then develop the machinery to establish \Cref{prop:measure_convergence_jacobi_and_cauchy} and \Cref{prop:measure_convergence_branching_corners}.

\begin{prop}\label{prop:measure_convergence_jacobi_and_cauchy}
Let $q,t \in (-1,1)$ be such that the structure coefficients $c_{\l,\mu}^\nu(q,t)$ are all nonnegative\footnote{Conjecturally, this is true if $q,t \in [0,1)$ or if $q,t \in (-1,0]$, see Matveev \cite{matveev2019macdonald}. For our application we will only need the case $q=0, t = 1/p \in (0,1)$, for which the nonnegativity follows from the interpretation of the structure coefficients in terms of the Hall algebra \cite[Ch. III]{mac}, or alternatively from \Cref{thm:exact_hl_results_in_p-adic_rmt} Part 3.}. Let $n \leq m \leq N$ be integers such that $n \leq N-m$, let $\l \in \Sig_n$, and let $a_1 \geq a_2 \geq \ldots \geq a_N > 0$ be real numbers and $\ba = (a_1,\ldots,a_N)$. Let $\mcj$ be the probability measure on $\Sig_n$ with distribution defined by taking the last $n$ parts of a random signature $\kappa=(D[N-n],\l) \boxtimes_{\ba} (D[N-m],0[m])$, or explicitly,
\begin{equation*}
    \mcj(\nu) = \sum_{\substack{\kappa \in \Sig_N \\ \kappa_{N-n+i}=\nu_i \text{ for all }i=1,\ldots,n}} c_{(D[N-n],\l),(D[N-m],0[m])}^\kappa(q,t) \frac{P_\kappa(\ba)}{P_{(D[N-n],\l)}(\ba)P_{(D[N-m],0[m])}(\ba)}.
\end{equation*}
Then for each $\nu \in \Sig_n$,
\begin{equation}\label{eq:limit_of_cauchy_jacobi_measure}
    \mcj(\nu) \to \frac{Q_{\nu/\l}(a_1^{-1},\ldots,a_{N-m}^{-1})P_\nu(a_{N-n+1},\ldots,a_N)}{P_\l(a_{N-n+1},\ldots,a_N) \Pi(a_1^{-1},\ldots,a_{N-m}^{-1};a_{N-n+1},\ldots,a_N)}
\end{equation}
as $D \to \infty$.
\end{prop}

\begin{prop}\label{prop:measure_convergence_branching_corners}
Let $q,t \in (-1,1)$ be such that the structure coefficients $c_{\l,\mu}^\nu(q,t)$ are all nonnegative. Let $0 < k \leq n$ be integers, let $\l \in \Sig_n$, and let $a_1 \geq a_2 \geq \ldots \geq a_n > 0$ be real numbers and $\ba = (a_1,\ldots,a_n)$. Let $\mbr$ be the probability measure on $\Sig_{n-k}$ with distribution defined by taking the last $n-k$ parts of a random signature $\kappa=\l \boxtimes_{\ba} (D[k],0[n-k])$, or explicitly,
\begin{equation*}
    \mbr(\mu) = \sum_{\substack{\kappa \in \Sig_n \\ \kappa_{k+i}=\mu_i \text{ for all }i=1,\ldots,n-k}} c_{\l,(D[k],0[n-k])}^\kappa(q,t) \frac{P_\kappa(\ba;q,t)}{P_{\l}(\ba;q,t)P_{(D[k],0[n-k])}(\ba;q,t)}.
\end{equation*}
Then for each $\mu \in \Sig_{n-k}$,
\begin{equation*}
    \mbr(\mu) \to P_{\l/\mu}(a_1,\ldots,a_k) \frac{P_\mu(a_{k+1},\ldots,a_n)}{P_\l(\ba)} \text{     as }D \to \infty.
\end{equation*}
\end{prop}

\begin{proof}[Proof of {\Cref{thm:exact_hl_results_in_p-adic_rmt}}, conditional on {\Cref{prop:measure_convergence_branching_corners}} and {\Cref{prop:measure_convergence_jacobi_and_cauchy}}.]
We first prove Part 3. In this proof we will use essentially the notation of \cite[Ch. V]{mac} to state the relevant results and then show how ours follow. Let $\G = \GL_n(\Q_p)$, $\K = \GL_n(\Z_p)$, and $L(\G,\K)$ denote the algebra of compactly supported functions $f:\G \to \C$ which are bi-invariant under $\K$, i.e. $f(k_1 x k_2) = f(x)$ for $x \in \G, k_1,k_2 \in \K$. Define a convolution operation on $L(\G,\K)$ by 
\begin{equation*}
    (f * g)(x) = \int_{\G} f(xy^{-1}) g(y) dy
\end{equation*}
where the integration is with respect to the Haar measure on $\G$ normalized such that $\K$ has measure $1$, mentioned earlier. This multiplication is associative, and may be checked to be commutative as well. By \Cref{prop:smith}, each double coset $\K x \K$ of $\K \backslash \G / \K$ has a unique representative of the form $\diag(p^\l)$ for some $\l \in \Sig_n$. We abuse notation slightly and write such a double coset as $\K p^\l \K$. We denote by $\bbone_\l$ the indicator function on such a double coset; clearly $\bbone_\l \in L(\G,\K)$. We will use the following two results, both of which may be found in the discussion after \cite[Ch. V, (2.7)]{mac}:
\begin{itemize}
    \item The map $\theta: L(\G,\K) \to \L_n[(x_1 \cdots x_n)^{-1}]$ given by $\theta(\bbone_\l) = t^{n(\l)}P_\l(x_1,\ldots,x_n;0,t)$ is a $\C$-algebra isomorphism, where $n(\l) := \sum_i (i-1) \l_i$. Equivalently,
    \begin{equation*}
        \bbone_\l * \bbone_\mu = \sum_{\nu \in \Sig_n} t^{n(\l)+n(\mu)-n(\nu)} c_{\l,\mu}^\nu(0,t) \bbone_\nu.
    \end{equation*}
    
    \item The measure of each double coset $\K p^\l \K$ is 
    \begin{equation*}
        \M(\K p^\l \K) = t^{n(\l)-(n-1)|\l|} P_\l(1,t,\ldots,t^{n-1};0,t).
    \end{equation*}
\end{itemize}
It follows directly that 
\begin{equation}\label{eq:follows_from_mac's_2}
    \frac{\int_{x \in \K p^\nu \K}(\bbone_\l * \bbone_\mu)(x) dx }{\M(\K p^\l \K) \M(\K p^\mu \K)} = c_{\l,\mu}^\nu(0,t) \frac{P_\nu(1,\ldots,t^{n-1};0,t)}{P_\l(1,\ldots,t^{n-1};0,t)P_\mu(1,\ldots,t^{n-1};0,t)},
\end{equation}
using the fact that either $(n-1)|\l| + (n-1)|\mu| = (n-1)|\nu|$ or the equality is trivial.

But the LHS of \eqref{eq:follows_from_mac's_2}, by the definition of the convolution product, is
\begin{equation*}
    \frac{1}{\M(\K p^\l \K) \M(\K p^\mu \K)} \int_{x,y \in \G} \bbone_\nu(x) \bbone_\l(xy^{-1}) \bbone_\mu(y) dx dy
\end{equation*} 
Setting $A = xy^{-1},B = y$, this is exactly the conditional probability $\Pr(\SN(AB) = \nu)$ as $A,B$ vary over $\K p^\l \K$ and $\K p^\mu \K$ respectively (both normalized to have total measure $1$), so Part 3 of \Cref{thm:exact_hl_results_in_p-adic_rmt} is proven.

Now consider Part 1. The distribution of the top $n$ rows of a Haar-distributed element of $\GL_N(\Z_p)$ is just the unique $\GL_n(\Z_p) \times \GL_N(\Z_p)$-invariant distribution on $M_{n \times N}(\Q_p)$ with singular numbers $\SN(A) = (0[n])$. Hence Part 1 is the special case of \eqref{eq:p-adic_cauchy_corners} when $\l = (0[n])$. Let us deduce \eqref{eq:p-adic_cauchy_corners} from \Cref{prop:measure_convergence_jacobi_and_cauchy}.

Fix $\l \in \Sig_n$. We wish to compute the distribution of the singular numbers $\SN(U \diag_{n \times N}(p^\l) V P_k)$, where $U \in \GL_n(\Z_p), V \in \GL_N(\Z_p)$ are Haar distributed and $P_k = \diag_{N \times N}(1[N-k],0[k])$ is a corank-$k$ projector. Consider a fixed, deterministic $V_0 \in \GL_N(\Z_p)$, and let $\nu = \SN(U \diag_{n \times N}(p^\l) V_0 P_k) \in \Sig_n$. First note that this is independent of $U$, and setting
\begin{equation*}
    A(V_0) := \diag_{N \times N}(p^\l,0[N-n]) V_0 P_k
\end{equation*}
we have $\SN(A(V_0)) = (\infty[N-n],\nu) \in \Sig_N^*$. For $D > \nu_1$ set
\begin{equation*}
    A_D(V_0) := \diag_{N \times N}(p^\l, p^D[N-n]) V_0 \diag_{N \times N}(1[N-k],p^D[k]).
\end{equation*}
We claim that $\SN(A_D)_{N-n+i} = \nu_i$ for each $i=1,\ldots,n$.

Let $r_D:M_N(\Z_p) \to M_N(\Z/p^D\Z)$ be the obvious map. \Cref{prop:smith} holds also for matrices over $\Z/p^D\Z$, so we may abuse notation and define $\SN$ on both $M_N(\Z_p)$ and $M_N(\Z/p^D\Z)$ as in \Cref{def:singular_numbers}. Let $\varphi_D: \Sig_N^* \to \Sig_N^*$ be the map defined as follows: for any $\kappa \in \Sig_N^*$, let $j = \max(\{i: \kappa_i \geq D\})$, and define $\varphi_D(\kappa) := (\infty[j],\kappa_{j+1},\ldots,\kappa_N)$. It is clear that the diagram

\begin{center}
\begin{tikzcd}
M_N(\Z_p) \arrow[r, "\SN"] \arrow[d, "r_D"]
& \Sig_N^* \arrow[d, "\varphi_D"] \\
M_N(\Z/p^D \Z) \arrow[r, "\SN"]
& \Sig_N^*
\end{tikzcd}
\end{center}

\noindent commutes, hence $\SN(A_D(V_0))_{N-n+i} = \nu_i$ for each $i=1,\ldots,n$, because $r_D(A_D(V_0)) = r_D(A(V_0))$ and $\nu_i < D$ for all $i$. Thus for any fixed $\nu$ and $D > \nu_1$, recalling that $V \in \GL_N(\Z_p)$ is Haar distributed, we have
\begin{equation*}
    \Pr(\SN(A_D(V))_{N-n+i} = \nu_i \text{ for }i=1,\ldots,n) = \Pr(\SN(A(V))_{N-n+i} = \nu_i \text{ for }i=1,\ldots,n).
\end{equation*}
This stabilization for $D > \nu_1$ in particular implies that 
\begin{equation*}
    \lim_{D \to \infty} \Pr(\SN(A_D(V))_{N-n+i} = \nu_i \text{ for }i=1,\ldots,n) = \Pr(\SN(A(V))_{N-n+i} = \nu_i \text{ for }i=1,\ldots,n).
\end{equation*}
The RHS is what we want to compute. By Part 3 of \Cref{thm:exact_hl_results_in_p-adic_rmt}, the LHS is equal to 
\begin{equation*}
    \lim_{D \to \infty} \sum_{\substack{\kappa \in \Sig_N \\ \kappa_{N-n+i}=\nu_i \text{ for all }i=1,\ldots,n}}  \frac{c_{(D[N-n],\l),(D[k],0[N-k])}^\kappa(0,t)P_\kappa(1,\ldots,t^{n-1};0,t)}{P_{(D[N-n],\l)}(1,\ldots,t^{n-1};0,t)P_{(D[k],0[N-k])}(1,\ldots,t^{n-1};0,t)}.
\end{equation*}
By \Cref{prop:measure_convergence_jacobi_and_cauchy} this is equal to
\begin{equation*}
    \frac{Q_{\nu/\l}(1,\ldots,t^{-(k-1)})P_\nu(t^{N-n},\ldots,t^{N-1})}{P_\l(t^{N-n},\ldots,t^{N-1}) \Pi(1,\ldots,t^{-(k-1)};t^{N-n},\ldots,t^{N-1})},
\end{equation*}
which is \eqref{eq:p-adic_cauchy_corners}. Setting $m=N-k$ and $\l = (0[n])$, and dividing all variables in both $P$ specializations by $t^{N-n}$ and multiplying those in the $Q$ specialization by $t^{N-n}$, yields \Cref{thm:exact_hl_results_in_p-adic_rmt} Part 1. 

It remains to prove the other case of Part 2, namely \eqref{eq:p-adic_corners}. One wishes to compute the distribution of $P_d U \diag_{n \times N}(p^\l) V$ where $P_d \in M_n(\Z_p)$ is a corank $d$ projector and $U,V$ are as above. We may ignore $V$, and by the same argument as before it suffices to consider the matrix $\diag_{n \times n}(0[n-d],p^D[d]) U \diag_{n \times n}(p^\l)$ for large $D$. One then applies \Cref{prop:measure_convergence_branching_corners} to yield \eqref{eq:p-adic_corners}.
\end{proof}

From \Cref{thm:exact_hl_results_in_p-adic_rmt} Part 1 we deduce the following. Recall that the intuition for this statement is that if $N$ is very large compared to $m,n$, then the entries of an $n \times m$ corner of a Haar distributed element of $\GL_N(\Z_p)$ become asymptotically iid from the additive Haar measure on $\Z_p$. The analogous statement holds in the complex case, namely that an $n \times m$ corner of a Haar distributed element of $U(N)$ becomes a matrix of iid Gaussians as $N \to \infty$ if one rescales appropriately.

\pwishart*

\begin{proof}
Let $A \in M_{n \times m}(\Z_p)$ be distributed as in \Cref{cor:reprove_wood}, and $B_N \in M_{n \times m}(\Z_p)$ be an $n \times m$ corner of a Haar distributed element of $\GL_N(\Z_p)$. By Part 1 of \Cref{thm:exact_hl_results_in_p-adic_rmt} one has 
\begin{equation*}
    \lim_{N \to \infty} \Pr(\SN(B_N) = \l) = \frac{P_\l(1,\ldots,t^{n-1})Q_\l(t^{m-n+1},t^{m-n+2},\ldots)}{\Pi_{(0,t)}(1,\ldots,t^{n-1};t^{m-n+1},t^{m-n+2},\ldots)},
\end{equation*}
hence it suffices to prove
\begin{equation*}
    \lim_{N \to \infty} \Pr(\SN(B_N) = \l) = \Pr(\SN(A) = \l).
\end{equation*}
It is clear that $A$ and $B_N$ are nonsingular with probability $1$, so $\SN(B_N)$ and $\SN(A)$ lie in $\Sig_n^+$ (rather than $\Sig_n^*$) with probability $1$. Letting $D > \l_1$ be an integer, we have by the argument in the proof of \Cref{thm:exact_hl_results_in_p-adic_rmt} that if $\SN(E) = \l$ then $\SN(r_D(E)) = \l$ as well for any fixed nonsingular $E \in M_{n \times m}(\Z_p)$, where $r_D$ is the reduction modulo $p^D$ map. Therefore $\Pr(\SN(B_N) = \l) = \Pr(\SN(r_D(B_N)) = \l)$ and similarly with $B_N$ replaced by $A$, so it suffices to prove
\begin{equation*}
    \lim_{N \to \infty} \Pr(\SN(r_D(B_N)  = \l) = \Pr(\SN(r_D(A)  = \l).
\end{equation*}
From the discussion of measures at the beginning of the section it follows that $r_D(A) $ has entries iid uniform over $\Z/p^D \Z$, and $r_D(B_N) $ has the distribution of an $n \times m$ corner of a uniformly random element of $\GL_N(\Z/p^D \Z)$. A uniformly random element $(a_{ij})_{1 \leq i,j \leq N} \in \GL_N(\Z/p^D \Z)$ may be sampled by first sampling a uniformly random element of $(a'_{ij})_{1 \leq i,j \leq N} \in \GL_N(\Z/p\Z)$, then choosing $a_{ij} \in \Z/p^D\Z$ independently, uniform in the congruence class of $a'_{ij}$. Thus it suffices to show that for any $C \in M_{n \times m}(\Z/p\Z)$, 
\begin{equation*}
    \lim_{N \to \infty} \Pr(r_1(B_N) = C) = \Pr(r_1(A) = C).
\end{equation*} 
For fixed $C$, 
\begin{equation*}
    \frac{(p^{N-m}-1)(p^{N-m}-p)\cdots(p^{N-m}-p^{n-1})}{p^{nN}} \leq \Pr(r_1(B_N)  = C)  \leq \frac{p^{n(N-m)}}{p^{nN}};
\end{equation*}
the lower bound is sharp when $C$ is the zero matrix, and the upper bound is sharp when $C$ is nonsingular. Both bounds come from counting the number of $n \times N$ nonsingular matrices over $\Z/p\Z$ with left $n \times m$ submatrix $C$, and dividing by $\# M_{n \times N}(\Z/p\Z)$. The upper bound is $p^{-nm}$ and the lower bound goes to $p^{-nm}$ as $N \to \infty$. Since $\Pr(r_1(A)= C) = p^{-nm}$ for any $C$, we are done.
\end{proof}

Note also that Parts 2, 3 of \Cref{thm:exact_hl_results_in_p-adic_rmt} (and the limiting case \Cref{cor:reprove_wood}) together with \Cref{prop:product_gives_proc} immediately imply that joint distributions of singular numbers of products of $p$-adic Haar corners are distributed according to Hall-Littlewood processes.

\begin{restatable}{cor}{hlprocrmt} \label{cor:product_gives_hl_proc}
Let $t=1/p$, fix $n \geq 1$ and let $N_1,N_2,\ldots \in \Z \cup \{\infty\}$ with with $N_i > n$ for all $i$. For each $i$, let $A_i$ be the top left $n \times n$ corner of a Haar distributed element of $\GL_{N_i}(\Z_p)$ if $N_i < \infty$, and let $A_i$ have iid entries distributed by the additive Haar measure on $\Z_p$ if $N_i = \infty$. Then for $\l^{(1)},\ldots,\l^{(k)} \in \Sig_n^+$,
\begin{multline*}
    \Pr(\SN(A_\tau \cdots A_1) = \l^{(\tau)}\text{ for all }\tau=1,\ldots,k) \\ = \frac{Q_{\l^{(1)}/(0[n])}(t,\ldots,t^{N_1-n}) Q_{\l^{(2)}/\l^{(1)}}(t,\ldots,t^{N_2-n}) \cdots Q_{\l^{(k)}/\l^{(k-1)}}(t,\ldots,t^{N_k-n}) P_{\l^{(k)}}(1,\ldots,t^{n-1})}{\Pi_{(0,t)}(1,\ldots,t^{n-1}; t,\ldots,t^{N_1-n},t,\ldots,t^{N_2-n},\ldots,t,\ldots,t^{N_k-n})}
\end{multline*}
\end{restatable}

The analogue of this result in the real/complex/quaternion case is given in \cite[Thm. 3.12]{ahn2019fluctuations}.

\section{Asymptotics of Macdonald polynomials and structure coefficients}\label{sec:3.2}

We now develop the machinery to prove \Cref{prop:measure_convergence_jacobi_and_cauchy} and \Cref{prop:measure_convergence_branching_corners}. Both of these statements involve limits of normalized structure coefficients 
\begin{equation*}
    \frac{P_{\nu(D)}(a_1,\ldots,a_n;q,t)}{P_{\l(D)}(a_1,\ldots,a_n;q,t)P_{\mu(D)}(a_1,\ldots,a_n;q,t)} c_{\l(D),\mu(D)}^{\nu(D)}(q,t)
\end{equation*}
for some signatures $\l(D),\mu(D),\nu(D)$, so we must establish asymptotics both on the Macdonald polynomials with real specializations $\ba$, and on the structure coefficients $c_{\l,\mu}^\nu(q,t)$ themselves. Both come from \Cref{thm:explicit_low_deg_terms} below, which treats the asymptotics of Macdonald polynomials in formal variables $x_1,\ldots,x_N$.

\explicitlow*

\begin{proof}
First, note that in the branching rule \eqref{eq:skewP_branch_formula} there is exactly on Gelfand-Tsetlin pattern $T_{max} \in \GT_P(\l(D)/())$ with weight $\l(D)$, namely the one with all entries as large as possible. One can check that $\psi(T_{max}) = 1$, so $T_{max}$ contributes the lexicographically highest-degree monomial $\bx^{\l(D)}$ of $P_{\l(D)}(x_1,\ldots,x_N)$. 

Define the signature $\hl(D) := ((L_0 \cdot D)[r_0] ,\ldots,(L_k \cdot D)[r_k])$, so $\bx^{\hl(D)} = \prod_{i=1}^{k} (x_{s_{i-1}+1} \cdots x_{s_{i}})^{L_i D}$. The idea of the proof is that for another monomial $\bx^{\hl(D) + \bd}$, the set of GT patterns of weight $\bx^{\hl(D) + \bd}$ stabilizes in size for all large $D$, and furthermore the structure of these GT patterns will be in a sense independent of $D$.

Fix $\bd \in \Z^n$ for the remainder of the proof. For any fixed monomial $\bx^\bd$ and sufficiently large $D$, all Gelfand-Tsetlin patterns contributing to the coefficient of $\bx^\bd$ in $P_{\l(D)}(x_1,\ldots,x_N)/\bx^{\hl(D)}$ will be as in \Cref{fig:gtdecomp1}, or in other words, all entries will be close to those of $T_{max}$.

\begin{figure}[htbp]
  \centering
    \includegraphics[scale=1]{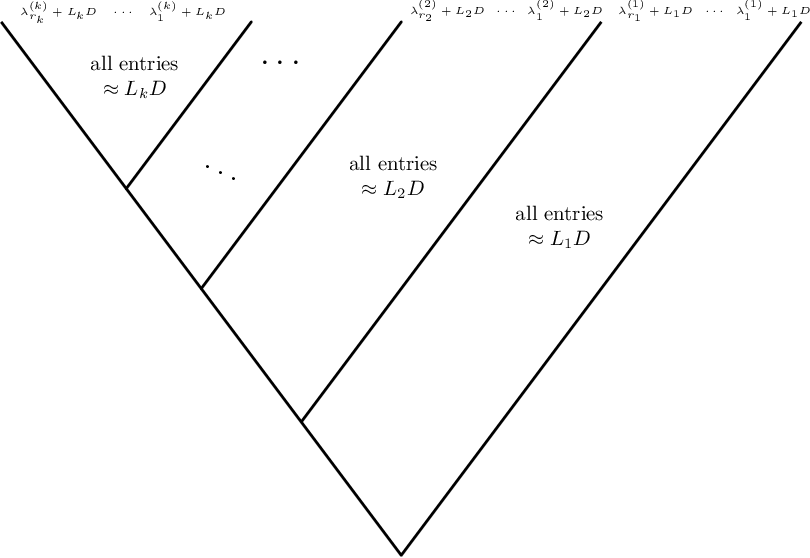}
\caption{The form of a Gelfand-Tsetlin pattern $T$ with $wt(T)$ close to $\l(D)$ for large $D$.}
\label{fig:gtdecomp1}
\end{figure}

It is natural to divide each of the `strips' of entries $\approx L_i D$ into two GT patterns, one triangular of the type in the $P$ branching rule and one rectangular of the type in the $Q$ branching rule, by splitting into the parts above and below row $s_i$ inclusive, see \Cref{fig:gtdecomp2}. Where each strip intersects the $s_i\tth$ row one has a signature $\kappa^{(i)} + L_i D$, so any $T \in \GT_P(\l(D))$ uniquely specifies smaller \emph{constituent GT patterns} $T_i^P \in \GT_P(\kappa^{(i)}/()), T_i^Q \in \GT_{Q,N-s_i}(\l^{(i)}/\kappa^{(i)})$ for each $i=1,\ldots,k$. It is also clear from the picture that any choice of these smaller GT patterns, i.e. choice of $\kappa^{(i)} \in \Sig_{r_i}, i =1,\ldots,k$ and elements of $\GT_P(\kappa^{(i)}/())$ and $\GT_{Q,s_i}(\l^{(i)}/\kappa^{(i)})$ for $i=1,\ldots,k$, uniquely specifies an element of $\GT(\l(D))$ provided $D$ is large enough that the rows are still weakly decreasing.

\begin{figure}[htbp]
  \centering
    \includegraphics[scale=1]{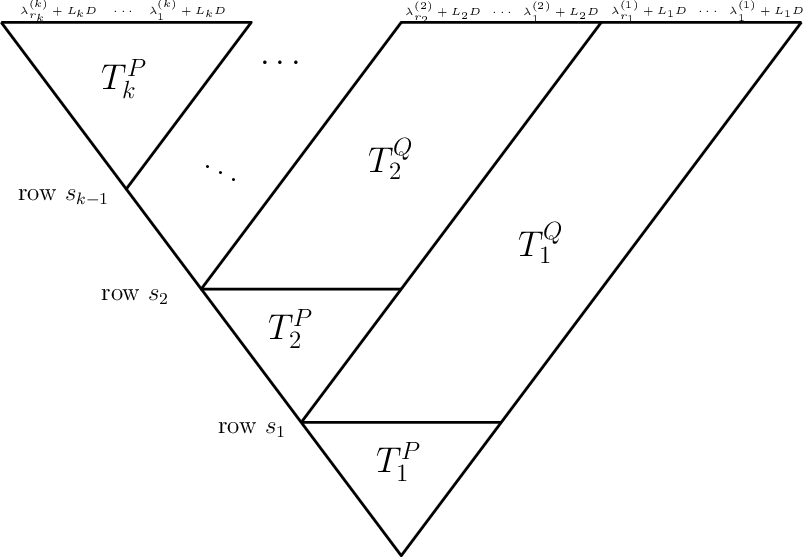}
\caption{The decomposition into constituent GT patterns.}
\label{fig:gtdecomp2}
\end{figure}

This motivates the following. For signatures $\kappa^{(1)} \in \Sig_{r_1},\ldots,\kappa^{(k)} \in \Sig_{r_k}$ and GT patterns $T_i^P \in \GT_P(\kappa^{(i)}/()), T_i^Q \in \GT_{Q,N-s_i}(\l^{(i)}/\kappa^{(i)})$ for $i=1,\ldots,k$, define (for all $D$ large enough that this makes sense) $BP_D(T_1^P,\ldots,T_k^P,T_1^Q,\ldots,T_k^Q) \in \GT_P(\l(D)/())$ to be the GT pattern of top row $\l(D)$ which decomposes into $T_1^P,\ldots,T_k^P,T_1^Q,\ldots,T_k^Q$ as above. For sufficiently large $D$, all GT patterns contributing to $\bx^{\bd}$ will be of the form $BP_D(T_1^P,\ldots,T_k^P,T_1^Q,\ldots,T_k^Q)$ for some $T_1^P,\ldots,T_k^Q$, and furthermore only finitely many such patterns will contribute to $\bx^{\bd}$ (and this finite number does not grow with $D$). Hence we may focus on describing the GT patterns $BP_D(T_1^P,\ldots,T_k^P,T_1^Q,\ldots,T_k^Q)$.

Given any $T \in \GT_{Q,s}(\mu/\nu)$ given by $\nu = \l^{(1)} \prec_Q \l^{(2)} \prec_Q \cdots \prec_Q \l^{(s)} = \mu$, define $\tT \in \GT_{Q,s}(-\nu/-\mu)$ by $-\mu  \prec -\l^{(s-1)} \prec \cdots \prec -\l^{(1)} = -\nu$. Similarly, given $T \in \GT_P(\kappa/())$ defined by $() \prec_P \l^{(1)} \prec_P \cdots \prec_P \l^{(\len(\kappa))} = \kappa$, let $\tT \in \GT_P(-\kappa/())$ be the GT pattern with $i\tth$ row $-\l^{(i)}$. We claim that 
\begin{equation}\label{eq:branch_coefs_converge}
    \lim_{D \to \infty} \psi(BP_D(T_1^P,\ldots,T_k^P,T_1^Q,\ldots,T_k^Q)) = \prod_{i=1}^k \psi(\tT_i^P)\varphi(\tT_i^Q).
\end{equation}

In the GT pattern $BP_D(T_1^P,\ldots,T_k^P,T_1^Q,\ldots,T_k^Q)$, we may view each entry as coming from one of the constituent GT patterns $T_i^P$ or $T_i^Q$. As $D \to \infty$, the difference between any two entries in a given row of $BP_D(T_1^P,\ldots,T_k^P,T_1^Q,\ldots,T_k^Q)$ which come from the same constituent GT pattern remains constant, while the difference between any two entries which come from different constituent GT patterns goes to infinity.

Recall from \eqref{eq:pbranch} that the $P$ branching coefficient $\psi(BP_D(T_1^P,\ldots,T_k^P,T_1^Q,\ldots,T_k^Q))$ is a product of factors 
\begin{equation}\label{eq:factors_in_pbranch1}
    \frac{f(t^{j-i}q^{\mu_i-\mu_j})}{f(t^{j-i}q^{\l_i-\mu_j})},
\end{equation}
and 
\begin{equation}\label{eq:factors_in_pbanch2}
    \frac{f(t^{j-i}q^{\l_i-\l_{j+1}})}{f(t^{j-i}q^{\mu_i-\l_{j+1}})}.
\end{equation}
Notice that when for example $\mu_i$ and $\mu_j$ come from different constituent GT patterns of the GT pattern $BP_D(T_1^P,\ldots,T_k^P,T_1^Q,\ldots,T_k^Q)$, $q^{\mu_i-\mu_j} \to 0$ as $D \to \infty$ since $|q|<1$, and hence $f(t^{j-i}q^{\mu_i-\mu_j}) \to f(0)=1$. Similarly the other three factors $f(t^{j-i}q^{\l_i-\l_{j+1}}),f(t^{j-i}q^{\l_i-\mu_j}),f(t^{j-i}q^{\mu_i-\l_{j+1}})$ converge to $1$. Because the number of these factors in \eqref{eq:pbranch} is finite independent of $D$, this implies that 
\begin{equation*}
    \lim_{D \to \infty} \psi(BP_D(T_1^P,\ldots,T_k^P,T_1^Q,\ldots,T_k^Q))
\end{equation*}
is equal the product of those $f(\cdots)^{\pm 1}$ factors corresponding to pairs of entries coming from the same constituent GT pattern. 

First let us consider the constituent GT patterns $T_i^P$. Because $f(t^{j-i}q^{\mu_i-\mu_j})$ depends only on the differences $j-i$ and $\mu_i-\mu_j$ but is independent of overall translation of the indices or the entries (and similarly for the other three $f$ terms), we see that the product of $f(\cdots)^{\pm 1}$ terms in $\psi(BP_D(T_0^P,\ldots,T_k^P,T_1^Q,\ldots,T_k^Q))$ corresponding to entries from a given constituent GT pattern $T_i^P$ is exactly $\psi(T_i^P)$. By the symmetry of \eqref{eq:pbranch} this is equal to $\psi(\tT_i^P)$, cf. \Cref{lem:basic_signature_func_properties}.

Now consider a constituent GT pattern $T_i^Q$. It follows from \eqref{eq:qbranch} that the product of factors \eqref{eq:factors_in_pbranch1} and \eqref{eq:factors_in_pbanch2} corresponding to pairs of entries in $T_i^Q$ is exactly $\varphi(\tT_i^Q)$, proving \eqref{eq:branch_coefs_converge}.

It is an easy check from our decomposition into constituent GT patterns that
\begin{equation}\label{eq:x_powers}
\frac{\bx^{wt(BP_D(T_1^P,\ldots,T_k^P,T_1^Q,\ldots,T_k^Q))}}{\bx^{\hl(D)}} = \prod_{i=1}^k (x_{s_{i-1}+1}^{wt(T_i^P)_1}\cdots x_{s_i}^{wt(T_i^P)_{r_i}})(x_{s_i+1}^{wt(T_i^Q)_1} \cdots x_N^{wt(T_i^Q)_{N-s_i}}). 
\end{equation}
Combining \eqref{eq:branch_coefs_converge}, \eqref{eq:x_powers}, and \Cref{lem:basic_signature_func_properties}, and summing over the $\kappa^{(i)}$, yields
\begin{multline*}
    \lim_{D \to \infty} \frac{P_{\l(D)/()}(x_1,\ldots,x_N)}{\bx^{\hl(D)}}[\bx^\bd] \\
    =\sum_{\kappa^{(i)} \in \Sig_{r_i},i=1,\ldots,k}
    \prod_{i=1}^k P_{\kappa^{(i)}/()}(x_{s_{i-1}+1},\ldots,x_{s_i}) Q_{-\kappa^{(i)}/-\l^{(i)}}(x_{s_i+1},\ldots,x_N)[\bx^\bd] \\
    = \prod_{i=1}^k P_{-\kappa^{(i)}/()}(x_{s_{i-1}+1}^{-1},\ldots,x_{s_i}^{-1}) Q_{-\kappa^{(i)}/-\l^{(i)}}(x_{s_i+1},\ldots,x_N)[\bx^\bd],
\end{multline*}
where $(\cdot)[\bx^\bd]$ denotes the coefficient of the $\bx^{\bd}$ term of the Laurent polynomial $(\cdot)$. Finally, applying the Cauchy identity \Cref{lem:asym_cauchy} to the RHS of the above completes the proof.
\end{proof}

\begin{rmk}
We have made no attempt to find the most general hypothesis on $q$ and $t$ under which the above result holds, as $q,t \in (-1,1)$ is the only range typically used in probabilistic applications. Some extra complications arise if $q,t$ are such that some of the $f(\cdots)$ factors in the denominator may vanish, or if $|q| \geq 1$--as then the argument that $f(\cdots)$ factors involving pairs of entries from different constituent GT patterns go to $1$ no longer holds. However, we believe that one may be able to prove the same result for more general values of $q,t$ with some additional analysis of cancellation between $f(\cdots)$ factors.
\end{rmk}

\begin{rmk}\label{rmk:schur_hl_stabilization}
For the Hall-Littlewood case $q=0$ or the Schur case $q=t$, the convergence statement \eqref{eq:branch_coefs_converge} is actually stabilization to equality for all sufficiently large $D$, and hence the coefficients of monomials in the statement of \Cref{thm:explicit_low_deg_terms} also stabilize for large $D$. This is because the branching coefficients $\psi(T)$ are `local' in these cases, meaning that they may be expressed as products over entries of the GT pattern rather than pairs of entries, so in particular entries of different constituent GT patterns do not interact. In the case $q=t$ this is particular clear, as $\psi(T) =1$ for any valid GT pattern $T$. 

Furthermore, at $q=t$ this stabilization is monotonic from below, i.e. for each $\mu \in \Z^N$, the coefficient of $\bx^\mu$ in the LHS of \Cref{thm:explicit_low_deg_terms} (which is an integer, as follows from the fact that the $\psi(T)$ are all $1$) is increasing in $D$ for all $D$ such that $\l(D)$ is a signature. 
\end{rmk}

\begin{rmk}
An asymptotic factorization statement somewhat similar to \Cref{thm:explicit_low_deg_terms} was proven for Jack functions in infinitely many variables by Okounkov-Olshanski \cite[Thm. 4.1]{okounkov1998asymptotics}, though we do not believe the two are directly related as our polynomials are in only finitely many variables.
\end{rmk}

We now convert \Cref{thm:explicit_low_deg_terms} into a statement about Macdonald polynomials specialized at real variables. 

\begin{defi}\label{def:proj_op}
For any finite subset $S \subset \Z^N$, let $\Proj_S: \C[x_1^{\pm 1},\ldots, x_N^{\pm 1}] \to \C[x_1^{\pm 1},\ldots, x_N^{\pm 1}]$ be the $\C$-linear operator with $\Proj_S \bx^\bd = \bbone(\bd \in S) \bx^\bd$.
\end{defi}

\begin{prop}\label{prop:spec_limits}
Let $t,q \in (-1,1)$, $\ba = (a_1,\ldots,a_N)$ with $a_1 > \ldots > a_N > 0$ be real numbers, and $L_i, r_i, s_i, \l(D), \hl(D)$ be as in \Cref{thm:explicit_low_deg_terms}. Then
\begin{multline}\label{eq:spec_limit}
    \lim_{D \to \infty} \frac{P_{\l(D)}(\ba)}{\ba^{\hl(D)}} =  \prod_{i=1}^{k} 
    P_{\l^{(i)}}(a_{s_{i-1}+1},\ldots,a_{s_i})  \prod_{i=1}^{k-1} \Pi_{(q,t)}(a_{s_{i-1}+1}^{-1},\ldots,a_{s_i}^{-1}; a_{s_i+1},\ldots,a_{N})
\end{multline}
\end{prop}
\begin{proof}

Let $R_{q,t}(\ba)$ be the RHS of \eqref{eq:spec_limit}. For any $\eps > 0$ we must find $D_0$ so that 
\begin{equation*}
    \abs*{ \frac{P_{\l(D)}(\ba)}{\ba^{\hl(D)}} - R_{q,t}(\ba) } < \eps
\end{equation*}
for all $D > D_0$. Below we will abuse notation and write $\Proj_S$ acting on a Laurent polynomial in the real numbers $a_1,\ldots,a_N$ to mean `take the corresponding polynomial in formal variables $x_i$, apply $\Proj_S$, then specialize $x_i=a_i$ for each $i$'.
Then for any $S$, we write
\begin{multline}\label{eq:3tobound}
    \abs*{ \frac{P_{\l(D)}(\ba)}{\ba^{\hl(D)}} - R_{q,t}(\ba) } \leq \abs*{ \Proj_{S}\left( \frac{P_{\l(D)}(\ba)}{\ba^{\hl(D)}} - R_{q,t}(\ba) \right)} \\ + \abs*{ (\Id - \Proj_{S})\frac{P_{\l(D)}(\ba)}{\ba^{\hl(D)}} } + \abs*{ (\Id - \Proj_{S}) R_{q,t}(\ba) },
\end{multline}

The first and third terms of the RHS are easy to bound. The first term of \eqref{eq:3tobound}, is a finite sum of $|S|$ Laurent monomials in $a_1,\ldots,a_N$ with coefficients that converge to $0$ as $D \to \infty$ by Theorem \ref{thm:explicit_low_deg_terms}. Hence for any $S$ we may choose $D_0$ so that 
\begin{equation*}
    \abs*{ \Proj_{S}\left( \frac{P_{\l(D)}(\ba)}{\ba^{\hl(D)}} - R_{q,t}(\ba) \right)} < \frac{\eps}{3}
\end{equation*}
for all $D > D_0$. For the third term, $R_{q,t}(\ba)$ is a convergent power series in the variables $a_j/a_i, j > i$, so we may choose $S$ sufficiently large that
\begin{equation}\label{eq:cauchy_tails}
    \abs*{(\Id - \Proj_{S})R_{q,t}(\ba)} < \frac{\eps}{3}.
\end{equation}

The second term is slightly trickier. Recall from \Cref{lem:branching_formulas} that $f(u) := (tu;q)_\infty/(qu;q)_\infty$. In particular, since $q,t \in (-1,1)$, $f(u)$ is defined, continuous, and nonzero on $[-1,1]$, so  
\begin{equation}\label{eq:qpocchammer_bound}
   \abs*{ \frac{f(u_1)}{f(u_2)}} \leq \frac{\sup_{u \in [-1,1]}f(u)}{\inf_{u \in [-1,1]} f(u)}.
\end{equation}
Recall that $\psi(T)$ is a finite product of factors $\frac{f(u_1)}{f(u_2)}$ for $u_1,u_2$ products of powers of $q,t$ (and in particular lying in $[-1,1]$). Hence there is a constant $C$ depending only on $N$, which is an appropriate power of the RHS of \eqref{eq:qpocchammer_bound}, such that for any $\kappa \in \Sig_N$ and $T \in \GT_P(\kappa)$ we have 
\begin{equation}\label{eq:psi_bound}
    |\psi(T)| \leq C.
\end{equation}

As in \eqref{eq:cauchy_tails} we may choose $S$ large enough so that  
\begin{equation}\label{eq:cauchy_tails_2}
    \abs*{(\Id - \Proj_{S})R_{q,q}(\ba)} < \frac{\eps}{3C}.
\end{equation}
By Remark \ref{rmk:schur_hl_stabilization} on the Schur case, for any $\bd \in \Z^d$ one has that $(P_{\l(D)}(\bx;q,q)/\bx^{\hl(D)})[\bx^\bd]$ is an increasing sequence (in $D$) of integers which stabilizes to $R_{q,q}(\bx)[\bx^\bd]$ for large $D$. It follows from this and the nonnegativity of the Laurent monomials $\ba^\bd$ that
\begin{equation}\label{eq:schur_cauchy_bound}
    (\Id - \Proj_{S})\frac{P_{\l(D)}(a_1,\ldots,a_N;q,q)}{\ba^{\hl(D)}} \leq (\Id - \Proj_{S}) R_{q,q}(\ba)
\end{equation}
(we drop absolute values because all terms in the hidden summations on each side are positive). Putting together \eqref{eq:psi_bound}, \eqref{eq:cauchy_tails_2} and \eqref{eq:schur_cauchy_bound} yields
\begin{align*}
     \abs*{ (\Id - \Proj_{S})\frac{P_{\l(D)}(\ba;q,t)}{\ba^{\hl(D)}} } 
    & \leq C \abs*{ (\Id - \Proj_{S})\frac{P_{\l(D)}(\ba;q,q)}{\ba^{\hl(D)}} } \\
    & \leq C \abs*{ (\Id - \Proj_{S}) R_{q,q}(\ba) } \\
    & < \eps/3.
\end{align*}
This handles the second term of \eqref{eq:3tobound}, so choosing $S$ large enough that all three $\eps/3$ bounds are simultaneously satisfied completes the proof.
\end{proof}

\Cref{thm:explicit_low_deg_terms} may also be used to control asymptotics of the structure coefficients $c_{\l,\mu}^\nu(q,t)$. We first prove a simple lemma. Define the \emph{dominance order} $\unlhd$ on $\Z^n$ by $\mathbf{v} \unlhd \mathbf{w}$ if $\sum_i v_i = \sum_i w_i$ and 
\begin{equation*}
    \sum_{i=1}^j v_i \leq \sum_{i=1}^j w_i \quad \text{ for }j=1,\ldots,n.
\end{equation*}

\begin{lemma}\label{lem:projections_suffice}
Let $\l(D), \mu(D), \kappa(D) \in \Sig_N$ be three sequences of signatures of the form in \Cref{thm:explicit_low_deg_terms} (possibly for different $L_i, r_i$), such that $\hl(D) + \hm(D) = \hk(D)$. Let $\tl \in \Z^N$ be the tuple such that $\l(D) = \hl(D)+\tl(D)$, and define $\tm, \tk$ similarly. Let $S$ denote the (finite) interval $[\tk,\tl+\tm]$ in $\Z^N$ with respect to the dominance order, or explicitly $S = \{\bd \in \Z^N: \tk \unlhd \bd \unlhd \tl+\tm\}$.  
Then
\begin{enumerate}[label={\arabic*.}]
    \item The set $\{\lim_{D \to \infty} \Proj_S P_{\hk(D)+\bd}(\bx)/\bx^{\hk(D)} : \bd \in S\}$ is a basis for $\Proj_S \C[x_1^{\pm},\ldots,x_N^{\pm}]$. Here the limits of Laurent polynomials are in the sense of convergence of coefficients of each monomial.
    \item The coefficient of $\lim_{D \to \infty} \Proj_S P_{\kappa(D)}/\bx^{\hk(D)}$ in the decomposition of 
    \begin{equation*}
        \lim_{D \to \infty} \Proj_S \frac{P_{\l(D)}(\bx)P_{\mu(D)}(\bx)}{\bx^{\hk(D)}}
    \end{equation*}
    in the above basis is $\lim_{D \to \infty} c^{\kappa(D)}_{\l(D),\mu(D)}(q,t)$.
\end{enumerate}
\end{lemma}
\begin{proof}
The first part follows since the Macdonald polynomials are homogeneous and may be written as $P_\l(\bx) = \bx^\l + (\text{lower order terms in the dominance order})$. The second part is then clear.
\end{proof}

\begin{defi}
Let $\l,\nu \in \Sig_n, \mu \in \Sig_r$. Define $d_{\l,\mu}^\nu(q,t)$ by\footnote{These are related by duality to the structure coefficients $c_{\l,\mu}^\nu$, see \cite[Ch. VI]{mac}, but we will not elaborate on this because we do not need it and due to our conventions with signatures it would be somewhat cumbersome to state.}
\begin{equation*}
    Q_{\nu/\l}(x_1,\ldots,x_r) = \sum_{\mu \in \Sig_r} d_{\l,\mu}^\nu(r) P_\mu(x_1,\ldots,x_r).
\end{equation*}
\end{defi}

\begin{prop}\label{prop:bet2}
Let $n \leq m \leq N$ such that $n \leq N-m$, let $\l,\nu \in \Sig_n$, and let $q,t \in (-1,1)$. Then
\begin{equation}\label{eq:bet2_1}
    \lim_{D \to \infty} c_{(D[N-n],\l),(D[N-m],0[m])}^{(2D[N-m-n],2D-\eta_n,\ldots,2D-\eta_1,D[m-n],\nu_1,\ldots,\nu_n)}(q,t) 
    = d_{\l,\eta}^\nu(q,t)
\end{equation}
and 
\begin{equation}\label{eq:bet2_2}
    \lim_{D \to \infty} c_{(D[N-n],0[n]),(D[N-m],0[m])}^{(2D[N-m]+\alpha,D[m-n]+\beta,\nu_1,\ldots,\nu_n)}(q,t) = 0
\end{equation}
for all $\alpha \in \Sig_{N-m}, \beta \in \Sig_{m-n}$ not as in \eqref{eq:bet2_1}. 
\end{prop}

\begin{proof}
Let $\l(D) = (D[N-n],\l), \mu(D) = (D[N-m],0[m]),$ and $\kappa(D;\alpha,\beta,\nu) = (2D[N-m]+\alpha,D[m-n]+\beta,\nu_1,\ldots,\nu_n)$, which we will write as simply $\kappa(D)$ when the other signatures are clear from context. Fix $\alpha,\beta,\nu$, and let $S$ be as in \Cref{lem:projections_suffice}. By \Cref{thm:explicit_low_deg_terms} we have
\begin{align*}
    P_{\l(D)}(\bx)/\bx^{\hl(D)} &\to P_\l(x_{N-n+1},\ldots,x_N) \Pi(x_1^{-1},\ldots,x_{N-n}^{-1}; x_{N-n+1},\ldots,x_N) \\
    P_{\mu(D)}(\bx)/\bx^{\hm(D)} &\to \Pi(x_1^{-1},\ldots,x_{N-m}^{-1}; x_{N-m+1},\ldots,x_N).
\end{align*}
Splitting the latter Cauchy kernel,
\begin{multline*}
    \lim_{D \to \infty} \Proj_S \frac{P_{\l(D)}(\bx)P_{\mu(D)}(\bx)}{\bx^{\hl(D)+\hm(D)}}  \\
    = \Proj_S  \Pi(x_1^{-1},\ldots,x_{N-n}^{-1}; x_{N-n+1},\ldots,x_N) \Pi(x_1^{-1},\ldots,x_{N-m}^{-1}; x_{N-m+1},\ldots,x_{N-n}) \\ \cdot \left(\Pi(x_1^{-1},\ldots,x_{N-m}^{-1} ; x_{N-n+1},\ldots,x_N) P_\l(x_{N-n+1},\ldots,x_N) \right)
\end{multline*}
Applying \Cref{lem:asym_cauchy}, the definition of the coefficients $d_{\l,\eta}^\nu$, and \Cref{lem:basic_signature_func_properties}, one has
\begin{align*}
    &\Pi(x_1^{-1},\ldots,x_{N-m}^{-1} ; x_{N-n+1},\ldots,x_N) P_\l(x_{N-n+1},\ldots,x_N) \\
    &= \sum_{\nu \in \Sig_n} P_{\nu}(x_{N-n+1},\ldots,x_N)Q_{\nu/\l}(x_1^{-1},\ldots,x_{N-m}^{-1}) \\
    &=  \sum_{\substack{\nu \in \Sig_n \\ \eta \in \Sig_{N-m}}} d_{\l,\eta}^\nu P_{\nu}(x_{N-n+1},\ldots,x_N)P_\eta(x_1^{-1},\ldots,x_{N-m}^{-1}) \\
    &= \sum_{\substack{\nu \in \Sig_n \\ \eta \in \Sig_{N-m}}} d_{\l,\eta}^\nu P_{\nu}(x_{N-n+1},\ldots,x_N)P_{-\eta}(x_1,\ldots,x_{N-m}).
\end{align*}
By straightforward application of \Cref{thm:explicit_low_deg_terms},
\begin{multline}\label{eq:kappa_limit}
    \lim_{D \to \infty} \Proj_S \frac{P_{\kappa(D)}(\bx)}{\bx^{\hk(D)}} = \Proj_S P_\alpha(x_1,\ldots,x_{N-m})P_\beta(x_{N-m+1},\ldots,x_{N-n})P_\nu(x_{N-n+1},\ldots,x_N)\\
    \cdot \Pi(x_1^{-1},\ldots,x_{N-m}^{-1};x_{N-m+1},\ldots,x_N)\Pi(x_{N-m+1}^{-1},\ldots,x_{N-n}^{-1};x_{N-n+1},\ldots,x_N) \\
    = \Proj_S P_\alpha(x_1,\ldots,x_{N-m})P_\beta(x_{N-m+1},\ldots,x_{N-n})P_\nu(x_{N-n+1},\ldots,x_N) \\
    \cdot \Pi(x_1^{-1},\ldots,x_{N-n}^{-1}; x_{N-n+1},\ldots,x_N) \Pi(x_1^{-1},\ldots,x_{N-m}^{-1}; x_{N-m+1},\ldots,x_{N-n})
\end{multline}
We see that 
\begin{multline*}
    \lim_{D \to \infty} \Proj_S \frac{P_{\l(D)}(\bx)P_{\mu(D)}(\bx)}{\bx^{\hl(D)+\hm(D)}} 
    = \Proj_S  \Pi(x_1^{-1},\ldots,x_{N-n}^{-1}; x_{N-n+1},\ldots,x_N) \\
    \cdot \Pi(x_1^{-1},\ldots,x_{N-m}^{-1}; x_{N-m+1},\ldots,x_{N-n}) \sum_{\substack{\nu \in \Sig_n \\ \eta \in \Sig_{N-m}}} d_{\l,\eta}^\nu P_{\nu}(x_{N-n+1},\ldots,x_N)P_{-\eta}(x_1,\ldots,x_{N-m})
\end{multline*}
is a finite (because $S$ is finite) sum of terms of the form RHS\eqref{eq:kappa_limit} for those $\kappa(D;\alpha,\beta,\nu)$ for which $\alpha = -\eta, \beta = (0[m-n])$. Furthermore, the coefficients of these terms are $d_{\l,\eta}^\nu$. Applying \Cref{lem:projections_suffice} completes the proof.
\end{proof}

We now combine these results.

\begin{proof}[Proof of \Cref{prop:measure_convergence_jacobi_and_cauchy}]
Let $\l(D),\mu(D),\kappa(D;\alpha,\beta,\nu)$ be as in the previous proof, and $S \subset \Sig_{N-m}$ be finite. Since we have assumed $q,t$ are such that the structure coefficients are nonnegative,
\begin{equation}\label{eq:struc_coef_bound}
    \mcj(\nu) \geq \sum_{\eta \in S} c_{\l(D),\mu(D)}^{\kappa(D;-\eta,(0[m-n]),\nu)} \frac{P_{\kappa(D;-\eta,(0[m-n]),\nu)}(\ba)}{P_{\l(D)}(\ba)P_{\mu(D)}(\ba)}
\end{equation}
by simply taking only finitely many of the terms in the sum used to define $\mcj$. Note that $\hk = \hl+\hm$, so we may divide numerator by $\ba^\hk$ and denominator by $\ba^\hl \cdot \ba^\hm$ to obtain the LHS of \Cref{prop:spec_limits}. Applying \Cref{prop:bet2} and \Cref{prop:spec_limits} to the structure coefficients and specialized Macdonald polynomials on the RHS of \eqref{eq:struc_coef_bound}, respectively, we obtain
\begin{align*}
    \lim_{D \to \infty} \mcj(\nu) &\geq \sum_{\eta \in S} \lim_{D \to \infty}  c_{\l(D),\mu(D)}^{\kappa(D;-\eta,(0[m-n]),\nu)} \frac{P_{\kappa(D;-\eta,(0[m-n]),\nu)}(\ba)/\ba^{\hk}}{P_{\l(D)}(\ba)/\ba^\hl \cdot P_{\mu(D)}(\ba)/\ba^{\hm}} \\
    &= \sum_{\eta \in S} d_{\l,\eta}^\nu \frac{P_{-\eta}(a_1,\ldots,a_{N-m})P_\nu(a_{N-n+1},\ldots,a_N)}{P_\l(a_{N-n+1},\ldots,a_N) \Pi(a_1^{-1},\ldots,a_{N-m}^{-1}; a_{N-n+1},\ldots,a_N)}.
\end{align*}
Because the bound holds for any finite $S$, we may replace $S$ by $\Sig_{N-m}$ in the above, and also replace $P_{-\eta}(a_1,\ldots,a_{N-m})$ by $P_\eta(a_1^{-1},\ldots,a_{N-m}^{-1})$ by \Cref{lem:basic_signature_func_properties}. Then the above becomes
\begin{multline*}
    \sum_{\eta \in \Sig_{N-m}} d_{\l,\eta}^\nu \frac{P_{\eta}(a_1^{-1},\ldots,a_{N-m}^{-1})P_\nu(a_{N-n+1},\ldots,a_N)}{P_\l(a_{N-n+1},\ldots,a_N) \Pi(a_1^{-1},\ldots,a_{N-m}^{-1}; a_{N-n+1},\ldots,a_N)} \\ = \frac{Q_{\nu/\l}(a_1^{-1},\ldots,a_{N-m}^{-1})P_\nu(a_{N-n+1},\ldots,a_N)}{P_\l(a_{N-n+1},\ldots,a_N) \Pi(a_1^{-1},\ldots,a_{N-m}^{-1}; a_{N-n+1},\ldots,a_N)}
\end{multline*}
by definition of the coefficients $d_{\l,\eta}^\nu$. Because 
\begin{equation*}
    \sum_{\nu \in \Sig_n} \frac{Q_{\nu/\l}(a_1^{-1},\ldots,a_{N-m}^{-1})P_\nu(a_{N-n+1},\ldots,a_N)}{P_\l(a_{N-n+1},\ldots,a_N) \Pi(a_1^{-1},\ldots,a_{N-m}^{-1}; a_{N-n+1},\ldots,a_N)} = 1
\end{equation*}
by \Cref{lem:asym_cauchy}, the inequalities 
\begin{equation*}
    \lim_{D \to \infty} \mcj(\nu)  \geq \frac{Q_{\nu/\l}(a_1^{-1},\ldots,a_{N-m}^{-1})P_\nu(a_{N-n+1},\ldots,a_N)}{P_\l(a_{N-n+1},\ldots,a_N) \Pi(a_1^{-1},\ldots,a_{N-m}^{-1}; a_{N-n+1},\ldots,a_N)} 
\end{equation*}
must all be equalities, completing the proof.
\end{proof}

\begin{proof}[Proof of \Cref{prop:measure_convergence_branching_corners}]
The proof is very similar to that of \Cref{prop:measure_convergence_jacobi_and_cauchy}, so we will go through the argument but neglect some of the analytic details which are the same as before. Proceeding as in \Cref{prop:bet2}, we compute the limiting structure coefficients
\begin{equation*}
    \lim_{D \to \infty} c_{\l,(D[k],0[n-k])}^{(D[k]+\eta, \mu)}
\end{equation*}
for $\mu \in \Sig_{N-k}$. Define $\tc_{\mu,\eta}^\l$ by 
\begin{equation*}
    P_{\l/\mu}(x_1,\ldots,x_k) = \sum_{\eta \in \Sig_k} \tc_{\mu,\eta}^\l P_\eta(x_1,\ldots,x_k) 
\end{equation*}
(these are the dual Littlewood-Richardson coefficients and are related to the usual $c_{\mu,\eta}^\l$, see \cite[Ch. VI]{mac}, though we will not need this).
We have 
\begin{align*}
    P_\l(\bx) \frac{P_{\mu(D)}(\bx)}{\bx^{(D[k],0[n-k])}} &\to P_\l(\bx) \Pi(x_1^{-1},\ldots,x_k^{-1}; x_{k+1},\ldots,x_n) \\
    &= \sum_{\mu \in \Sig_{n-k}} \tc_{\mu,\eta}^\l P_\eta(x_1,\ldots,x_)P_\mu(x_{k+1},\ldots,x_n) \Pi(x_1^{-1},\ldots,x_k^{-1}; x_{k+1},\ldots,x_n).
\end{align*}
Likewise we have
\begin{equation*}
    \frac{P_{(D[k]-\mu, \mu)}(\bx)}{\bx^{(D[k],0[n-k])}} \to P_{\eta}(x_1,\ldots,x_k)P_\mu(x_{k+1},\ldots,x_n)\Pi(x_1^{-1},\ldots,x_k^{-1}; x_{k+1},\ldots,x_n).
\end{equation*}
Hence by the same argument as before, 
\begin{equation*}
    \lim_{D \to \infty} c_{\l,(D[k],0[n-k])}^{(D[k]+\eta, \mu)} = \tc_{\mu,\eta}^\l.
\end{equation*}
We thus have, again using positivity of all of the structure coefficients and specialized Macdonald polynomials, that
\begin{align*}
    \lim_{D \to \infty} \mbr(\mu) &\geq \sum_{\eta \in \Sig_k} \lim_{D \to \infty} c_{\l,(D[k],0[n-k])}^{(D[k]+\eta, \mu)} \frac{P_{(D[k]+\eta, \mu)}(\ba)/\ba^{(D[k],0[N-k])}}{P_\l(\ba) P_{(D[k],0[n-k])}(\ba)/\ba^{(D[k],0[N-k])}} \\
    &= \sum_{\eta \in \Sig_k} \tc_{\mu,\eta}^\l \frac{P_{\eta}(a_1,\ldots,a_k)P_\mu(a_{k+1},\ldots,a_n)}{P_\l(a_1,\ldots,a_n)}  \\
    &= \frac{P_{\l/\mu}(a_1,\ldots,a_k)P_\mu(a_{k+1},\ldots,a_n)}{P_\l(a_1,\ldots,a_n)}.
\end{align*}
These sum to $1$ by the branching rule, so the inequalities are equalities, completing the proof.
\end{proof}

\section{An exact sampling algorithm, and the LLN and CLT for the product process}\label{sec:4}

The plan of this section is as follows. In \Cref{subsec:rmt_lln_clt} we state a general law of large numbers and functional central limit theorem for Hall-Littlewood processes, \Cref{thm:general_hl_lln_clt}, and deduce the LLN and CLT for matrix products \Cref{thm:lln_and_func_clt_rmt_version} from it. We spend the remainder of the section proving \Cref{thm:general_hl_lln_clt}. In \Cref{subsec:sampling} we introduce the random sampling algorithm for Hall-Littlewood processes with one principal specialization $1,t,\ldots,t^{n-1}$ by a PushTASEP-like particle system. In \Cref{subsec:non-interacting_intro} we introduce a simpler variant of this particle system which is easier to analyze asymptotically, and show that the two may be coupled with small error. In \Cref{subsec:non-interacting_analysis} we complete the proof by analyzing this particle system. In \Cref{subsec:lyapunov} we prove the universality of Lyapunov exponents stated earlier as \Cref{thm:lyapunov}.

\subsection{Asymptotics of products of random matrices.}\label{subsec:rmt_lln_clt} 

Recall the main result.

\llncltrmt*

In view of \Cref{cor:product_gives_hl_proc}, this is a special case of the result below.

\begin{thm}\label{thm:general_hl_lln_clt}
Fix the Hall-Littlewood parameter $t \in (0,1)$, and $n \in \Z_{>0}$. Let $x_1,x_2,\ldots \in (\delta,1-\delta)$ for some $\delta>0$, and let $\hx_i = (x_i,tx_i,\ldots,t^{m_i-1}x_i)$ be collections of variables in $t$-geometric progression, possibly infinite, for each $i$. Let $(\l(1),\l(2),\ldots)$ be an infinite sequence of random signatures whose marginals are given by a Hall-Littlewood process,
\begin{equation*}
    \Pr(\l(1) = \l^{(1)},\ldots,\l(N) = \l^{(N)}) = \frac{Q_{\l^{(N)}/\l^{(N-1)}}(\hx_N) \cdots Q_{\l^{(2)}/\l^{(1)}}(\hx_2)Q_{\l^{(1)}}(\hx_1) P_{\l_N}(1,\ldots,t^{n-1})}{\Pi(1,\ldots,t^{n-1}; \hx_1,\ldots,\hx_N)}.
\end{equation*}

Then we have the following strong law of large numbers. For each $i=1,\ldots,n$,
\begin{equation}\label{eq:lln}
    \frac{\l_i(k)}{\sum_{j=1}^k \sum_{\ell=0}^{m_j-1} \frac{t^{i+\ell-1}x_j(1-t)}{(1-t^{i+\ell}x_j)(1-t^{i+\ell-1}x_j)}} \to 1 \text{       a.s.  as $k \to \infty$}.
\end{equation}
We also have the following functional central limit theorem. Let 
\begin{equation*}
\bl_i(k) := \l_i(k) - \sum_{j=1}^k \sum_{\ell=0}^{m_j-1} \frac{t^{i+\ell-1}x_j(1-t)}{(1-t^{i+\ell}x_j)(1-t^{i+\ell-1}x_j)}.    
\end{equation*} 
Let $f_{\bl_i,k}$ be the random element of $C[0,1]$ defined as follows: set $f(0)=0$ and
\begin{equation*}
        (f_{\bl_i,k}(1/k),f_{\bl_i,k}(2/k),\ldots,f_{\bl_i,k}(1)) = \frac{1}{\sqrt{\sum_{j=1}^k \sum_{\ell=0}^{m_j-1} \frac{t^{i+\ell-1}x_j(1-t)(1-t^{2i+2\ell-1}x_j^2)}{(1-t^{i+\ell-1}x_j)^2(1-t^{i+\ell}x_j)^2}}}(\bl_i(1),\ldots,\bl_i(k)),
\end{equation*}
then linearly interpolate from these values on each interval $[\ell/k,(\ell+1)/k]$. Then as $k \to \infty$, the $n$-tuple of random functions $(f_{\bl_1,k},\ldots,f_{\bl_n,k})$ converges in law in the sup norm topology to $n$ independent standard Brownian motions.
\end{thm}

\begin{rmk}
Though we have avoided it for the sake of simplicity, it is possible to define the product process more generally, allowing for nonsquare matrices. In the usual archimedean case this is done in \cite[Appendix A]{ahn2019fluctuations}, and the $p$-adic case is exactly the same.
\end{rmk}

\subsection{Sampling algorithm for Hall-Littlewood processes with one principal specialization.}\label{subsec:sampling}

In what follows, we will identify signatures $\l \in \Sig_n$ with configurations of $n$ particles on $\Z$ by placing $m_i(\l)$ particles at each position $i \in \Z$. Each particle corresponds to a part of $\l$, and we will refer to them as the $1^{\text{st}},\ldots,n\tth$ particle or `particle $1,\ldots,$ particle $n$' to reflect this, even when some are in the same location. In this numbering, particle $j$ will correspond to a particle at position $\l_j$. 

\begin{defi}\label{def:interacting_insertion}
Define the `insertion map' $\iota: \Z_{\geq 0}^n \times \Sig_n \to \Sig_n$ by defining $\iota(a_1,\ldots,a_n;\l)$ as follows. First assign to each particle $j$ an `impulse' $a_j$. Particle $n$ then moves to the right until it has either moved $a_n$ steps or encountered particle $n-1$. If it encounters particle $n-1$, then it is `blocked' by particle $n-1$ and donates the remainder $a_n - (\l_{n-1}-\l_n)$ of its impulse to particle $n-1$. Particle $n-1$ now has impulse $a_{n-1} + \max(0, a_n - (\l_{n-1}-\l_n))$, and moves in the same manner, possibly donating some of its impulse to particle $n-2$; all further particle evolve in the same manner.
\end{defi}

\begin{example}\label{ex:particle_evolution}
To compute $\iota(1,4,2; (5,3,-1)) = (8,5,1)$ the particles jump as above. The numbers above the particles represent their impulses; note that impulse-donation from particle $2$ to particle $1$ occurs at the third step shown.
\vspace{3ex}
\begin{center}
\begin{tikzpicture}[
dot/.style = {circle, fill, minimum size=#1,
              inner sep=0pt, outer sep=0pt},
dot/.default = 6pt  % size of the circle diameter 
                    ] %dot commands taken from https://tex.stackexchange.com/questions/445946/how-set-tikz-circle-radius-in-nodecircle
\def\h{0}
\draw[latex-latex] (-3.5,\h) -- (9.5,\h) ; %edit here for the axis
\foreach \x in  {-3,-2,-1,0,1,2,3,4,5,6,7,8,9} % edit here for the vertical lines
\draw[shift={(\x,\h)},color=black] (0pt,3pt) -- (0pt,-3pt);
\foreach \x in {-3,-2,-1,0,1,2,3,4,5,6,7,8,9} % edit here for the numbers
\draw[shift={(\x,\h)},color=black] (0pt,0pt) -- (0pt,-3pt) node[below] 
{$\x$};
\node[dot,label=above:$2$] at (-1,\h) {};
\node[dot,label=above:$4$] at (3,\h) {};
\node[dot,label=above:$1$] at (5,\h) {};

\def\h{-2}
\draw[latex-latex] (-3.5,\h) -- (9.5,\h) ; %edit here for the axis
\foreach \x in  {-3,-2,-1,0,1,2,3,4,5,6,7,8,9} % edit here for the vertical lines
\draw[shift={(\x,\h)},color=black] (0pt,3pt) -- (0pt,-3pt);
\foreach \x in {-3,-2,-1,0,1,2,3,4,5,6,7,8,9} % edit here for the numbers
\draw[shift={(\x,\h)},color=black] (0pt,0pt) -- (0pt,-3pt) node[below] 
{$\x$};
\node[dot,label=above:$0$] at (1,\h) {};
\node[dot,label=above:$4$] at (3,\h) {};
\node[dot,label=above:$1$] at (5,\h) {};
\draw [->] (-1,\h) to [out=30,in=150] ((1-.15,\h+.15);

\def\h{-4}
\draw[latex-latex] (-3.5,\h) -- (9.5,\h) ; %edit here for the axis
\foreach \x in  {-3,-2,-1,0,1,2,3,4,5,6,7,8,9} % edit here for the vertical lines
\draw[shift={(\x,\h)},color=black] (0pt,3pt) -- (0pt,-3pt);
\foreach \x in {-3,-2,-1,0,1,2,3,4,5,6,7,8,9} % edit here for the numbers
\draw[shift={(\x,\h)},color=black] (0pt,0pt) -- (0pt,-3pt) node[below] 
{$\x$};
\node[dot,label=above:$0$] at (1,\h) {};
\node[dot,label=above:$0$] at (5-.25,\h) {};
\node[dot,label=above:$3$] at (5,\h) {};
\draw [->] (3,\h) to [out=30,in=150] ((5-.15-.25,\h+.15);

\def\h{-6}
\draw[latex-latex] (-3.5,\h) -- (9.5,\h) ; %edit here for the axis
\foreach \x in  {-3,-2,-1,0,1,2,3,4,5,6,7,8,9} % edit here for the vertical lines
\draw[shift={(\x,\h)},color=black] (0pt,3pt) -- (0pt,-3pt);
\foreach \x in {-3,-2,-1,0,1,2,3,4,5,6,7,8,9} % edit here for the numbers
\draw[shift={(\x,\h)},color=black] (0pt,0pt) -- (0pt,-3pt) node[below] 
{$\x$};
\node[dot,label=above:$0$] at (1,\h) {};
\node[dot,label=above:$0$] at (5,\h) {};
\node[dot,label=above:$0$] at (8,\h) {};
\draw [->] (5+.25,\h) to [out=30,in=150] ((8-.15,\h+.15);
\end{tikzpicture}
\end{center}
\vspace{3ex}
\end{example}

It is obvious from \Cref{def:interacting_insertion} that $\l \prec_Q \iota(a_1,\ldots,a_n; \l)$ for any $\ba \in \Z_{\geq 0}^n$. It is also not hard to check by induction on $i$ that one may equivalently define $\iota$ by defining the $(n-i)\tth$ part
\begin{equation}\label{eq:iota_uniform_def}
    \iota(a_1,\ldots,a_n; \l)_{n-i} = \min(\l_{n-i-1}, \max(\l_{n-i}+a_{n-i},\l_{n-i+1}+a_{n-i}+a_{n-i+1},\ldots,\l_n+a_{n-i}+\ldots+a_n))
\end{equation}
for each $i=0,\ldots,n-1$, where we formally take $\l_0 = \infty$ in the edge case $i=n-1$. 

We now use the insertion $\iota$ with random input $a_1,\ldots,a_n$ to define random signatures, which we will show in Proposition \ref{prop:sampling_alg} yields the `Cauchy' Markov transition dynamics of \Cref{def:mac_cauchy_dynamics}. First we define the measures which will be the distributions of the $a_i$.

\begin{defi}\label{def:geom_diff_rvs}
Let $G_x$ be the measure on $\Z_{\geq 0}$ which is the distribution of $\max(X-T,0)$ where $X \sim \Geom(x), T \sim \Geom(t)$. Explicitly,
\begin{equation}\label{eq:trunc_geom_formula}
    G_x(\ell) = \frac{1-x}{1-t x}(1-t)^{\bbone(\ell>0)}x^\ell.
\end{equation}
Equivalently $G_x$ is defined by the generating function 
\begin{equation}\label{eq:trunc_geom_pgf}
    \sum_{\ell \geq 0} G_x(\ell) z^\ell = \frac{1-x}{1-t x} \frac{1-t x z}{1-x z} = \frac{\Pi_{(0,t)}(z;x)}{\Pi_{(0,t)}(1;x)}.
\end{equation}
\end{defi}

\begin{prop}\label{prop:sampling_alg}
For $0<x<1$, let $X_1,\ldots,X_n$ be independent with $X_i \sim G_{xt^{i-1}}$. Let $\l,\nu \in \Sig_n$ with $\l \prec_Q \nu$. Then
\begin{align}\label{eq:sampling_alg}
    \Pr(\iota(X_1,\ldots,X_n;\l) = \nu) &= \frac{1-x}{1-t^nx} \prod_{j: m_j(\l)=m_j(\nu)+1}(1-t^{m_j(\l)}) \prod_{i=1}^n (xt^{i-1})^{\nu_i-\l_i} \\
    &=  \frac{Q_{\nu/\l}(x)P_\nu(1,\ldots,t^{n-1})}{P_\l(1,\ldots,t^{n-1}) \Pi_{(0,t)}(x;1,\ldots,t^{n-1})}.\label{eq:sampling_alg2}
\end{align}
\end{prop}

\begin{proof}
We let $\Pr_x(\l \to \nu) := \Pr(\iota(X_1,\ldots,X_n;\l) = \nu)$.
The equality of the RHS of \eqref{eq:sampling_alg} with \eqref{eq:sampling_alg2} follows by \Cref{prop:hl_principal_formulas} while the first requires proof. We will explicitly compute $\Pr_x(\l \to \nu)$ from the definition of $\iota$.

Let $\l = (a_1[k_1],\ldots,a_r[k_r])$, where the $a_i$ are distinct, $k_i$ are integers $\geq 1$ with $\sum_i k_i = n$. To avoid cumbersome notation for edge cases, we formally take $\l_0 = a_0 = \infty$ in some formulas below.

It is clear from the definition that $\Pr_x(\l \to \nu)$ is nonzero only if $\l \prec \nu$. By interlacing, only the rightmost particle in the group of $k_i$ particles at location $a_i$ can exit to the right; the location where it stops is $\iota(X_1,\ldots,X_n;\l)_{n-(k_r+\ldots+k_i)+1}$.

Let us define random variables $N_i, 1 \leq i \leq r$, to be the location of the particle that jumps out of the $i\tth$ clump after its jump (if no particle leaves the clump, then $N_i = a_i$. Explicitly, $N_i = \iota(X_1,\ldots,X_n;\l)_{n-(k_r+\ldots+k_i)+1}$, and so
\begin{equation*}
    \Pr(N_1 = \nu_1,N_2 = \nu_{k_1+1},\ldots,N_r = \nu_{n-k_r+1}) = \Pr_x(\l \to \nu)
\end{equation*}
for any $\nu \succ_Q \l$. We will explicitly compute the joint distribution of the $N_i$, starting with the distribution of $N_r$.

By \Cref{def:interacting_insertion}, $N_r$ has distribution
\begin{equation*}
    \min(a_{r-1}, a_r + X_{n-k_r+1} + \ldots + X_n).
\end{equation*}
Using the probability generating function \eqref{eq:trunc_geom_pgf}, we have that 
\begin{align}
    \Pr(X_{n-k_r+1} + \ldots + X_n = \ell) &= \left(\prod_{i=n-k_r+1}^n\frac{1-t^{i-1}x}{1-t^i x} \frac{1-t^i x z}{1-t^{i-1}x z}\right)[z^\ell] \\
    &=\left(\frac{1-t^{n-k_r}x}{1-t^n x} \frac{1-t^n x z}{1-t^{n-k_r}x z}\right)[z^\ell].
\end{align}
Expanding this out, we have
\begin{equation}\label{eq:leftnudist}
    \Pr(N_r = a_r+\ell) = \begin{cases}
    \frac{1-t^{n-k_r}x}{1-t^n x} & \ell = 0 \\
    \frac{1-t^{n-k_r}x}{1-t^n x}(1-t^{k_r})(xt^{n-k_r})^{\ell} & 0 < \ell < a_{r-1}-a_r \\
    \frac{1}{1-t^n x}(1-t^{k_r}) (xt^{n-k_r})^\ell & \ell = a_{r-1}-a_r
    \end{cases}.
\end{equation}
Note that this formula still makes sense when $r=1$, as the last case $\ell = \infty$ has probability $0$.

Now let us find the distribution of $N_{r-1}$. Its distribution, conditional on $N_r$, depends on whether $N_r < a_{r-1}$ or $N_r=a_{r-1}$. 

\textbf{Case I: $N_r < a_{r-1}$.}

In this case, we may compute the conditional distribution of $N_r$ exactly as before, obtaining
\begin{equation}\label{eq:case1nudist}
    \Pr(N_{r-1} = a_{r-1}+\ell) = \begin{cases}
    \frac{1-t^{n-k_r-k_{r-1}}x}{1-t^{n-k_r} x} & \ell = 0 \\
    \frac{1-t^{n-k_r-k_{r-1}}x}{1-t^{n-k_r} x}(1-t^{k_{r-1}})(xt^{n-k_r-k_{r-1}})^{\ell} & 0 < \ell < a_{r-2}-a_{r-1} \\
    \frac{1}{1-t^{n-k_r} x}(1-t^{k_{r-1}}) (xt^{n-k_r-k_{r-1}})^\ell & \ell = a_{r-2}-a_{r-1}
    \end{cases}.
\end{equation}

\textbf{Case II: $N_r = a_{r-1}$.}

In this case, the computation is different: Because $N_r$ may donate some of its jump, Definition \ref{def:interacting_insertion} yields that $N_{r-1}$ has distribution
\begin{equation*}
    \min(a_{r-2}, a_{r-1} + Y + X_{n-k_r-k_{r-1}+1} + \ldots + X_{n-k_r}),
\end{equation*}
where $Y$ comes from the possible jump-donation of $N_r$ and has distribution given by
\begin{equation}\label{eq:defY}
    \Pr(Y=\ell) = \Pr((X_{n-k_r+1} + \ldots + X_n) - (a_{r-1}-a_r) = \ell | X_{n-k_r+1} + \ldots + X_n  \geq a_{r-1}-a_r).
\end{equation}
This looks overly complicated, but let us back up and see what it all means. As we noted before, $X_{n-k_r+1} + \ldots + X_n$ has probability generating function
\begin{equation*}
    \frac{1-t^{n-k_r}x}{1-t^n x} \frac{1-t^n x z}{1-t^{n-k_r}x z},
\end{equation*}
hence
\begin{equation*}
    X_{n-k_r+1} + \ldots + X_n \sim \max(\Geom(t^{n-k_r}x) - \Geom(t^{k_r}), 0).
\end{equation*}
How, in general, would one sample $Z \sim \Geom(x)-\Geom(w)$? A simple way is to take two coin with probability $x$ and $w$ of heads respectively, and keep flipping them until one comes up tails, then see how many additional flips it takes before the other comes up tails--call this (random) number $\ell$. If the $w$-coin came up tails first, then $Z=\ell$; if the $x$-coin came up tails first, $Z=-\ell$. From this description it is clear that if we condition on $Z \geq 1$, or indeed $Z \geq c$ for any $c \geq 1$, we are conditioning on the event that the $w$-coin comes up tails first \emph{and} the $x$-coin comes up heads for at least $c$ additional rounds. It is thus clear that the conditional distribution of $Z$, given $Z \geq c$, is $c+\Geom(x)$. 

Applying this to our above situation, we have that conditioning $X_{n-k_r+1} + \ldots + X_n$ to be above some positive number, it will have a geometric distribution. Specifically,
\begin{equation*}
    \Pr((X_{n-k_r+1} + \ldots + X_n) - (a_{r-1}-a_r) = \ell | X_{n-k_r+1} + \ldots + X_n  \geq a_{r-1}-a_r) = (1-t^{n-k_r}x)(t^{n-k_r}x)^\ell.
\end{equation*}
Hence by \eqref{eq:defY}, $Y \sim \Geom(t^{n-k_r}x)$, so $Y$ has probability generating function $\frac{1-t^{n-k_r}x}{1-t^{n-k_r}xz}$. Thus $Y + X_{n-k_r-k_{r-1}+1} + \ldots + X_{n-k_r}$ has probability generating function
\begin{equation*}
    \frac{1-t^{n-k_r}x}{1-t^{n-k_r}xz} \cdot \left(\frac{1-t^{n-k_r-k_{r-1}}x}{1-t^{n-k_r} x} \frac{1-t^{n-k_r} x z}{1-t^{n-k_r-k_{r-1}}x z}\right) = \frac{1-t^{n-k_r-k_{r-1}}x}{1-t^{n-k_r-k_{r-1}}x z},
\end{equation*}
i.e. 
\begin{equation*}
    Y + X_{n-k_r-k_{r-1}+1} + \ldots + X_{n-k_r} \sim \Geom(t^{n-k_r-k_{r-1}}x).
\end{equation*}
Thus at last we have
\begin{equation}\label{eq:case2nudist}
    \Pr(N_{r-1} = a_{r-1}+\ell) = \begin{cases}
    (1-t^{n-k_r-k_{r-1}}x)(t^{n-k_r-k_{r-1}}x)^\ell  & 0 \leq \ell < a_{r-2}-a_{r-1} \\
    (t^{n-k_r-k_{r-1}}x)^\ell & \ell = a_{r-2}-a_{r-1}
    \end{cases},
\end{equation}
which concludes the computation of Case II.

A key feature of the distributions computed in \eqref{eq:leftnudist}, \eqref{eq:case1nudist} and \eqref{eq:case2nudist} is that the $1-t^{n-k_r} x$ term in \eqref{eq:leftnudist}, which appears only in the case $N_r < a_{r-1}$ (Case I), cancels with the $1-t^{n-k_r} x$ appearing in the computation \eqref{eq:case1nudist} for Case I; meanwhile, when $N_r = a_{r-1}$ (Case II), it appears neither in \eqref{eq:leftnudist} nor in \eqref{eq:case2nudist}. 

Together, \eqref{eq:leftnudist}, \eqref{eq:case1nudist} and \eqref{eq:case2nudist} imply the joint distribution
\begin{align*}
    &\Pr(N_r = a_r+\ell_1 \text{ and }N_{r-1} = a_{r-1}+\ell_2) \\
    &= 
    \frac{(1-t^{n-k_r-k_{r-1}}x)^{\bbone(\ell_2 < a_{r-2}-a_{r-1})}}{1-t^n x}(1-t^{k_r})^{\bbone(\ell_1>0)}(1-t^{k_{r-1}})^{\bbone(\ell_1 < a_{r-1}-a_r \text{ and }\ell_2 > 0)} \\ &\cdot (t^{n-k_r}x)^{\ell_1}(t^{n-k_r-k_{r-1}}x)^{\ell_2}\\
    &= \frac{(1-t^{n-k_r-k_{r-1}}x)^{\bbone(\ell_2 < a_{r-2}-a_{r-1})}}{1-t^n x}(1-t^{k_r})^{\bbone(m_{a_r}(\l)=m_{a_r}(\nu)+1)}(1-t^{k_{r-1}})^{\bbone(m_{a_{r-1}}(\l) = m_{a_{r-1}}(\nu)+1)}\\
    & \cdot (t^{n-k_r}x)^{\ell_1}(t^{n-k_r-k_{r-1}}x)^{\ell_2}
\end{align*}
But we see that the computation of the distribution of $N_{r-1}$ is exactly the same for any $N_i$. There is the same division into Case I and Case II depending on whether $N_{i+1}$ achieves its maximum, and the feature that the $1-t^{n-k_r} x$ terms cancel in both Case I and Case II is also the same. Hence these terms telescope, and we are left with 
\begin{equation*}
    \frac{1-t^{n-k_r-\ldots-k_1}x}{1-t^nx} = \frac{1-x}{1-t^n x},
\end{equation*}
where the $1-t^{n-k_r-\ldots-k_1}x$ appears because the last such term does not cancel. Hence continuing the above computation yields
\begin{align}
    &\Pr(N_r = \nu_{n-k_r+1},N_{r-1} = \nu_{n-k_r-k_{r-1}+1},\ldots,N_1 = \nu_1) \\
    &=\frac{1-x}{1-t^n x}\prod_{i=1}^r (1-t^{k_i})^{\bbone(m_{a_i}(\l)=m_{a_i}(\nu)+1)}\prod_{i=0}^{r-1}(t^{n-k_r-\ldots-k_{r-i}}x)^{\nu_{n-k_r-\ldots-k_{r-i}+1}-a_{r-i}} \\
    &= \frac{1-x}{1-t^nx} \prod_{j: m_j(\l)=m_j(\nu)+1}(1-t^{m_j(\l)}) \prod_{i=1}^n (xt^{i-1})^{\nu_i-\l_i},
\end{align}
concluding the proof.
\end{proof}

\begin{rmk}
Since the sum over $\nu$ of the LHS of \eqref{eq:sampling_alg} is clearly $1$, Proposition \ref{prop:sampling_alg} implies that the sum of the RHS of \eqref{eq:sampling_alg} is $1$, which gives a proof of the corresponding case of the skew Hall-Littlewood Cauchy identity (\Cref{lem:asym_cauchy}).
\end{rmk}

It is very important to note that the random variables $X_i$ above satisfy $\E[X_i] > \E[X_j]$ when $i < j$. This means that the $i\tth$ particle, which is already ahead of the $j\tth$ particle, is likely to pull even further ahead if one iterates the above dynamics. Empirically this may be seen in \Cref{fig:random_walks}. This observation is key to the proof of \Cref{thm:general_hl_lln_clt}, as it implies that while there may be some interactions between particles, as one iterates the above dynamics the particles should spread apart and interactions should not contribute to the limit. Hence by Donsker's theorem the rescaled fluctuations of the particles should look like independent Brownian motions. 

The rest of this section is devoted to making the above heuristic argument precise. We implement it by coupling the interacting particle dynamics of \Cref{prop:sampling_alg} to dynamics in which the particles do not interact at all, and showing that the error between the two is small in the limit.

\subsection{Coupling to non-interacting particle dynamics.}\label{subsec:non-interacting_intro}

Proposition \ref{prop:sampling_alg} gives an explicit sampling algorithm for Hall-Littlewood processes
\begin{equation*}
    \Pr(\l_1,\ldots,\l_N) = \frac{Q_{\l_N/\l_{N-1}}(x_N) \cdots Q_{\l_2/\l_1}(x_2)Q_{\l_1/(0[n])}(x_1) P_{\l_N}(1,\ldots,t^{n-1})}{\Pi(1,\ldots,t^{n-1}; x_1,\ldots,x_N)}.
\end{equation*}
\Cref{thm:lln_and_func_clt_rmt_version} and \Cref{thm:general_hl_lln_clt} treat Hall-Littlewood processes as above but with the variables $x_i$ replaced by geometric progressions $\hx_i$ (sometimes infinite), and we must extend our notation slightly to deal with these. We begin by setting up the appropriate probability space on which the random variables $X_i$ of Proposition \ref{prop:sampling_alg} can be defined in this more general setting.

\begin{defi}\label{def:generalized_variable_stuff}
A \emph{generalized variable} $\hx$ is a tuple $(x, tx,\ldots,t^{m-1}x)$ or $(x,tx,\ldots)$ in finite or infinite geometric progression with common ratio $t$. For a generalized variable, define probability spaces
\begin{equation*}
    \Omega_{\hx} = \begin{cases}
    (\Z_{\geq 0}^n)^m & \hx = (x,\ldots,t^{m-1}x) \\
    \{\bsomega = (\omega^{(1)},\omega^{(2)},\ldots) \in (\Z_{\geq 0}^n)^\infty: \text{ only finitely many $\omega^{(i)}$ nonzero}\} & \hx = (x,tx,\ldots)
    \end{cases}.
\end{equation*}
Recall the definition of the measure $G_x$ in Definition \ref{def:geom_diff_rvs}. Now define the measure $G_{\hx}$ on $\Omega_{\hx}$ by 
\begin{equation*}
    G_{\hx} = 
    \begin{cases}
    G_x \times \cdots \times G_{t^{m-1} x} & \hx = (x, tx,\ldots,t^{m-1}x) \\
    G_x \times G_{tx} \times \cdots & \hx = (x, tx, \ldots)
    \end{cases}.
\end{equation*}
\end{defi}

Two things must be justified in this definition. The first is that the infinite product measure $G_x \times G_{tx} \times \cdots$ on $(\Z_{\geq 0}^n)^\infty$ makes sense, which follows from the Kolmogorov extension theorem. The second is that this measure is actually supported on the subset $\Omega_{\hx}$, which follows from a standard Borel-Cantelli argument.

\begin{defi}\label{def:gen_var_stuff2}
We inductively define $\iota$ on $(\Z_{\geq 0}^n)^m, m > 1$ as follows. For $\omega_i \in \Z_{\geq 0}^n$, set
\begin{equation}\label{eq:iota_on_generalized_variables}
    \iota((\omega_1,\ldots,\omega_m);\l) := \iota(\omega_m;\iota((\omega_1,\ldots,\omega_{m-1});\l)),
\end{equation}
We define $\iota: \Omega_{\hx} \times \Sig_n \to \Sig_n$ as above when $\hx$ is a finite geometric progression, and when $\hx$ is an infinite geometric progression the definition readily extends because $\Omega_{\hx}$ consists of sequences with only finitely many nonzero $\omega_i \in \Z_{\geq 0}^n$. Given a sequence $\hx_1,\hx_2,\ldots$ of generalized variables as in Theorem \ref{thm:general_hl_lln_clt}, we will use the following notations.
\begin{itemize}
    \item $\Omega := \Omega_{\hx_1} \times \Omega_{\hx_2} \times \cdots$.
    \item $\bsomega  = (\omega^{(1)}, \omega^{(2)},\ldots)$ will denote an element of $\Omega$, with each $\omega^{(i)}$ denoting an element of $\Omega_{\hx_i}$.
\end{itemize}
\end{defi}

\begin{defi}\label{def:lambda_process}
Define the sequence $\l(0), \l(1),\ldots$ of random signatures on the probability space $\Omega$ of Definition \ref{def:gen_var_stuff2} by setting $\l(0) = (0[n])$ and inductively defining
\begin{equation*}
    \l(k,\bsomega) := \iota(\omega^{(k)}; \l(k-1,\bsomega))
\end{equation*}
for $\bsomega \in \Omega$, where $\iota$ is defined on $\omega^{(k)} \in \Omega_{\hx_k}$ via \eqref{eq:iota_on_generalized_variables}. We will usually omit the dependence on the element of the probability space $\Omega$ and simply write $\l(k)$.
\end{defi}

In other words, if $\hx_{k} = (x_k,\ldots,t^{m-1}x_k)$, then $\l(k+1)$ comes from $\l(k)$ by inserting random arrays as in Proposition \ref{prop:sampling_alg} with distributions corresponding to the variables $x_k,tx_k,\ldots,t^{m-1}x_k$.

We now define the non-interacting variant of the randomized insertion algorithm of Proposition \ref{prop:sampling_alg}, where each particle's movement is independent of the others. This is easier to analyze, and it will be shown in Proposition \ref{prop:coupling_error} that the two may be coupled with asymptotically negligible effect on the particles' positions, thus reducing the analysis of the sampling algorithm in Proposition \ref{prop:sampling_alg} to something much simpler.

\begin{defi}\label{def:v_process}
Define the \emph{non-interacting insertion map} $\eta: \Z_{\geq 0}^n \times \Z^n \to \Z^n$ by 
\begin{equation*}
    \eta(a_1,\ldots,a_n; \bv) = (v_1+a_1,\ldots,v_n+a_n),
\end{equation*}
and extend to $\eta: \Omega_{\hx} \times \Z^n \to \Z^n$ as in \eqref{eq:iota_on_generalized_variables}. Define a random sequence $\bv(0),\bv(1),\ldots$ with $\bv(i) \in \Z^n$ on $\Omega$ by setting $\bv(0) = (0[n])$ and
\begin{equation*}
    \bv(k,\bsomega) := \eta(\omega^{(k)},\bv(k-1,\bsomega)).
\end{equation*}
\end{defi}

\begin{rmk}
Neither the input tuple nor the output tuple of $\eta$ must be a signature, and if either one happens to be, it does not imply that the other one is.
\end{rmk}

We now state the result mentioned earlier, that the `interacting' and `non-interacting' dynamics $\l(k)$ and $\bv(k)$ may be coupled together with a negligible difference between them.

\begin{prop}\label{prop:coupling_error}
Let $\hx_1,\hx_2,\ldots$ be a sequence of generalized variables, $\hx_i = (x_i, tx_i,\ldots,t^{m_i-1}x_i)$ (where we allow $m_i=\infty$), such that there exists $\delta>0$ for which $x_i \in (\delta,1-\delta)$.

Then with probability $1$ with respect to the product measure\footnote{Defined on $\Omega_{\hx_1} \times \Omega_{\hx_2} \times \cdots$ via the Kolmogorov extension theorem.} $G_{\hx_1} \times G_{\hx_2} \times \cdots$ on $\Omega$, 
\begin{equation}\label{eq:bounded_error}
    \sup_{k \in \Z_{\geq 0}} |\l_i(k) - v_i(k)|
\end{equation}
is bounded for every $i$.
\end{prop}

Informally, the particles interact when a particle behind jumps to the position of the particle in front. The following lemma shows that in the non-interacting case, such overlaps occur a negligible amount, which will be used in the proof of Proposition \ref{prop:coupling_error} to show that interactions contribute negligibly overall as well. 

\begin{lemma}\label{lem:finitely_many_overlaps}
With the same hypotheses on the $\hx_i$ as in Proposition \ref{prop:coupling_error}, we have that with probability $1$, the set
\begin{equation*}
    \{k \in \Z_{\geq 0}: v_i(k) \leq v_{i+1}(k+1)+B\}
\end{equation*}
is finite for all $B \in \Z$ and all $i$.
\end{lemma}

The proof of \Cref{lem:finitely_many_overlaps} will be deferred to \Cref{subsec:non-interacting_analysis}.

\begin{proof}[Proof of Prop. \ref{prop:coupling_error}]
We construct a sequence $\l^{(1)}(k) = \bv(k), \l^{(2)}(k),\ldots,\l^{(n)}(k) = \l(k)$ of discrete-time stochastic processes on the space of particle configurations, all defined on $\Omega_{\hx_1} \times \Omega_{\hx_2} \times \cdots$. Informally, $\l^{(j)}$ is the process in which the last $j$ particles $\l^{(j)}_n, \ldots, \l^{(j)}_{n-j}$ interact as in Definition \ref{def:lambda_process}, but particles $\l^{(j)}_{n-j+1},\ldots,\l^{(j)}_1$ do not interact with any other particles, as in \Cref{def:v_process}. We will then prove by induction on $j$ that with probability $1$,
\begin{equation}\label{eq:inductive_bounded_error}
    \sup_{k \in \Z_{\geq 0}} |\l_i^{(j)}(k) - v_i(k)|
\end{equation}
is finite for all $i$. When $j=n$, this will prove \eqref{eq:bounded_error}. 

Now let us be more formal. Let $\Sig_n^{(j)} = \{(v_1,\ldots,v_n) \in \Z^n: v_n \leq \ldots \leq v_{n-j+1}\}$. Following the indexing theme above, we see that $\Sig_n^{(1)} = \Z^n$ and $\Sig_n^{(n)} = \Sig_n$. Once we have defined $\l^{(i)}$ it will be true that $\l^{(i)}$ takes values in $\Sig_n^{(i)}$.

Now, define $\eta^{(j)}: \Z_{\geq 0}^n \times \Sig_n^{(j)} \to \Sig_n^{(j)}$ by
\begin{equation}\label{eq:hybrid_insertion}
    \eta^{(j)}(a_1,\ldots,a_n; v) = (\eta(a_1,\ldots,a_{n-j} ; v_1, \ldots, v_{n-j}),\iota(a_{n-j+1},\ldots,a_n;v_{n-j+1},\ldots,v_n)).
\end{equation}
In other words, particles $v_{n-j+1},\ldots,v_n$ try to jump by $a_{n-j+1},\ldots,a_n$ units respectively, but may donate some of their jumps to the next particle as in the definition of $\iota$, while particles $v_1,\ldots,v_{n-j}$ each jump by $a_1,\ldots,a_{n-j}$ units respectively, independent of the positions of all other particles. It is clear from this description that the image of $\eta^{(j)}$ is indeed $\Sig_n^{(j)}$. It is also clear that $\eta^{(n)} = \iota$, and that
\begin{equation*}
    \l^{(j)}_i(k) = v_i(k) \text{   for $i=1,\ldots,n-j$.}
\end{equation*}
When $j=1$ this means that the first $n-1$ particles do not interact and hence the $n\tth$ particle has no one to interact with, therefore $\eta^{(1)} = \eta$. Just as in \eqref{eq:iota_on_generalized_variables}, we extend $\eta^{(j)}$ to a map $\Omega_{\hx} \times \Sig_n^{(j)} \to \Sig_n^{(j)}$ for any generalized variable $\hx$.

Finally, given generalized variables $\hx_1,\hx_2,\ldots$, we define the discrete-time stochastic processes $\l^{(j)}(k)$ on the probability space $\Omega := \Omega_{\hx_1} \times \Omega_{\hx_2} \times \cdots$ by setting $\l^{(j)}(0) = (0[n])$ and
\begin{equation}\label{eq:def_intermediate_arrays}
    \l^{(j)}(k,\bsomega) = \eta^{(j)}(\omega^{(k)},\l^{(j)}(k-1,\bsomega))
\end{equation}
where $\bsomega = (\omega^{(1)},\ldots) \in \Omega$ as in Definition \ref{def:gen_var_stuff2}. We will usually write the random variable $\l^{(j)}(k)$ without the dependence on $\bsomega$.

We claim that the inequalities
\begin{equation}\label{eq:ind_lower_bound}
    \l^{(j+1)}_{n-j}(k) \geq \l^{(j)}_{n-j}(k)
\end{equation}
and 
\begin{equation}\label{eq:ind_upper_bound}
    \l^{(j+1)}_{n-j+i}(k) \leq \l^{(j)}_{n-j+i}(k) \text{   for }i=1,\ldots,j
\end{equation}
hold for all $k$. We prove this by induction on $k$, the base case $k=0$ following since $\l^{(\ell)}(0) = (0[n])$. Suppose that \eqref{eq:ind_lower_bound} and \eqref{eq:ind_upper_bound} hold for some $k$. Since \eqref{eq:def_intermediate_arrays} defines $\l^{(j+1)}(k+1)$ and $\l^{(j)}(k+1)$ by inserting $\omega^{(k)} \in \Omega_{\hx_k}$, which is a sequence of elements of $\Z_{\geq 0}^n$, it suffices to show that the inequalities \eqref{eq:ind_lower_bound} and \eqref{eq:ind_upper_bound} remain true after inserting a single element of $\Z_{\geq 0}^n$. To be precise, it suffices to show that for any $\nu \in \Sig_n^{(j+1)}, \mu \in \Sig_n^{(j)}$ such that\footnote{Note that \eqref{eq:nu_ineq1}, \eqref{eq:nu_ineq2} are the same as \eqref{eq:ind_lower_bound} and \eqref{eq:ind_upper_bound}.}
\begin{align}
    \nu_{n-j} &\geq \mu_{n-j} \label{eq:nu_ineq1}\\
    \nu_{n-j+i} & \leq \mu_{n-j+i} \text{ for $i=1,\ldots,j$} \label{eq:nu_ineq2}
\end{align}
and $a \in \Z_{\geq 0}^n$, one has
\begin{align}
    \eta^{(j+1)}(a;\nu)_{n-j} &\geq \eta^{(j)}(a;\mu)_{n-j}\label{eq:eta_geq_to_show} \\
    \eta^{(j+1)}(a;\nu)_{n-j+i} &\leq \eta^{(j)}(a;\mu)_{n-j+i} \text{ for $i=1,\ldots,j$}. \label{eq:etq_leq_to_show}
\end{align}
\eqref{eq:eta_geq_to_show} is clear because $\eta^{(j)}(a;\mu)_{n-j} = \mu_{n-j}+a_{n-j}$ while $\eta^{(j+1)}(a;\nu)_{n-j} \geq \nu_{n-j}+a_{n-j}$ (where the possible $>$ comes from the fact that $\nu_{n-j}$ may get pushed by the preceding particle). We now turn to \eqref{eq:etq_leq_to_show}

Applying \eqref{eq:iota_uniform_def} to the $\iota$ in \eqref{eq:hybrid_insertion}, we have
\begin{multline}\label{eq:short-term1}
    \eta^{(j+1)}(a;\nu)_{n-j+i}\\ = \min(\nu_{n-j+i-1}, \max(\nu_{n-j+i}+a_{n-j+i},\nu_{n-j+i+1}+a_{n-j+i}+a_{n-j+i+1},\ldots,\nu_n+a_{n-j+i}+\ldots+a_n))
\end{multline}
for $i=1,\ldots,j$.
Similarly, \eqref{eq:iota_uniform_def} implies
\begin{multline}\label{eq:short-term2}
    \eta^{(j)}(a;\mu)_{n-j+i} =  \begin{cases}
    \min(\mu_{n-j+i-1},\max(\mu_{n-j+i}+a_{n-j+i},\ldots,\mu_n+a_{n-j+i}+\ldots+a_n)) & i \geq 2 \\
    \max(\mu_{n-j+i}+a_{n-j+i},\ldots,\mu_n+a_{n-j+i}+\ldots+a_n) & i=1
    \end{cases}
\end{multline}
Because $\nu_{n-j+i} \leq \mu_{n-j+i},\ldots,\nu_n \leq \mu_n$, we have
\begin{multline}\label{eq:mu_nu_ineq}
    \max(\nu_{n-j+i}+a_{n-j+i},\nu_{n-j+i+1}+a_{n-j+i}+a_{n-j+i+1},\ldots,\nu_n+a_{n-j+i}+\ldots+a_n)) \\
    \leq \max(\mu_{n-j+i}+a_{n-j+i},\mu_{n-j+i+1}+a_{n-j+i}+a_{n-j+i+1},\ldots,\mu_n+a_{n-j+i}+\ldots+a_n)).
\end{multline}
and 
\begin{equation}\label{eq:mu_nu_ineq2}
    \nu_{n-j+i-1} \leq \mu_{n-j+i-1}\text{ when $i \geq 2$}.
\end{equation}
Combining the definitions \eqref{eq:short-term1} and \eqref{eq:short-term2} with the inequalities \eqref{eq:mu_nu_ineq} and \eqref{eq:mu_nu_ineq2} yields the desired
\begin{equation*}
    \eta^{(j+1)}(a;\nu)_{n-j+i} \leq \eta^{(j)}(a;\mu)_{n-j+i} \text{    for $i=1,\ldots,j$.}
\end{equation*}

Thus we have proven \eqref{eq:ind_lower_bound} and \eqref{eq:ind_upper_bound}.

We finally turn to the proof of \eqref{eq:inductive_bounded_error}, by induction on $j$. The base case $j=1$ follows because $\l^{(1)}(k) = \bv(k)$ for all $k$ as noted earlier. Thus we will suppose that \eqref{eq:inductive_bounded_error} holds for some $j \geq 1$ and verify that it holds for $j+1$.

We will first show 
\begin{equation}\label{eq:inductive_bounded_error_special_case}
    \sup_{k \in \Z_{\geq 0}} |\l_{n-j}^{(j+1)}(k) - v_{n-j}(k)| < \infty
\end{equation}
almost surely. First note that for $k$ such that $\l_{n-j+1}^{(j+1)}(k+1) \leq \l^{(j+1)}_{n-j}(k)$,
\begin{equation*}
    \l_{n-j}^{(j+1)}(k+1)-\l_{n-j}^{(j+1)}(k) = v_{n-j}(k+1)-v_{n-j}(k)
\end{equation*}
because no pushing occurs. For $k$ such that 
\begin{equation}\label{eq:bad_overlap}
\l_{n-j+1}^{(j+1)}(k+1) > \l^{(j+1)}_{n-j}(k)    
\end{equation}
$\l^{(j+1)}_{n-j}$ may receive some push from $\l_{n-j+1}^{(j+1)}$, causing it to move further than $v_{n-j}$ does during that round. Hence to show \eqref{eq:inductive_bounded_error_special_case}, it suffices to show that the number of $k$ for which \eqref{eq:bad_overlap} holds is almost surely finite, as then the error $\sup_{k \in \Z_{\geq 0}} |\l_{n-j}^{(j+1)}(k) - v_{n-j}(k)|$ is a sum of a finite number of almost surely finite random variables (each one representing the amount by which the $(n-j)\tth$ particle gets pushed).

By \eqref{eq:ind_lower_bound}, $\l^{(j+1)}_{n-j}(k) \geq \l^{(j)}_{n-j}(k)$, and by \eqref{eq:ind_upper_bound} $\l_{n-j+1}^{(j+1)}(k+1) \leq \l^{(j)}_{n-j+1}(k+1)$. Hence for $k$ such that \eqref{eq:bad_overlap} holds,
\begin{equation*}
\l_{n-j+1}^{(j)}(k+1) > \l^{(j)}_{n-j}(k)  
\end{equation*}
also holds, and since $ \l^{(j)}_{n-j}(k)  = v_{n-j}(k) $, we have that
\begin{equation}\label{eq:bad_overlap2}
v_{n-j+1}(k+1) + (\l_{n-j+1}^{(j)}(k+1)-v_{n-j+1}(k+1)) > v_{n-j}(k)  
\end{equation}
holds as well, so it suffices to show that
\begin{equation}\label{eq:final_finiteness}
    |\{k: v_{n-j+1}(k+1) + (\l_{n-j+1}^{(j)}(k+1)-v_{n-j+1}(k+1)) > v_{n-j}(k) \}| < \infty \text{    a.s.}
\end{equation}
By the inductive hypothesis that \eqref{eq:inductive_bounded_error} holds for $j$, we have that 
\begin{equation}\label{eq:another_special_case_of_bound_error}
    \sup_{k \in \Z_{\geq 0}} |\l_{n-j+1}^{(j)}(k) - v_{n-j+1}(k)| < \infty
\end{equation}
almost surely. Since by Lemma \ref{lem:finitely_many_overlaps},
\begin{equation*}
    \{k: v_{n-j+1}(k+1) + B > v_{n-j}(k)\}
\end{equation*}
is almost surely finite for all $B$, it is in particular almost surely finite for the \emph{random} 
\begin{equation*}
B = \sup_{k \in \Z_{\geq 0}} |\l_{n-j+1}^{(j)}(k) - v_{n-j+1}(k)|    
\end{equation*}
(the order of quantifiers in \Cref{lem:finitely_many_overlaps} is important for this conclusion). Since $\l_{n-j}^{(j)}(k+1)-v_{n-j}(k+1)$ is almost surely bounded, \eqref{eq:final_finiteness} follows. This completes the proof of \eqref{eq:inductive_bounded_error_special_case}.

Now, since 
\begin{align}
    \l^{(j+1)}_{n-j}(k) &\geq \l^{(j)}_{n-j}(k) \\
    \l^{(j+1)}_{n-j+i}(k) &\leq \l^{(j)}_{n-j+i}(k) \text{     for $i=1,\ldots,j$} \\
    \l^{(j+1)}_{i}(k) &= \l^{(j)}_{i}(k) \text{     for $i=1,\ldots,n-j-1$}\\
    \sum_{i=1}^n \l^{(j+1)}_{i}(k) &= \sum_{i=1}^n  \l^{(j)}_{i}(k),
\end{align}
it follows that 
\begin{equation}\label{eq:back_and_forward_pushing_equality}
    \sum_{i=1}^j (\l^{(j)}_{n-j+i}(k) - \l^{(j+1)}_{n-j+i}(k)) = \l^{(j+1)}_{n-j}(k) - \l^{(j)}_{n-j}(k)
\end{equation}
(this is a kind of conservation of momentum: the amount that the $(n-j)\tth$ particle is pushed forward from collisions equals the amount that the particles behind it are pushed backward). We have
\begin{equation*}
    \sup_k |\l^{(j+1)}_{n-j}(k) - \l^{(j)}_{n-j}(k)| \leq \sup_k |\l^{(j+1)}_{n-j}(k)-v_{n-j}(k)| + \sup_k|\l^{(j)}_{n-j}(k)-v_{n-j}(k)| < \infty \text{    a.s.}
\end{equation*}
by applying \eqref{eq:inductive_bounded_error_special_case} to the first term and the inductive hypothesis to the second.
Because the summands $\l^{(j)}_{n-j+i}(k) - \l^{(j+1)}_{n-j+i}(k)$ on the LHS of \eqref{eq:back_and_forward_pushing_equality} are nonnegative, it follows that 
\begin{equation}\label{eq:lambda_diff_bound_inductive}
    \sup_k |\l^{(j)}_{n-j+i}(k) - \l^{(j+1)}_{n-j+i}(k)| < \infty \text{    a.s.}
\end{equation}
for $i=1,\ldots,j$. We thus have
\begin{equation*}
    \sup_k |\l^{(j+1)}_{n-j+i}(k) - v_{n-j+i}(k)| \leq \sup_k |\l^{(j)}_{n-j+i}(k) - \l^{(j+1)}_{n-j+i}(k)| + \sup_k |\l^{(j)}_{n-j+i}(k) -v_{n-j+i}| < \infty \text{    a.s.}
\end{equation*}
by applying \eqref{eq:lambda_diff_bound_inductive} to the first summand and the inductive hypothesis to the second. This establishes \eqref{eq:inductive_bounded_error} for $j+1$, for $i=n-j,n-j+1,\ldots,n$, and the equation is trivial (the supremum is just $0$) when $i=1,\ldots,n-j-1$. This completes the induction on $j$, showing that \eqref{eq:inductive_bounded_error} holds for all $i$ and $j$. In particular it holds for $j=n$, which proves Proposition \ref{prop:coupling_error}.

\end{proof}

\subsection{Analysis of non-interacting particle dynamics $\bv(k)$ and proof of {\Cref{thm:general_hl_lln_clt}}.} \label{subsec:non-interacting_analysis}

In the previous subsection, we phrased the relevant Hall-Littlewood process in terms of a particle system $\l(k)$ in which particles interact, then coupled it to a system $\bv(k)$ where they do not interact. In this subsection we analyze $\bv(k)$ to prove our results. We first record facts about the means, variances and fourth moments of jumps of $\bv(k)$ in Lemma \ref{lem:facts_about_differences}, some of which are used to give an overdue proof of Lemma \ref{lem:finitely_many_overlaps}, used in the previous subsection. We then apply them and Donsker's theorem to prove an analogue of Theorem \ref{thm:general_hl_lln_clt} for $\bv(k)$, and conclude the desired result for $\l(k)$ by our coupling and Proposition \ref{prop:coupling_error}

\begin{lemma}\label{lem:facts_about_differences}
Let $\delta > 0$ and let $\hx_1,\hx_2,\ldots$ be generalized variables such that $\delta < x_i < 1-\delta$ for all $i$ as in Proposition \ref{prop:coupling_error}. Let $\bv(k)$ be as in the previous subsection, $Y_i(k) := v_i(k)-v_i(k-1)$ for $k \geq 0$, and $\bY_i(k) = Y_i(k) - \E Y_i(k)$. Then
\begin{enumerate}[label={\arabic*.}]
    \item If $\hx_k = (x_k, \ldots, t^{m-1}x_k)$, then 
    \begin{equation}\label{eq:Yik_mean}
        \E Y_i(k) = \sum_{j=0}^{m-1}\frac{x_kt^{j+i-1}(1-t)}{(1-t^{j+i}x_k)(1-t^{j+i-1}x_k)}
    \end{equation}
    where we allow $m=\infty$. Consequently, there exist constants $b_i, B_i > 0$ such that $b_i < \E Y_i(k) < B_i$ for all $k$.
    \item For $\hx_k$ as above, we have
    \begin{equation}\label{eq:Yik_variance}
        \E[\bY_i(k)^2] = \sum_{j=0}^{m-1} \frac{t^{j+i-1}x_k(1 - t)  (1 - t^{2j+2i-1} x_k^2)}{(1 - t^{j+i-1}x_k)^2 (1 - t^{j+i} x_k)^2}.
    \end{equation}
    Consequently, there exist constants $c_i,C_i > 0$ such that $c_i < \E[\bY_i(k)^2] < C_i$ for all $k$.
    \item There exist constants $D_i > 0$ such that $\E[\bY_i(k)^4] < D_i$ for all $k$.
\end{enumerate}
\end{lemma}

\begin{defi}\label{def:mu_and_sigma}
We let $\mu(t^{i-1}\hx_k)$ denote the RHS of \eqref{eq:Yik_mean}, and $\sigma^2(t^{i-1}\hx_k)$ denote the RHS of \eqref{eq:Yik_variance}, which by Lemma \ref{lem:facts_about_differences} are the mean and variance of $Y_i(k)$ respectively.
\end{defi}

\begin{proof}
It follows from the definition of $v$ that
\begin{equation}\label{eq:what_v_is_explicitly}
    Y_i(k) = v_i(k) - v_i(k-1)= \sum_{j=0}^{m-1} Z_{t^{j+(i-1)} x_k},
\end{equation}
where $Z_{x} \sim G_x$ are independent. We compute
\begin{align*}
    \E Z_x &= \dderiv{}{y}|_{y=1} \sum_{\ell \geq 0} \Pr(Z_x = \ell)y^\ell \\
    &=  \dderiv{}{y}|_{y=1} \sum_{\ell \geq 0} \frac{1-x}{1-t x}(1-t)^{\bbone(\ell>0)}(xy)^\ell \\
    &= \dderiv{}{y}|_{y=1} \frac{1-x}{1-tx} \frac{1-txy}{1-xy} \\
    &= \left[\frac{1-x}{1-tx} \frac{-tx(1-xy)+x(1-txy)}{(1-xy)^2} \right]_{y=1} \\
    &= \frac{x(1-t)}{(1-tx)(1-x)}.
\end{align*}
Combining with \eqref{eq:what_v_is_explicitly} yields \eqref{eq:Yik_mean}. We have
\begin{equation*}
    \E Y_i(k) \geq \E Z_{t^{i-1}x_k} = \frac{t^{i-1}x_k(1-t)}{(1-t^{i}x_k)(1-t^{i-1}x_k)} > t^{i-1}x_k(1-t) \geq t^{i-1}(1-t)\delta,
\end{equation*}
so setting $b_i = t^{i-1}(1-t)\delta$ we have $b_i < \E Y_i(k)$. For the other bound,
\begin{align*}
    \E Y_i(k) &\leq \sum_{j=0}^\infty \frac{t^{j+i-1}x_k(1-t)}{(1-t^{j+i}x_k)(1-t^{j+i-1}x_k)} \\
    & < \frac{1}{(1-t^{i}(1-\delta))(1-t^{i-1}(1-\delta)))} \sum_{j=0}^\infty (1-t)t^{i-1}x_k \cdot t^j \\
    & < \frac{t^{i-1}\delta}{(1-t^{i}(1-\delta))(1-t^{i-1}(1-\delta)))},
\end{align*}
so we may set $B_i = \frac{t^{i-1}\delta}{(1-t^{i}(1-\delta))(1-t^{i-1}(1-\delta)))}$. This proves Part 1 of the lemma.

For Part 2, we have 
\begin{equation*}
    \bY_i(k) = \sum_{j=0}^{m-1} (Z_{t^{j+(i-1)} x_k} - \E[Z_{t^{j+(i-1)} x_k}]). 
\end{equation*}
Set $\bZ_x = Z_x - \E Z_x$. Since the $\bZ$'s above are independent, the variances add. Hence the lower bound $c_i < \E[\bY_i(k)^2]$ follows because $\E[\bZ_{t^{j+(i-1)} x_k}^2]$ is bounded below for $x_k \in (\delta,1-\delta)$. For the upper bound, we compute
\begin{align*}
    \E Z_x^2 &= \dderiv{}{y}|_{y=1} y\dderiv{}{y} \sum_{\ell \geq 0} \Pr(Z_x = \ell)y^\ell \\
    &= \dderiv{}{y}|_{y=1} y \frac{1-x}{1-tx} \frac{-tx(1-xy)+x(1-txy)}{(1-xy)^2} \\
    &= \frac{x(1 - t) (1 + x)}{(1 - x)^2 (1 - t x)},
\end{align*}
so
\begin{align}
    \E[\bZ_x^2] = \E Z_x^2 - (\E Z_x)^2 = \frac{x(1 - t)  (1 - t x^2)}{(1 - x)^2 (1 - t x)^2},
\end{align}
proving \eqref{eq:Yik_variance}. This is bounded above by $\frac{x}{\delta^2 (1-t(1-\delta))^2}$ for $0<x < 1-\delta$, hence 
\begin{align}
    \E[\bY_i(k)^2] &= \sum_{j=0}^{m-1} \E[\bZ_{t^{j+(i-1)} x_k}^2]\\
    & \leq \frac{1}{\delta^2 (1-t(1-\delta))^2} \frac{t^{i-1}x_k}{1-t} \\
    & < \frac{1}{\delta^2 (1-t(1-\delta))^2} \frac{t^{i-1}(1-\delta)}{1-t}
\end{align}
for all $x_k \in (\delta,1-\delta)$, so we may set $C_i$ to be the final expression. This proves Part 2. 

For the fourth moment,
\begin{align}\label{eq:4th_moment_expansion}
    \E[\bY_i(k)^4] = \sum_{j=0}^{m-1} \E[\bZ_{t^{j+(i-1)} x_k}^4] + \sum_{0 \leq j \neq \ell \leq m-1} \E[\bZ_{t^{j+(i-1)} x_k}^2] \E[\bZ_{t^{\ell+(i-1)} x_k}^2].
\end{align}
We have 
\begin{equation*}
    \E[\bZ_x^4] = \frac{(1-t) x (x^2 + 4 x + 1)}{(1-x)^3 (1-tx)}
\end{equation*}
by a similar generating function computation as before, and bounding the sum of these for $x,tx,t^2x,\ldots$ in terms of geometric series as before yields that
\begin{equation*}
    \sum_{j=0}^\infty \E[\bZ_{t^{j+i-1}x}^4]
\end{equation*}
is bounded uniformly over $x \in (0,1-\delta)$. Likewise, 
\begin{align*}
    &\sum_{0 \leq j \neq \ell \leq m-1} \E[\bZ_{t^{j+(i-1)} x_k}^2] \E[\bZ_{t^{\ell+(i-1)} x_k}^2] \\
    &\leq \left(\sum_{0 \leq j \leq m-1}\E[\bZ_{t^{j+(i-1)} x_k}^2]\right) \\
    &< \left(\frac{1}{\delta^2 (1-t(1-\delta))^2} \frac{t^{i-1}(1-\delta)}{1-t}\right)^2
\end{align*}
by using our previous variance bound at the last step. Hence we have bounded both sums on the RHS of \eqref{eq:4th_moment_expansion} uniformly in $x_k \in (0,1-\delta)$, and Part 3 follows.
\end{proof}

We now prove Lemma \ref{lem:finitely_many_overlaps} as promised.

\begin{proof}[Proof of Lemma \ref{lem:finitely_many_overlaps}]
We first claim that it suffices to show that for any given $B$, 
\begin{equation*}
    |\{k \in \Z_{\geq 0}: v_i(k) \leq v_{i+1}(k+1)+B\}| < \infty \text{    a.s.}
\end{equation*}
(this differs from the statement of Lemma \ref{lem:finitely_many_overlaps} in order of quantifiers). This is immediate because
\begin{multline}
    \{\bsomega \in \Omega: |\{k \in \Z_{\geq 0}: v_i(k,\bsomega) \leq v_{i+1}(k+1,\bsomega)+B'\}| = \infty \text{ for some }B'\} \\
    = \bigcup_{B \in \N} \{\bsomega \in \Omega: |\{k \in \Z_{\geq 0}: v_i(k,\bsomega) \leq v_{i+1}(k+1,\bsomega)+B\}|\}
\end{multline}
so it suffices to show the sets on the RHS have measure $0$. This is what we will now do.

It follows from the formula in Lemma \ref{lem:facts_about_differences} Part 1 that $\E[Y_{i+1}(k)] \leq t \E[Y_i(k)]$, hence
\begin{equation}\label{eq:mean_bound}
    \sum_{j=1}^k \E Y_i(j)-\sum_{j=1}^{k+1} \E Y_{i+1}(j) \geq (1-t)\sum_{j=1}^k \E Y_i(k) - \E Y_{i+1}(k+1) \geq (1-t)b_i \cdot k - B_{i+1}
\end{equation}
where $b_i, B_{i+1}$ are the constants in Lemma \ref{lem:facts_about_differences}. For $k$ such that the RHS of \eqref{eq:mean_bound} is positive,
\begin{align*}
    \Pr(v_i(k) \leq v_{i+1}(k+1)+B) &= \Pr\left( \sum_{j=1}^k \bY_i(j) - \sum_{j=1}^{k+1}\bY_{i+1}(j) \leq B + \sum_{j=1}^{k+1} \E Y_{i+1}(j)-\sum_{j=1}^k \E Y_i(j)\right) \\
    & \leq \Pr\left( \sum_{j=1}^k \bY_i(j) - \sum_{j=1}^{k+1}\bY_{i+1}(j) \leq B + (1-t)b_i \cdot k - B_{i+1}\right) \\
    & \leq \Pr\left(\abs*{\sum_{j=1}^k \bY_i(j) - \sum_{j=1}^{k+1}\bY_{i+1}(j)} \leq B + (1-t)b_i \cdot k - B_{i+1}\right)
\end{align*}
(the last step is the only one using the positivity assumption). By Markov's inequality,
\begin{multline}\label{eq:markov_ineq}
    \Pr\left(\abs*{\sum_{j=1}^k \bY_i(j) - \sum_{j=1}^{k+1}\bY_{i+1}(j)} \leq B + (1-t)b_i \cdot k - B_{i+1}\right) \\
    \leq \frac{\E\left[\left(\sum_{j=1}^k \bY_i(j) - \sum_{j=1}^{k+1}\bY_{i+1}(j)\right)^4\right]}{(B + (1-t)b_i \cdot k - B_{i+1})^4}.
\end{multline}
By Lemma \ref{lem:facts_about_differences},
\begin{align*}
    &\E\left[\left(\sum_{j=1}^k \bY_i(j) - \sum_{j=1}^{k+1}\bY_{i+1}(j)\right)^4\right] \\
    &= \sum_{j=1}^k \E[\bY_i(j)^4] + \sum_{j=1}^{k+1} \E[\bY_{i+1}(j)^4] + \sum_{j=1}^k \sum_{\ell=1}^{k+1} \E[\bY_i(j)^2]\E[\bY_{i+1}(j)^2] \\
    & < k D_i + (k+1) D_{i+1} + k(k+1) C_i C_{i+1} \\
    & = O(k^2).
\end{align*}
Hence the RHS of \eqref{eq:markov_ineq} is $O(1/k^2)$. Thus 
\begin{equation*}
    \sum_k \Pr(v_i(k) \leq v_{i+1}(k+1)+B) < \infty
\end{equation*}
and so by Borel-Cantelli, $\{k \in \Z_{\geq 0}: v_i(k) \leq v_{i+1}(k+1)+B\}$ is almost-surely finite, completing the proof.
\end{proof}

\begin{proof}[Proof of Theorem \ref{thm:general_hl_lln_clt}]
We begin with the first claim \eqref{eq:lln}, the law of large numbers. By Lemma \ref{lem:facts_about_differences}, $\mu(t^{i-1}\hx_k)$ and $\sigma^2(t^{i-1}\hx_k)$ are the mean and variance, respectively, of $Y_i(k)$.
Since $\sum_{j=1}^k Y_i(j) = v_i(k)$, it suffices to show
\begin{equation}\label{eq:simpler_lln}
    \frac{\l_i(k) - \E v_i(k)}{k} \to 1 \text{       a.s.   as $k \to \infty$.}
\end{equation}
By Proposition \ref{prop:coupling_error}, $|\l_i(k) - v_i(k)|$ is almost-surely bounded as $k \to \infty$.
It follows that 
\begin{equation}\label{eq:lambda_v_diff_doesnt_affect_lln}
    \frac{v_i(k) - \l_i(k)}{k} \to 0 \text{    a.s.   as $k \to \infty$.}
\end{equation}
The uniform variance bound in Lemma \ref{lem:basic_weak_convergence_fact} Part 2 ensures that the sequence of random variables $Y_i(1),Y_i(2),\ldots$ satisfies the hypothesis of Kolmogorov's strong law of large numbers \cite[Ch. IV.\textsection 3, Thm. 2]{shiryaev}, hence
\begin{equation}\label{eq:kolmogorov_lln_for_v}
    \frac{\sum_{j=1}^k Y_i(j) - \sum_{j=1}^k \E Y_i(j)}{k} \to 0 \text{      a.s.    as $k \to \infty$.}
\end{equation}
We have
\begin{equation}\label{eq:split_for_lln}
    \frac{\l_i(k)}{\sum_{j=1}^k \mu(t^{i-1}\hx_j)}  = 1+\frac{k}{\sum_{j=1}^k \mu(t^{i-1}\hx_j)}\left(\frac{v_i(k)-\E[v_i(k)]}{k} + \frac{\l_i(k)-v_i(k)}{k}\right).
\end{equation}
By Lemma \ref{lem:facts_about_differences} Part 1, 
\begin{equation*}
    \abs*{\frac{k}{\sum_{j=1}^k \mu(t^{i-1}\hx_j)}} \leq \frac{1}{b_i},
\end{equation*}
so
\begin{equation*}
    \abs*{\frac{\l_i(k)}{\sum_{j=1}^k \mu(t^{i-1}\hx_j)}-1} \leq \frac{1}{b_i}\left(\frac{v_i(k)-\E[v_i(k)]}{k} + \frac{\l_i(k)-v_i(k)}{k}\right).
\end{equation*}
By \eqref{eq:kolmogorov_lln_for_v} and \eqref{eq:lambda_v_diff_doesnt_affect_lln} respectively, the two terms inside the parentheses on the RHS go to $0$ almost surely as $k \to \infty$. This proves \eqref{eq:lln}, the law of large numbers.

To show the second claim of Theorem \ref{thm:general_hl_lln_clt}, namely the convergence of the rescaled $\l_i$ to Brownian motions, we will use the same strategy of first showing convergence for the $v_i$ and then utilizing the coupling. Let
\begin{equation*}
    \barv_i(k) = v_i(k) - \sum_{j=1}^k \mu(t^{i-1}\hx_j) = \sum_{j=1}^k \bY_i(j).
\end{equation*}
Define $f_{\barv_i,k}$, a $C[0,1]$-valued random variable on the probability space $\Omega$ of \Cref{def:gen_var_stuff2} by setting $f_{\tv_i,k}(0)=0$,
\begin{equation*}
    (f_{\barv_i,k}(1/k),f_{\barv_i,k}(2/k),\ldots,f_{\barv_i,k}(1)) = \frac{1}{\sqrt{\sum_{j=1}^k \sigma^2(t^{i-1}\hx_j)}}(\barv_i(1),\ldots,\barv_i(k))
\end{equation*}
and linearly interpolating $f_{\barv_i,k}$ at other values in $[0,1]$. Let the measure $M_{\barv_i,k}$ on $C[0,1]$ be the distribution of $f_{\barv_i,k}$. We claim that as $k \to \infty$, $M_{\barv_i,k}$ converges weakly to the Wiener measure $P_W$ on $C[0,1]$. By Donsker's theorem\footnote{Many versions in print require that the increments $\bY_i(j)$ be identically distributed as well as independent, but a version for random walks with distinct independent increments may be obtained by specializing Donsker's theorem for martingales \cite[Thm. 3]{brown1971martingale} to this case.}, this convergence holds if the Lindeberg condition is satisfied, and it is well-known (see e.g. \cite[Ch. III, \textsection 4.I.2]{shiryaev}) that the Lindeberg condition is implied by the Lyapunov condition. The latter, in our case, reads that for some $\delta > 0$, 
\begin{equation}\label{eq:lyapunov}
    \frac{1}{(\sum_{j=1}^k \sigma^2(t^{i-1}\hx_j))^{1+\delta/2}} \sum_{j=1}^k \E[\bY_i(j)^{2+\delta}] \to 0 \text{        as }k \to \infty.
\end{equation}
We will prove \eqref{eq:lyapunov} when $\delta = 2$. Letting $c_i,D_i$ be as in Lemma \ref{lem:facts_about_differences}, we have
\begin{equation*}
    \left(\sum_{j=1}^k \sigma^2(t^{i-1}\hx_j)\right)^2 > k^2c_i^2
\end{equation*}
and 
\begin{equation*}
     \sum_{j=1}^k \E[\bY_i(j)^{4}] < k D_i.
\end{equation*}
Hence the expression in \eqref{eq:lyapunov} is bounded above by $\frac{D_i}{c_i^2} \frac{1}{k}$, and \eqref{eq:lyapunov} follows immediately. This verifies that $M_{\barv_i,k}$ converges weakly to $P_W$ as $k \to \infty$. Because $v_1,\ldots,v_n$ are independent, we also have that the product measure $M_{\barv_i,k} \times \cdots \times M_{\barv_i,k}$ on $(C[0,1])^n$ converges weakly to $P_W^n$, i.e. $(f_{\barv_1,k},\ldots,f_{\barv_n,k})$ converges in distribution to $n$ independent Brownian motions.

We wish to show via our coupling that that $(f_{\bl_1,k},\ldots,f_{\bl_n,k})$ converges in distribution to $P_W^n$ as well. We will use the following basic lemma.

\begin{lemma}\label{lem:basic_weak_convergence_fact}
Let $S$ be a metric space with Borel $\sigma$-algebra $\Sigma$, and $P$ a probability measure on $(S,\Sigma)$. Let $X_n,Y_n$ be random variables defined on the same probability space and taking values in $S$, such that $|X_n - Y_n| \to 0$ in probability (where $| \cdot |$ denotes the norm induced by the metric on $S$) and such that the distribution $\mu_{Y_n}$ of $Y_n$ converges weakly to $P$. Then $\mu_{X_n}$ converges weakly to $P$ as well.
\end{lemma}
\begin{proof}
We must show that for all $f \in \mc C_b(S)$, $\E[f(X_n)] \to \int_S f dP$. By hypothesis, $\E[f(Y_n)] \to \int_S f dP$, so it suffices to show $\E[f(X_n)-f(Y_n)] \to 0$. By the convergence in probability hypothesis, $\Pr(|X_n-Y_n| > \delta) \to 0$ for any $\delta > 0$.

$f$ is bounded, so let $B$ be such that $f \leq B$. We have
\begin{align*}
    \E[|f(X_n)-f(Y_n)|] &= \E[\bbone_{|X_n-Y_n| > \delta}|f(X_n)-f(Y_n)|] + \E[\bbone_{|X_n-Y_n| \leq \delta} |f(X_n)-f(Y_n)|] \\
    & \leq 2B \Pr(|X_n-Y_n|>\delta) + \E\left[\sup_{x : |x-Y_n| \leq \delta} |f(x)-f(Y_n)|\right].
\end{align*}
Since $g_\delta(y) := \sup_{x \in \bar{B}_\delta(y)} |f(x)-f(y)|$ is a continuous, bounded function of $y$, the weak convergence hypothesis yields $\E[g_\delta(Y_n)] \to \int_S g_\delta dP$ as $n \to \infty$. Because $f$ is continuous, $\limsup_{\delta \to 0} g_\delta(y) = 0$ for all $y$, and since $g_\delta$ is uniformly bounded by $2B$ which is integrable on $(S,\Sigma,P)$, we have by reverse Fatou's lemma that
\begin{equation*}
    \limsup_{\delta \to 0} \int_S g_\delta dP \leq \int_S \limsup_{\delta \to 0} g_\delta dP = 0.
\end{equation*}

By the above, for any $\eps > 0$, we may first choose $\delta$ so that $|\int_S g_\delta dP| < \eps/3$, then for all large enough $n$ we have $|\E[g_\delta(Y_n)] - \int_S g_\delta dP < \eps/3$ and $2B \Pr(|X_n-Y_n|>\delta) < \eps/3$, yielding that $\E[|f(X_n)-f(Y_n)|] < \eps$. Hence 
\begin{equation*}
    \E[f(X_n)-f(Y_n)] \to 0
\end{equation*}
as $n \to \infty$, completing the proof.
\end{proof}
%\fix{perhaps this proof is too basic and should get moved to an appendix or even omitted entirely}

When $S = C[0,1]^n$ with metric 
\begin{equation*}
    d((f_1,\ldots,f_n),(g_1,\ldots,g_n)) = \sup_{1 \leq i \leq n} \sup_{x \in [0,1]} |f_i(x)-g_i(x)|,
\end{equation*}
and 
\begin{align}
    X_k &= (f_{\bl_1,k},\ldots,f_{\bl_n,k}) \\
    Y_k &= (f_{\barv_1,k},\ldots,f_{\barv_n,k})
\end{align}
we have from above that $Y_k$ converges in distribution to $P_W^n$. To conclude from Lemma \ref{lem:basic_weak_convergence_fact} that $X_k$ converges in distribution to $P_W^n$, it suffices to show that $d(X_k,Y_k) \to 0$ in probability. So we must show that for any $\delta, \eps > 0$, $\Pr(d(X_k,Y_k) > \delta) < \eps$ for all sufficiently large $k$. Proposition \ref{prop:coupling_error}, together with the fact that $\l_i(k)-\bl_i(k) = \E[v_i(k)] = v_i(k) - \barv_i(k)$, yields that 
\begin{equation*}
    B(\bsomega) := \sup_{1 \leq i \leq n} \sup_k |\bl_i(k,\bsomega)-\barv_i(k, \bsomega)|
\end{equation*}
is an almost-surely finite random variable on $\Omega$. Hence there exists $D$ such that 
\begin{equation}\label{eq:B_omega}
\Pr(B(\bsomega) > D) < \eps.    
\end{equation}
By the lower bound $c_i$ in Lemma \ref{lem:facts_about_differences} Part 2, $\sum_{j=1}^\infty \sigma^2(t^{i-1}\hx_j)$ diverges, hence there exists $K$ such that for all $k>K$,
\begin{equation}\label{eq:K_gives_big_variance}
    \frac{D}{\sum_{j=1}^k \sigma^2(t^{i-1}\hx_j)} < \delta
\end{equation}
for $i=1,\ldots,n$.

We therefore have that for $k>K$ and $\bsomega$ such that $B(\bsomega) \leq D$, 
\begin{equation}\label{eq:lambda_minus_v_bound_corollary}
    \sup_{1 \leq i \leq n} \sup_{0 \leq \ell \leq k} \frac{1}{\sum_{j=1}^k \sigma^2(t^{i-1}\hx_j)}|\bl_i(\ell,\bsomega)-\barv_i(\ell,\bsomega)| < \delta.
\end{equation}
Because $\sup_{x \in [0,1]} |f-g| = \sup_{x=0,1/k,\ldots,1} |f-g|$ if $f,g$ are piecewise linear on each interval $[\frac{\ell}{k},\frac{\ell+1}{k}]$, we have
\begin{equation*}
    \sup_{0 \leq \ell \leq k} \frac{1}{\sum_{j=1}^k \sigma^2(t^{i-1}\hx_j)}|\bl_i(\ell,\bsomega)-\barv_i(\ell,\bsomega)| = \sup_{x \in [0,1]} |f_{\barv_i,k}(x)-f_{\bl_i,k}(x)|.
\end{equation*}
Thus \eqref{eq:lambda_minus_v_bound_corollary} implies that for $\bsomega$ such that $B(\bsomega) \leq D$ and $k>K$,
\begin{equation*}
    d(X_k(\bsomega),Y_k(\bsomega)) = \sup_{1 \leq i \leq n}\sup_{x \in [0,1]} |f_{\barv_i,k}(x)-f_{\bl_i,k}(x)| < \delta.
\end{equation*}
Together with \eqref{eq:B_omega} this implies that
\begin{equation*}
    \Pr(d(X_k,Y_k) > \delta) < \eps
\end{equation*}
for all $k>K$. Since $\delta,\eps$ were arbitrary, this is exactly the statement that $d(X_k,Y_k) \to 0$ in probability. This completes the proof of Theorem \ref{thm:general_hl_lln_clt}.
\end{proof}

\begin{rmk}
It is worth noting that fact that $\l_i$ jumps further than $\l_j$ in expectation for $i < j$ comes from the fact that $t < 1$. Hence our technique would no longer hold if one were to take a simultaneous limit $t \to 1$ as well, because the hopping particles would not outpace the ones behind them and hence interactions between them could contribute nontrivially in the limit. Such $t \to 1$ limits of Hall-Littlewood processes have been studied by Dimitrov \cite{dimitrov2018kpz} and Corwin-Dimitrov \cite{corwin2018transversal}, though we do not see how their results would apply directly to our specific case. We note also that the connection to $p$-adic random matrices is lost in this regime. 
\end{rmk}

\subsection{Universality of Lyapunov exponents.}\label{subsec:lyapunov}

Given the law of large numbers in \Cref{thm:lln_and_func_clt_rmt_version} and the formulas in \Cref{lem:facts_about_differences}, the proof of \Cref{thm:lyapunov} is quite easy. First recall the statement. 

\lyapunov*

\begin{proof}
For existence of Lyapunov exponents, we have
\begin{equation*}
    L_i^{(n)} := \lim_{k \to \infty} \frac{\l_{n-i+1}(k)}{k} = \lim_{k \to \infty} \frac{\frac{\l_{n-i+1}(k)}{\sum_{j=1}^k \mu(t^{n-i+1},\ldots,t^{N_j^{(n)}-i})}}{\frac{k}{\sum_{j=1}^k  \mu(t^{n-i+1},\ldots,t^{N_j^{(n)}-i})}}.
\end{equation*}
The limit of the numerator exists almost surely by \Cref{thm:lln_and_func_clt_rmt_version}. For the denominator we have
\begin{equation*}
    \frac{\sum_{j=1}^k  \mu(t^{n-i+1},\ldots,t^{N_j^{(n)}-i})}{k} = \sum_{N > n} \frac{|\{1 \leq j \leq k: N_j^{(n)}=N\}|}{k} \mu(t^{n-i+1},\ldots,t^{N_j^{(n)}-i}),
\end{equation*}
hence
\begin{equation*}
    \lim_{k \to \infty} \frac{\sum_{j=1}^k  \mu(t^{n-i+1},\ldots,t^{N_j^{(n)}-i})}{k} = \sum_{N > n} \rho_n(N)  \mu(t^{n-i+1},\ldots,t^{N_j^{(n)}-i}),
\end{equation*}
(this uses the fact that $\mu(t^{n-i+1},\ldots,t^{N-i})$ is bounded as a function of $N$). Therefore $L_i^{(n)}$ exists almost surely.

Recall that for $Z_x \sim G_x$,
\begin{equation*}
    \E Z_x = \frac{x(1-t)}{(1-tx)(1-x)}.
\end{equation*}
It is an elementary check that there exist constants $B(i)$ depending only on $t$ and $i$ such that for any $x < 1$ and $n \geq i$,
\begin{equation}\label{eq:bound_to_L_tilde}
    \abs*{\E Z_{t^{n-i}x} - (1-t)xt^{n-i}} < B(i) t^{2n}x.
\end{equation}
Let
\begin{equation*}
    \tL_i^{(n)} := \lim_{k \to \infty} \frac{\sum_{j=1}^k \sum_{\ell=1}^{N_j^{(n)}-n} (1-t)t^{n-i}\cdot t^\ell}{k}.
\end{equation*}
Then
\begin{equation}\label{eq:ltilde_limit}
    \tL_i^{(n)} = \frac{\sum_{j=1}^k t^{n-i+1}(1-t^{N_j^{(n)}-n})}{k} = t^{n-i+1}(1-c(n)).
\end{equation}
But also by \eqref{eq:bound_to_L_tilde},
\begin{equation*}
    \abs*{\sum_{j=1}^k \mu(t^{n-i+1},\ldots,t^{N_j^{(n)}-i}) - \sum_{j=1}^k \sum_{\ell=1}^{N_j^{(n)}-n} (1-t)t^{n-i}\cdot t^\ell} < \sum_{j=1}^k \sum_{\ell=1}^{N_j^{(n)}-n} B(i)t^{2n}t^\ell < k B(i) t^{2n} \frac{t}{1-t},
\end{equation*}
hence 
\begin{equation*}
    \abs*{L_i^{(n)} - \tL_i^{(n)}} < B(i)t^{2n} \frac{t}{1-t}.
\end{equation*}
Since $c(n) \leq t < 1$, $\frac{B(i)t^{2n} \frac{t}{1-t}}{t^n(1-c(n))} = o(1)$ as $n \to \infty$, so \eqref{eq:ltilde_limit} implies
\begin{equation*}
    \lim_{n \to \infty} \frac{L_i^{(n)}}{t^n(1-c(n))} = \lim_{n \to \infty} \frac{\tL_i^{(n)}}{t^n(1-c(n))} + \frac{L_i^{(n)} - \tL_i^{(n)}}{t^n(1-c(n))} = t^{1-i}.
\end{equation*}
\end{proof}

\begin{rmk}
It is worth noting that if one instead considers
\begin{equation*}
    \lim_{k \to \infty} \frac{1}{k}\log(\l_i(k))
\end{equation*}
(which in our analogy corresponds to the $i\tth$ smallest singular value), then the $n \to \infty$ limits are not universal and indeed the limits may not exist for some choices of the $N_j^{(n)}$. If the $N_j^{(n)}$ are all the same for any fixed $n$, then one has 
\begin{equation*}
    \lim_{k \to \infty} \frac{1}{k}\log(\l_i(k)) = \E[Y_i(k)]
\end{equation*}
and this clearly depends on the choice of $N_j^{(n)}$. 
\end{rmk}

\appendix

\section{Relations to the Archimedean case and alternate proof of {\Cref{prop:measure_convergence_jacobi_and_cauchy}}}\label{sec:appendixA}

As mentioned in the Introduction, \Cref{thm:exact_hl_results_in_p-adic_rmt} is exactly analogous to results on singular values of corners and products in the real, complex and quaternion cases. We will first informally state these results in more detail than in the Introduction in order to highlight the parallel, and give references to more complete treatments. We will then give an alternate proof of \Cref{prop:measure_convergence_jacobi_and_cauchy} which is simpler, but valid only under additional assumptions which do \emph{not} cover the Hall-Littlewood case $q=0$. This was the first proof we found, but we were unable to justify the $q \to 0$ limit and hence resorted to the stronger results proven in \Cref{sec:3}. However, the proof below has the advantage that it survives the limit to the real/complex/quaternion cases which we are about to describe, and hence could be used to adapt the convolution-of-projectors method of \Cref{thm:exact_hl_results_in_p-adic_rmt} to prove the analogous result in this setting. At the end of the appendix we will outline how this could be carried out.

Fix a parameter $\beta > 0$ and let $q=e^{-\eps}$, $t = q^{\beta/2}$. In all cases below, we assume that the integers $n,m,N,k$ satisfy the same constraints as in \Cref{thm:exact_hl_results_in_p-adic_rmt}. Below we give an informal statement of the analogue of \Cref{thm:exact_hl_results_in_p-adic_rmt} in the real, complex and quaternion setting.
\begin{enumerate}
    \item Define the random signature $\l(\eps)$ by
    \begin{equation*}
        \Pr(\l(\eps)=\l) = \frac{P_\l(1,t,\ldots,t^{n-1};q,t)Q_\l(t^{m-n+1},\ldots,t^{N-n};q,t)}{\Pi_{(q,t)}(1,t,\ldots,t^{n-1};t^{m-n+1},\ldots,t^{N-n})}
    \end{equation*}
    for any $\l \in \Sig_n^+$, with $q,t$ depending on $\eps$ as above. Then as $\eps \to 0$, the random real signature $\eps \l(\eps) = (\eps \l_1(\eps),\ldots,\eps \l_n(\eps))$  converges in distribution to some limiting random real signature $\l(0)$. When $\beta = 1,2,4$, $\l(0)$ has the same distribution as $(-\log(r_n),\ldots,-\log(r_1))$, where $r_1 \geq \cdots \geq r_n$ are the squared singular values of an $n \times m$ corner of a Haar-distributed element of $\O(n), \U(n)$ or $\Sp^*(n)$ respectively. This is due to Forrester-Rains \cite{forrester2005interpretations}, see also Borodin-Gorin \cite[Thm. 2.8]{borodin2015general}. 
    \item Fix a real signature $\ell$ of length $n$ and define the nonrandom signature 
    \begin{equation*}
        \l(\eps) := (\lfloor \ell_1/\eps \rfloor, \ldots, \lfloor \ell_n/\eps \rfloor) \in \Sig_n.
    \end{equation*}
    Define the random signature $\nu(\eps)$ by
    \begin{equation*}
        \Pr(\nu(\eps)=\nu) = \frac{Q_{\nu/\l(\eps)}(1,\ldots,t^{-(k-1)};q,t)P_\nu(t^{N-n},\ldots,t^{N-1};q,t)}{P_\l(t^{N-n},\ldots,t^{N-1};q,t) \Pi_{(q,t)}(1,\ldots,t^{-(k-1)};t^{N-n},\ldots,t^{N-1})}
    \end{equation*}
    for any $\nu \in \Sig_n$. Then as $\eps \to 0$, $\nu(\eps)$ converges to a random real signature $\nu(0)$. Suppose $\beta = 1,2,4$ and $\F = \R,\C,\H$ respectively, and $A_{col} \in M_{n \times (N-k)}(\F)$ is the first $N-k$ columns of $A \in M_{n \times N}(\F)$ with fixed singular values $e^{-\ell}:=(e^{-\ell_1},\ldots,e^{-\ell_n})$ and distribution invariant under the orthogonal, unitary or symplectic groups acting on the right and left. Then the distribution of the negative logarithms of the squared singular values of $A_{col}$ is given by $\nu(0)$. The statement for \eqref{eq:p-adic_corners} is exactly analogous. We could not locate these exact statements in the literature but essentially equivalent ones appear in Borodin-Gorin \cite{borodin2015general} and Sun \cite{sun2016matrix} when considering the Jacobi corners process. 
    \item Fix real signatures $r,\ell$ of length $n$ and define nonrandom integer signatures $\l(\eps)$ as above and $\rho(\eps)$ similarly with $r$ in place of $\ell$. Then as $\eps \to 0$, $\eps \cdot (\rho(\eps) \boxtimes_{(1,\ldots,t^{n-1})} \l(\eps))$ (where we abuse notation and use $\boxtimes$ to refer to the convolution operation with Macdonald polynomials instead of Hall-Littlewood) converges to a random real signature $s$. When $\beta = 1,2,4$, $e^{-s}$ gives the distribution of singular values of $AB$ where $A,B$ are bi-invariant under the orthogonal, unitary or symplectic group and have fixed singular values $e^{-r}$ and $e^{-\ell}$. See Gorin-Marcus \cite[Prop. 2.2]{gorin2020crystallization}.
\end{enumerate}

More general background on these limits may be found in Ahn \cite{ahn2019fluctuations}, Borodin-Gorin \cite{borodin2015general}, Gorin-Marcus \cite{gorin2020crystallization}, and Sun \cite{sun2016matrix}.

\begin{rmk}
The explicit formulas for the above distributions are uniform expressions in terms of $\beta$, and the distributions for general $\beta \in [0,\infty)$ are referred to as \emph{$\beta$-ensembles}. $\beta$ is then seen as an inverse temperature parameter, and the zero-temperature limit $\beta \to \infty$ has in particular been studied, both because it provides tractable though accurate approximations to $\beta=1,2,4$, and because it exhibits asymptotic behaviors interesting in their own right. In particular, the product convolution and corners operation--the analogues of \Cref{thm:exact_hl_results_in_p-adic_rmt} Parts 3 and 2 respectively--become deterministic in this limit and are controlled by certain orthogonal polynomials. See Gorin-Marcus \cite{gorin2020crystallization} and Gorin-Kleptsyn \cite{gorin2020universal} for a discussion of the eigenvalue (as opposed to singular value) case, and Borodin-Gorin \cite[Cor. 5.4]{borodin2015general} for the deterministic $\beta \to \infty$ limit of Jacobi corners; we are not aware of anywhere the $\beta \to \infty$ limits of general corners and products (the analogues of Parts 2, 3 of \Cref{thm:exact_hl_results_in_p-adic_rmt}) are worked out explicitly in the literature. In our setting, viewing the measures and operations of \Cref{thm:exact_hl_results_in_p-adic_rmt} for arbitrary $t \in (0,1)$ not necessarily a prime power is exactly analogous to this extrapolation to general $\beta$.

We observe the exact same freezing to a deterministic operation in the $p$-adic case of products and corners in the limit $p \to \infty$, i.e. $t \to 0$. It is interesting to note that while the $\beta \to \infty$ limit requires extrapolation away from the usual matrix models, the $t \to 0$ limit does not because one can find arbitrarily large primes. In the corners case, the partition $\nu$ in the notation of \Cref{thm:exact_hl_results_in_p-adic_rmt} concentrates around $\l$, and the partition $\mu$ concentrates around $(\l_{d+1},\ldots,\l_n)$. In the product case, $\nu$ concentrates around $(\l_1+\mu_1,\ldots,\l_n+\mu_n)$. These facts may be easily verified using the explicit formulas for Hall-Littlewood polynomials in \Cref{subsec:hl}, and may also be seen heuristically directly from the matrix models without any formulas.
\end{rmk}

Below we prove \Cref{prop:measure_convergence_jacobi_and_cauchy} under the additional assumptions that $\ba = (1,t,\ldots,t^{N-1})$ and $q,t \in (0,1)$. We remark that \Cref{prop:measure_convergence_branching_corners} may be proven by similar label-variable duality manipulations under the restricted hypotheses as above; the modifications to the proof below are not difficult.

\begin{proof}

For the remainder of the proof, we will denote $\boxtimes_{(1,\ldots,t^{N-1})}$ by $\boxtimes_t$ and use $\Supp$ for the support of a measure. Let $\l(D) = (D[N-n],\l)$ and $\mu(D) = (D[N-m],0[m])$. Recall that 
\begin{equation*}
    \mcj(\nu) := \sum_{\substack{\kappa \in \Sig_N \\ \kappa_{N-n+i}=\nu_i \text{ for all }i=1,\ldots,n}} c_{\l(D),\mu(D)}^\kappa(q,t) \frac{P_\kappa(t^{N-1},\ldots,1)}{P_{\l(D)}(t^{N-1},\ldots,1)P_{\mu(D)}(t^{N-1},\ldots,1)}
\end{equation*}
and we wish to show 
\begin{equation}\label{limit_to_t_mac}
    \mcj(\nu) \to \frac{P_\nu(t^{n-1},\ldots,1)Q_{\nu/\l}(t^{N-n},\ldots,t^{m-n+1})}{P_\l(t^{n-1},\ldots,1) \Pi(t^{n-1},\ldots,1;t^{N-n},\ldots,t^{m-n+1})}
\end{equation}
(note that we have written the measure in a different form from \Cref{prop:measure_convergence_jacobi_and_cauchy} by using homogeneity to rearrange powers of $t$).

Denote the limiting measure of \eqref{limit_to_t_mac} by $\mc{M}$. 
The proof is by a kind of moments method which consists of showing the convergence of expectations of observables
\begin{equation}\label{eq:observable_convergence}
    \E_{\nu \sim \mcj}[P_\alpha(q^{\nu_1}t^{n-1},\ldots,q^{\nu_n})] \to \E_{\nu \sim \mc{M}}[P_\alpha(q^{\nu_1}t^{n-1},\ldots,q^{\nu_n})]
\end{equation}
as $D \to \infty$ for each $\alpha \in \Sig_n^+$, followed by an argument that these `moments' are sufficient to give convergence of measures. We rely on the nontrivial label-variable duality satisfied by these observables, see \cite[Section 6]{mac}:
\begin{equation}\label{eq:label-variable}
    \frac{P_\nu(q^{\alpha_1}t^{n-1},\ldots,q^{\alpha_n})}{P_\nu(t^{n-1},\ldots,1)} = \frac{P_\alpha(q^{\nu_1}t^{n-1},\ldots,q^{\nu_n})}{P_\alpha(t^{n-1},\ldots,1)}.
\end{equation}
Such a strategy is used to prove similar statements in \cite[Section 4]{gorin2020crystallization}.

We first show 
\begin{equation}\label{eq:longer_sig}
    \abs*{\E_{\nu \sim \mcj}[P_\alpha(q^{\nu_1}t^{n-1},\ldots,q^{\nu_n})] - \E_{\kappa \sim \l(D) \boxtimes_t \mu(D)}[P_{(\alpha,0[N-n])}(q^{\kappa_1}t^{N-1},\ldots,q^{\kappa_N})]} \to 0
\end{equation}
as $D \to \infty$. To show \eqref{eq:longer_sig} it suffices to show that there exist constants $C(\alpha,D)$ independent of $\kappa \in \Supp(\l(D) \boxtimes_t \mu(D))$ such that $C(\alpha,D) \to 0$ as $D \to \infty$ and
\begin{equation}\label{eq:change_observable}
    |P_{(\alpha,0[N-n])}(q^{\kappa_1}t^{N-1},\ldots,q^{\kappa_N}) -P_\alpha(q^{\nu_1}t^{n-1},\ldots,q^{\nu_n})| < C(\alpha,D).
\end{equation}
where $\nu$ is defined by $\nu_i = \kappa_{N-n+i}$.
This suffices because the support $\Supp(\l(D) \boxtimes_t \mu(D))$ of this measure contains only $\kappa$ for which $\kappa \supset \l(D)$ by basic properties of the structure coefficients, hence
\begin{equation*}
    \abs*{\E_{\nu \sim \mcj}[P_\alpha(q^{\nu_1}t^{n-1},\ldots,q^{\nu_n})] - \E_{\kappa \sim \l(D) \boxtimes_t \mu(D)}[P_{(\alpha,0[N-n])}(q^{\kappa_1}t^{N-1},\ldots,q^{\kappa_N})]} <  C(\alpha,D)
\end{equation*}
by \eqref{eq:change_observable} and linearity of expectation. So let us prove \eqref{eq:change_observable}.

$P_{(\alpha,0[N-n])}$ is a polynomial in $N$ variables $q^{\kappa_1}t^{N-1},\ldots,q^{\kappa_N}$, which we split into two collections of variables, the first $N-n$ and the last $n$. As $D \to \infty$, the first $N-n$ variables go to $0$ because $\kappa_i \geq \l(D)_i = D$ for $i=1,\ldots,N-n$, for any $\kappa \in \Supp(\l(D) \boxtimes_t \mu(D))$, hence 
\begin{multline*}
    P_{(\alpha,0[N-n])}(q^{\kappa(D)_1}t^{N-1},\ldots,q^{\kappa(D)_N}) \to P_{(\alpha,0[N-n])}(0[N-n],t^{n-1}q^{\kappa(D)_{N-n+1}},\ldots,q^{\kappa(D)_N}) \\
    = P_\alpha(q^{\nu_1}t^{n-1},\ldots,q^{\nu_n})
\end{multline*}
for any sequence $\kappa(D) \in \Supp(\l(D) \boxtimes_t \mu(D))$ with last $n$ parts given by $\nu$. The last $n$ variables always lie in a compact interval $[0,q^{\l_n}]$ because $\kappa_i \geq \l_n$ for $i=N-n+1,\ldots,N$ by interlacing, for any $\kappa \in \Supp(\l(D) \boxtimes_t \mu(D))$. Hence the above convergence is uniform over $\nu$ and $\kappa$, i.e. \eqref{eq:change_observable} holds.

Thus to show \eqref{eq:observable_convergence}, it suffices to show
\begin{equation}\label{eq:observable_convergence2}
    \E_{\kappa \sim \l(D) \boxtimes_t \mu(D)}[P_{(\alpha,0[N-n])}(q^{\kappa_1}t^{N-1},\ldots,q^{\kappa_N})] \to \E_{\nu \sim \mc{M}}[P_\alpha(q^{\nu_1}t^{n-1},\ldots,q^{\nu_n})] .
\end{equation}
Now, using label-variable duality \eqref{eq:label-variable},
\begin{align*}
    &\E_{\kappa \sim \l(D) \boxtimes_t \mu(D)}[P_{(\alpha,0[N-n])}(q^{\kappa_1}t^{N-1},\ldots,q^{\kappa_N})] \\
    &=\E_{\kappa \sim \l(D) \boxtimes_t \mu(D)}\left[\frac{P_\kappa(q^{\alpha_1}t^{N-1},\ldots,1)}{P_\kappa(t^{N-1},\ldots,1)}P_{(\alpha,0[N-n])}(t^{N-1},\ldots,1)\right] \\
    &= \sum_{\kappa \in \Sig_N} c_{\l(D),\mu(D)}^\kappa(q,t) \frac{P_\kappa(t^{N-1},\ldots,1)P_{(\alpha,0[N-n])}(t^{N-1},\ldots,1)}{P_{\l(D)}(t^{N-1},\ldots,1)P_{\mu(D)}(t^{N-1},\ldots,1)}\frac{P_\kappa(q^{\alpha_1}t^{N-1},\ldots,1)}{P_\kappa(t^{N-1},\ldots,1)} \\
    &= \frac{P_{(\alpha,0[N-n])}(t^{N-1},\ldots,1)}{P_{\l(D)}(t^{N-1},\ldots,1)P_{\mu(D)}(t^{N-1},\ldots,1)} \sum_{\kappa \in \Sig_N} c_{\l(D),\mu(D)}^\kappa(q,t)P_\kappa(q^{\alpha_1}t^{N-1},\ldots,1) \\
    &= \frac{P_{(\alpha,0[N-n])}(t^{N-1},\ldots,1)}{P_{\l(D)}(t^{N-1},\ldots,1)P_{\mu(D)}(t^{N-1},\ldots,1)} P_{\l(D)}(q^{\alpha_1}t^{N-1},\ldots,1)P_{\mu(D)}(q^{\alpha_1}t^{N-1},\ldots,1) \\
    &= \frac{P_{(\alpha,D[N-n])}(q^Dt^{N-1},\ldots,q^D t^n, q^{\l_1}t^{n-1},\ldots,q^{\l_n})P_{(\alpha,D[N-n])}(q^Dt^{N-1},\ldots,q^D t^m, t^{m-1},\ldots,1)}{P_{(\alpha,D[N-n])}(t^{N-1},\ldots,1)}.
\end{align*}
As $D \to \infty$, the above clearly converges to 
\begin{equation*}
    \frac{P_{\alpha}(q^{\l_1}t^{n-1},\ldots,q^{\l_n}) P_{\alpha}(t^{m-1},\ldots,1)}{P_\alpha(t^{N-1},\ldots,1)},
\end{equation*}
so we must show 
\begin{equation}\label{eq:1expectation}
    \E_{\nu \sim \mc{M}}[P_\alpha(q^{\nu_1}t^{n-1},\ldots,q^{\nu_n})] =   \frac{P_{\alpha}(q^{\l_1}t^{n-1},\ldots,q^{\l_n}) P_{\alpha}(t^{m-1},\ldots,1)}{P_\alpha(t^{N-1},\ldots,1)}.
\end{equation}
Again using label-variable duality, and the Cauchy identity \Cref{lem:asym_cauchy}, we have
\begin{align*}
    & \E_{\nu \sim \mc{M}}[P_\alpha(q^{\nu_1}t^{n-1},\ldots,q^{\nu_n})] \\
    & = \E_{\nu \sim \mc{M}}\left[\frac{P_\alpha(t^{n-1},\ldots,1) P_\nu(q^{\alpha_1}t^{n-1},\ldots,q^{\alpha_n})}{P_\nu(t^{n-1},\ldots,1)}\right]\\
    &= \frac{P_\alpha(t^{n-1},\ldots,1)\sum_{\nu \in \Sig_n} P_\nu(q^{\alpha_1}t^{n-1},\ldots,q^{\alpha_n})Q_{\nu/\l}(t^{N-n},\ldots,t^{m-n+1})}{\Pi(t^{n-1},\ldots,1;t^{N-n},\ldots,t^{m-n+1})P_\l(t^{n-1},\ldots,1)} \\
    &=  \frac{\Pi(q^{\alpha_1}t^{n-1},\ldots,q^{\alpha_n}; t^{N-n},\ldots,t^{m-n+1})}{\Pi(t^{n-1},\ldots,1;t^{N-n},\ldots,t^{m-n+1})} \frac{P_\alpha(t^{n-1},\ldots,1)P_\l(q^{\alpha_1}t^{n-1},\ldots,q^{\alpha_n})}{P_\l(t^{n-1},\ldots,1)} \\
    &= \frac{\Pi(q^{\alpha_1}t^{n-1},\ldots,q^{\alpha_n}; t^{N-n},\ldots,t^{m-n+1})}{\Pi(t^{n-1},\ldots,1;t^{N-n},\ldots,t^{m-n+1})} P_\alpha(q^{\l_1}t^{n-1},\ldots,q^{\l_n}).
\end{align*}
Hence \eqref{eq:1expectation} is equivalent to
\begin{equation}\label{eq:cauchy_and_P}
     \frac{ \Pi(q^{\alpha_1}t^{n-1},\ldots,q^{\alpha_n}; t^{N-n},\ldots,t^{m-n+1})}{\Pi(t^{n-1},\ldots,1;t^{N-n},\ldots,t^{m-n+1})} = \frac{P_{\alpha}(t^{m-1},\ldots,1)}{P_\alpha(t^{N-1},\ldots,1)}.
\end{equation}
\eqref{eq:cauchy_and_P} follows by applying the explicit formula for principally specialized Macdonald polynomials, \cite[(6.11')]{mac}, to the numerator and denominator of the RHS, expanding the LHS into infinite products and noting that all but finitely many terms cancel, and comparing the resulting expressions. 

We have proven convergence of `moments', so let us upgrade this to convergence of measures. Consider the compact set 
\begin{equation*}
    \U^n := \{(u_1,\ldots,u_n) \in \R^n: 0 \leq u_1 \leq \cdots \leq u_n \leq q^{\l_n}\}.
\end{equation*}
Then we have a map $\phi: \Sig_n \to \U^n$ given by $\phi(\nu_1,\ldots,\nu_n) = (q^{\nu_1},\ldots,q^{\nu_n})$. Also, \begin{equation*}
    f_\alpha(u_1,\ldots,u_n) := \frac{P_\alpha(u_1t^{n-1},\ldots,u_n)}{P_\alpha(t^{n-1},\ldots,1)}
\end{equation*}
defines a function on $\U^n$. The subalgebra of $\mc C(\U^n)$ generated by the functions $f_\alpha$ is just the set of finite linear combinations of $f_\alpha$ because products of Macdonald polynomials may be expanded as linear combinations of Macdonald polynomials. This algebra contains the constant functions ($f_{(0[n])}$ is constant) and separates points, so by the Stone-Weierstrass theorem it is dense in $\mc C(\U^n)$ with sup norm. 

By hypothesis, the structure coefficients are nonnegative and hence $\mcj$ is indeed a probability measure for each $D$. To show weak convergence $\mcj \to \mc{M}$, we must show for any $f \in \mc C(\U^n)$ that $\int_{\U^n} f d\phi_*(\mcj) \to \int_{\U^n} f d\phi_*(\mc{M})$. By the above, there exists a linear combination $g$ of $f_\alpha$s such that $\sup_{u \in U^n}|f(u)-g(u)| < \eps/3$. Since $\mc{M}$ and $\mcj$ are probability measures it follows that $ \int_{\U^n} |f-g| d\phi_*(\mc{M})  < \eps/3$ and similarly with $\mc{M}$ replaced by any $\mcj$. By \eqref{eq:observable_convergence}, we may choose $D$ such that 
\begin{equation*}
     \abs*{\int_{\U^n} g d\phi_*(\mcj) - \int_{\U^n} g d\phi_*(\mc{M})} < \eps/3.
\end{equation*}
   
Putting together the three inequalities yields 
\begin{equation*}
    \abs*{\int_{\U^n} f d\phi_*(\mcj) -\int_{\U^n} f d\phi_*(\mc{M})} < \eps,
\end{equation*}
hence $\phi_*(\mcj)$ converges weakly to $\phi_*(\mc{M})$. Because both measures are supported on a discrete subset $\phi(\{\nu \in \Sig_n: \nu_n \geq \l_n\})$ of $\U^n$, this implies $\mcj(\nu) = \phi_*(\mcj)(\phi(\nu)) \to \phi_*(\mc{M})(\phi(\nu)) = \mc{M}(\nu)$ for each $\nu \in \Sig_n$, completing the proof.
\end{proof}

The proofs of \Cref{prop:measure_convergence_jacobi_and_cauchy} and \Cref{prop:measure_convergence_branching_corners} in \Cref{sec:3.2} heavily used the discrete structure of the set of integer signatures, and we have no idea how they would be modified to the continuum limit to real signatures described earlier. However, we claim that the above proof could be modified with no substantial changes. Let us briefly outline why this is so.

\begin{defi}\label{def:HO}
Let $r = (r_1,\ldots,r_n) \in \Sig_n^\R$ have distinct parts, $\theta > 0$ a parameter, and $y_1,\ldots,y_n$ complex variables. Setting $\l(\eps) = \lfloor \eps^{-1}(r_1,\ldots,r_n) \rfloor$, we define the (type A) Heckman-Opdam hypergeometric function
\begin{equation*}
    \mc F_r(y_1,\ldots,y_n;\theta) := \lim_{\eps \to 0} \eps^{\theta \binom{n}{2}} P_\l(e^{\eps y_1},\ldots,e^{\eps y_n}; q=e^{-\eps}, t = e^{-\theta \eps})
\end{equation*}
\end{defi}

The dual Heckman-Opdam function may be obtained similarly by degenerating $Q$. Instead of defining the measures appearing in the singular value setting as limits of Macdonald measures, as we did earlier in this Appendix, one may instead first take the limit to Heckman-Opdam functions and then define measures in terms of these. When one takes the limit of \eqref{eq:label-variable} in the above regime, one obtains 
\begin{equation*}
    \frac{\mc F_r(-\l_1-(n-1)\theta, -\l_2-(n-2)\theta,\ldots,-\l_n;\theta)}{\mc F_r(-(n-1)\theta,\ldots,0;\theta)} = \frac{J_\l(e^{-r_1},\ldots,e^{-r_n};\theta)}{J_\l(1,\ldots,1;\theta)}
\end{equation*}
where $J_\l$ is the classical Jack polynomial. The same argument used to prove \Cref{prop:measure_convergence_jacobi_and_cauchy} above may be used after this limit, with the Macdonald polynomials replaced by Heckman-Opdam functions or Jack polynomials as appropriate given the above, and the sums replaced by integrals. This post-limit version of \Cref{prop:measure_convergence_jacobi_and_cauchy} may then be used to implement the convolution-of-projectors strategy we used in \Cref{sec:3} to prove the analogue of \Cref{thm:exact_hl_results_in_p-adic_rmt} in the real/complex/quaternion setting. We refer to \cite{gorin2020crystallization} for similar random matrix arguments utilizing label-variable duality and Jack/Heckman-Opdam functions.

%\begin{thebibliography}{0}

% \bibliographystyle{plain}
% \setlength{\itemsep}{2ex}\normalsize
% \bibliography{references.bib}
%\end{thebibliography}

%note: the way to get the bbl file on overleaf is to first uncomment the three lines above that make the usual bibtex bibliography, then go to 'other logs and files' at the bottom right of the 'view warnings' pane, download 'output.bbl', then reupload it.

\end{document}